\documentclass{amsart}

\usepackage{amssymb}
\usepackage{graphicx}
\usepackage{enumerate}
\usepackage{amscd}
\usepackage{amsmath}
\usepackage{braket}

\title{four-manifolds admitting hyperelliptic broken Lefschetz fibrations}
\author{Kenta Hayano}
\address{Department of Mathematics, Graduate School of Science, 
Osaka University, Toyonaka, Osaka 560-0043, Japan}
\email{k-hayano@cr.math.sci.osaka-u.ac.jp}
\author{Masatoshi Sato}
\address{Department of Mathematics, Graduate School of Science,
Osaka University, Toyonaka, Osaka 560-0043, Japan}
\email{m-sato@cr.math.sci.osaka-u.ac.jp}

\theoremstyle{plain}
\newtheorem{thm}{Theorem}[section]
\newtheorem{cor}[thm]{Corollary}
\newtheorem{lem}[thm]{Lemma}
\newtheorem{prop}[thm]{Proposition}

\theoremstyle{definition}
\newtheorem{defn}[thm]{Definition}

\newtheorem{rem}[thm]{Remark}

\makeatletter

\@addtoreset{equation}{section}
\makeatother

\def\Im{\operatorname{Im}}
\def\Ker{\operatorname{Ker}}

\def\Int{\operatorname{Int}}
\def\Diff{\operatorname{Diff}}
\def\res{\operatorname{res}}
\def\Re{\operatorname{Re}}
\def\Im{\operatorname{Im}}

\begin{document}

\maketitle

\begin{abstract}

We introduce hyperelliptic simplified (more generally, directed) broken Lefschetz fibrations, 
which is a generalization of hyperelliptic Lefschetz fibrations. 
We construct involutions on the total spaces of such fibrations of genus $g\geq 3$ 
and extend these involutions to the four-manifolds obtained by blowing up the total spaces. 
The extended involutions induce double branched coverings over blown up sphere bundles over the sphere. 
We also show that the regular fiber of such a fibration of genus $g\geq 3$ 
represents a non-trivial rational homology class of the total space. 

\end{abstract}

\section{Introduction}

A broken Lefschetz fibration is a smooth map from a four-manifold to a surface which has 
at most two types of singularities, called Lefschetz singularity and indefinite fold singularity. 
This fibration was introduced in \cite{ADK} as a fibration structure compatible with near-symplectic structures. 

A simplified broken Lefschetz fibration is a broken Lefschetz fibration over the sphere 
which satisfies several conditions on fibers and singularities. 
This fibration was first defined by Baykur \cite{Ba}. 
In spite of the strict conditions in the definition of this fibration, 
it is known that every closed oriented four-manifold admits a simplified broken Lefschetz fibration 
(this fact follows from the result of Williams \cite{Wil} together with a certain move of singularities defined by Lekili \cite{Lek}). 
For a simplified broken Lefschetz fibration, we can define a monodromy representation of this fibration 
as we define for a Lefschetz fibration. 
So we can define hyperelliptic simplified broken Lefschetz fibrations as a generalization of hyperelliptic Lefschetz fibrations. 
Hyperelliptic Lefschetz fibrations have been studied in many fields, algebraic geometry and topology for example, 
and it has been shown that the total spaces of such fibrations satisfy strong conditions on the signature, the Euler characteristic and so on (see e.g. \cite{Gurtas}). 
So it is natural to ask how far total spaces of hyperelliptic simplified broken Lefschetz fibrations are restricted 
or what conditions these spaces satisfy. 
The following result gives a partial answer of these questions: 

\begin{thm}\label{main1}

Let $f:M\rightarrow S^2$ be a genus-$g$ hyperelliptic simplified broken Lefschetz fibration. 
We assume that $g$ is greater than or equal to $3$.

\begin{enumerate}[(i)]

\item Let $s$ be the number of Lefschetz singularities of $f$ whose vanishing cycles are separating. 
Then there exists an involution 
\[
\omega: M\rightarrow M
\]
such that the fixed point set of $\omega$ is the union of  (possibly nonorientable) surfaces and $s$ isolated points. 
Moreover, $\omega$ can be extended to an involution 
\[
\overline{\omega}:M\sharp s\overline{\mathbb{CP}^2}\rightarrow M\sharp s\overline{\mathbb{CP}^2}
\]
such that $M\sharp s\overline{\mathbb{CP}^2}/\overline{\omega}$ is diffeomorphic to $S\sharp 2s\overline{\mathbb{CP}^2}$, 
where $S$ is $S^2$-bundle over $S^2$, 
and the quotient map 
\[
/\overline{\omega}:M\sharp s\overline{\mathbb{CP}^2}\rightarrow M\sharp s\overline{\mathbb{CP}^2}/\overline{\omega}\cong S\sharp 2s\overline{\mathbb{CP}^2}
\]
is the double branched covering. 

\item Let $F\in M$ be a regular fiber of $f$. 
Then $F$ represents a non-trivial rational homology class of $M$, that is, $[F]\neq 0$ in $H_2(M;\mathbb{Q})$. 

\end{enumerate}

\end{thm}

\begin{rem}

Theorem \ref{main1} can be generalized to directed broken Lefschetz fibrations, which are broken Lefschetz fibrations over the sphere satisfying certain conditions on singularities (cf. Theorem \ref{main1-directed}). 

\end{rem}

The statement (i) in Theorem \ref{main1} is a generalization of the results of Fuller \cite{Fuller} and Siebert-Tian \cite{ST} on hyperelliptic Lefschetz fibrations. 
Indeed, they proved independently that, after blowing up $s$ times, the total space of a hyperelliptic Lefschetz fibration (with arbitrary genus) 
is a double branched covering of a manifold obtained by blowing up a sphere bundle over the sphere $2s$ times, 
where $s$ is the number of Lefschetz singularities of the fibration whose vanishing cycles are separating.  
Fuller used handle decompositions and Kirby diagrams to prove the above statement, 
while Siebert and Tian used complex geometrical techniques. 
We also use handle decompositions to prove the statement {\it (i)} in Theorem \ref{main1}, 
but our method is slight different from the one Fuller used; 
ours can give an involution of the total space of a fibration explicitly 
and this explicit description is used in the proof of the statement (ii) in Theorem \ref{main1}. 

\par

Auroux, Donaldson and Katzarkov \cite{ADK} showed that a closed oriented four-manifold $M$ admits a near-symplectic form if and only if 
$M$ admits a broken Lefschetz pencil (or fibration) $f$ which has a cohomology class $h\in H^2(M)$ 
such that $h(\Sigma)>0$ for every connected component of every fiber of $f$. 
Moreover, for a broken Lefschetz fibration $f$ satisfying the cohomological condition above, 
we can take a near-symplectic form $\theta$ so that all the fibers of $f$ are symplectic outside of the singularities. 
Since every fiber of a simplified broken Lefschetz fibration is connected, 
we obtain the following corollary. 

\begin{cor}\label{cor:near-symplectic}

Let $f:M\rightarrow S^2$ be a hyperelliptic simplified broken Lefschetz fibration of genus $g\geq 3$. 
Then there exists a near-symplectic form $\theta$ on $M$ which makes all the fibers of $f$ symplectic outside of the singularities. 

\end{cor}

Since the self-intersection of a regular fiber of a broken Lefschetz fibration is equal to $0$, 
we also obtain: 

\begin{cor}\label{cor:definitemfd}

A closed oriented four-manifold with definite intersection form cannot admit any hyperelliptic simpified broken Lefschetz fibrations of genus $g\geq 3$. 

\end{cor}

We emphasize that the condition $g\geq 3$ in the above statements is essential. 
Indeed, it is proved in \cite{ADK} that $S^4$ and $\sharp n \overline{\mathbb{CP}^2}$ ($n\geq 1$) admit genus-$1$ simplified broken Lefschetz fibrations. 
Since every simplified broken Lefschetz fibration with genus less than $3$ is hyperelliptic, 
these examples mean that Corollary \ref{cor:near-symplectic} and \ref{cor:definitemfd} do not hold without the assumption $g\geq 3$. 

\par

It is shown in \cite{H} that a simply connected four-manifold with a positive definite intersection form 
cannot admit any genus-$1$ simplified broken Lefschetz fibrations except $S^4$. 
In particular, $\sharp n \mathbb{CP}^2$ ($n\geq 1$) cannot admit any genus-$1$ simplified broken Lefschetz fibrations. 
However, we prove the following theorem. 

\begin{thm}\label{main2}

For each $n\geq 0$, $\sharp n\mathbb{CP}^2$ admits a genus-$2$ simplified broken Lefschetz fibration. 

\end{thm}

The above theorem also means that Corollary \ref{cor:definitemfd} does not hold without the assumption on genus. 
Moreover, it is easy to see that the fibration in Theorem \ref{main2} cannot be compatible with any near-symplectic forms 
although $\sharp n\mathbb{CP}^2$ ($n\geq 1$) admits a near-symplectic form.  

\par

In general, a genus-$g$ simplified broken Lefschetz fibration can be changed into a genus-$(g+1)$ simplfied broken Lefschetz fibration 
by a certain homotopy of fibrations, called {\it flip and slip} (for the detail of this homotopy, see e.g. \cite{Ba2}). 
Therefore, for any $g\geq 3$, we can easily construct genus-$g$ simplified broken Lefschetz fibrations on $S^4$, $\sharp n\mathbb{CP}^2$ and $\sharp n\overline{\mathbb{CP}^2}$ ($n\geq 1$). 
However, these fibrations are not hyperelliptic because of Corollary \ref{cor:definitemfd}. 

\par

In Section \ref{section:preliminaries}, we review the definitions of broken Lefschetz fibrations and simplified ones. 
We also review the basic properties of monodromy representations of broken Lefschetz fibrations. 
After reviewing the hyperelliptic mapping class group, we give the definition of hyperelliptic simplified broken Lefschetz fibrations. 
In Section \ref{section:preserve_c}, we prove a certain Lemma on the subgroup of the hyperelliptic mapping class group 
which consists of elements preserving a simple closed curve $c$. 
This lemma plays a key role in the proof of Theorem \ref{main1}. 
In Section \ref{section:involution}, we give the proof of Theorem \ref{main1}. 
In Section \ref{section:genus-2}, we construct a genus-$2$ simplified broken Lefschetz fibration
on $\sharp n \mathbb{CP}^2$ ($n\geq 1$) to prove Theorem \ref{main2}.


\section{Preliminaries}\label{section:preliminaries}

\subsection{Broken Lefschetz fibrations}

We first give the precise definition of broken Lefschetz fibrations. 

\begin{defn}\label{definitionBLF}

Let $M$ and $\Sigma$ be compact oriented smooth manifolds of dimension $4$ and $2$, respectively. 
A smooth map $f:M\rightarrow \Sigma$ is called a {\it broken Lefschetz fibration} ({\it BLF}, for short) 
if it satisfies the following conditions: 

\begin{enumerate}

\item $f^{-1}(\partial \Sigma)=\partial M$; 

\item $f$ has at most the two types of singularities which is locally written as follows: 

\begin{itemize}

\item $(z_1,z_2)\mapsto \xi = z_1z_2$, where $(z_1,z_2)$ (resp. $\xi$) 
is a complex local coordinate of $M$ (resp. $\Sigma$) 
compatible with its orientation; 

\item $(t,x_1,x_2,x_3)\mapsto (y_1,y_2)=(t,{x_1}^2+{x_2}^2-{x_3}^2)$, 
where $(t,x_1,x_2,x_3)$ (resp. $(y_1,y_2)$) is a real coordinate of $M$ (resp. $\Sigma$). 

\end{itemize}

\end{enumerate}

\end{defn}

The first singularity in the condition (2) of Definition \ref{definitionBLF} is called a 
{\it Lefschetz singularity} and the second one is called an {\it indefinite fold singularity}. 
We denote by $\mathcal{C}_f$ the set of Lefschetz singularities of $f$ 
and by $Z_f$ the set of indefinite fold singularities of $f$. 
We remark that a Lefschetz fibration is a BLF which has no indefinite fold singularities. 

Let $f:M\rightarrow S^2$ be a BLF over the $2$-sphere. 
Suppose that the restriction of $f$ to the set of singularities is injective and that  
the image $f(Z_f)$ is the disjoint union of embedded circles parallel to the equator of $S^2$. 
We put $f(Z_f)=Z_1\amalg \cdots \amalg Z_m$, where $Z_i$ is the embedded circle in $S^2$. 
We choose a path $\alpha:[0,1]\rightarrow S^2$ satisfying the following properties: 

\begin{enumerate}

\item $\Im{\alpha}$ is contained in the complement of $f(\mathcal{C}_f)$; 

\item $\alpha$ starts at the south pole $p_s\in S^2$ and connects the south pole to the north pole $p_n\in S^2$; 

\item $\alpha$ intersects each component of $f(Z_f)$ at one point transversely. 

\end{enumerate}

\noindent
We put $\{q_i\}=Z_i\cap \Im{\alpha}$ and $\alpha(t_i)=q_i$. 
We assume that $q_1,\ldots,q_m$ appear in this order when we go along $\alpha$ from $p_s$ to $p_n$ (see Figure \ref{directedBLF}). 

\begin{figure}[htbp]
\begin{center}
\includegraphics[width=50mm]{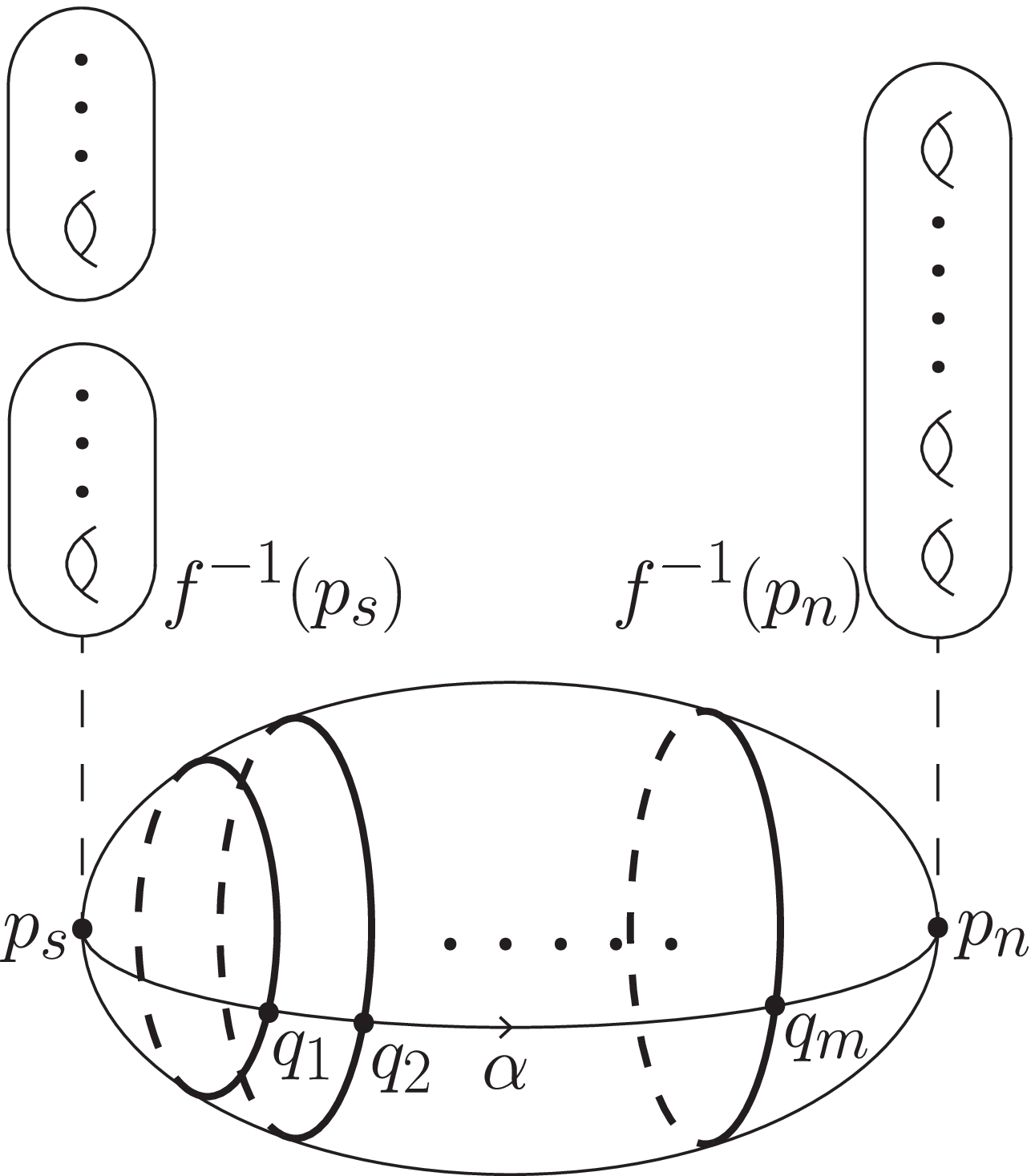}
\end{center}
\caption{The example of the path $\alpha$. 
The bold circles describe $f(Z_f)$. }
\label{directedBLF}
\end{figure}

The preimage $f^{-1}(\Im{\alpha})$ is a $3$-manifold which is a cobordism  
between $f^{-1}(p_s)$ and $f^{-1}(p_n)$. 
By the local coordinate description of the indefinite fold singularity, 
it is easy to see that $f^{-1}(\alpha([0,t_i+\varepsilon]))$ is obtained from 
$f^{-1}(\alpha([0,t_i-\varepsilon]))$ by either $1$ or $2$-handle attachment for each $i=1,\ldots,m$. 
In particular, we obtain a handle decomposition of the cobordism $f^{-1}(\Im{\alpha})$. 

\begin{defn}\label{defn:directedBLF}

A BLF $f$ is said to be {\it directed} if it satisfies the following conditions: 

\begin{enumerate}

\item the restriction of $f$ to the set of singularities is injective and   
the image $f(Z_f)$ is the disjoint union of embedded circles parallel to the equator of $S^2$; 

\item all the handles in the above handle decomposition of $f^{-1}(\Im{\alpha})$ is index-$1$;  

\item all Lefschetz singularities of $f$ are in the preimage of the component of $S^2\setminus (Z_1\amalg\cdots\amalg Z_m)$ 
which contains the point $p_n$. 

\end{enumerate}

\end{defn}

The third condition in the above definition is not essential. 
Indeed, we can change a BLF $f$ which satisfies the conditions (1) and (2) so that it satisfies the condition (3) (cf. Baykur \cite{Ba}). 

\par

For a directed BLF $f$, we assume that the set of indefinite fold singularities of $f$ is connected 
and that all the fibers of $f$ are connected. 
We call such a BLF a {\it simplified broken Lefschetz fibration}. 
For a simplified BLF $f$, $Z_f$ is empty set or an embedded circle in $M$. 
If $Z_f$ is not empty, the image $f(Z_f)$ is an embedded circle in $S^2$. 
So $S^2\setminus \Int{\nu f(Z_f)}$ consists of two $2$-disks $D_1$ and $D_2$ 
and the genus of the regular fiber of the fibration $\text{res}f:f^{-1}(D_1)\rightarrow D_1$ is just one higher than 
that of the fibration $\text{res}f:f^{-1}(D_2)\rightarrow D_2$. 
we call $f^{-1}(D_1)$ (resp. $f^{-1}(D_2)$) the {\it higher side} (resp. {\it lower side}) of $f$ 
and $f^{-1}(\nu f(Z_f))$ the {\it round cobordism} of $f$. 
By the definition, all the Lefschetz singularities of $f$ are in the higher side of $f$. 
We call the genus of the regular fiber in the higher side the {\it genus} of $f$. 

In this paper, we will use the abbreviation SBLF 
to refer to a simplified BLF.

\subsection{Monodromy representations}\label{subsec_monodromy}

Let $f:M\rightarrow B$ be a genus-$g$ Lefschetz fibration. 
We denote by $\mathcal{C}_f=\{z_1,\ldots,z_n\}$ the set of Lefschetz singularities of $f$ 
and put $y_i=f(z_i)$. 
For a base point $y_0\in B\setminus f(\mathcal{C}_f)$, 
a homomorphism $\varrho_f:\pi_1(B\setminus f(\mathcal{C}_f) , y_0)\rightarrow \mathcal{M}_g$, 
called a {\it monodromy representation} of $f$, can be defined, 
where $\mathcal{M}_g=\Diff{^+\Sigma_g}/\Diff{^+_0\Sigma_g}$ is the mapping class group of the closed oriented surface $\Sigma_g$ 
We endow the $C^\infty$ topology with $\Diff{^+\Sigma_g}$ and then $\Diff{^+_0\Sigma_g}$ is the component of $\Diff{^+\Sigma_g}$ 
containing the identity map. 
(the reader should refer to \cite{GS} for the precise definition of this homomorphism). 

We look at the case $B=D^2$. 
For each $i=1,\ldots, n$, we take an embedded path $\alpha_i\subset D^2$ satisfying the following conditions: 

\begin{itemize}

\item each $\alpha_i$ connects $y_0$ to $y_i$; 

\item $\alpha_i\cap f(\mathcal{C}_f)=\{y_i\}$; 

\item $\alpha_i\cap\alpha_j=\{y_0\}$, for all $i\neq j$; 

\item $\alpha_1,\ldots, \alpha_n$ appear in this order when we travel counterclockwise around $y_0$. 

\end{itemize} 

\noindent
For each $i=1,\ldots,n$, we denote by $a_i\in \pi_1(D^2\setminus f(\mathcal{C}_f) , y_0)$ 
the element represented by the loop obtained by connecting a counterclockwise circle 
around $y_i$ to $y_0$ by using $\alpha_i$. 
We put $W_f=(\varrho_f(a_1),\ldots, \varrho_f(a_n))\in {\mathcal{M}_g}^n$. 
This sequence is called a {\it Hurwitz system} of $f$. 
By the conditions on paths $\alpha_1,\ldots,\alpha_n$, 
the product $\varrho_f(a_1)\cdot\cdots\cdot\varrho_f(a_n)$ is equal to 
the monodromy of the boundary of $D^2$. 
It is known that each $\varrho_f(a_i)$ is the right-handed Dehn twist 
along a certain simple closed curve $c_i$, called a {\it vanishing cycle} 
of the Lefschetz singularity $z_i$ (cf. \cite{Kas} or \cite{Matsumoto3}). 

\begin{rem}

$W_f$ is not unique for $f$. 
Indeed, it depends on the choice of paths $\alpha_1,\ldots, \alpha_n$ 
and the choice of the identification of $f^{-1}(y_0)$ with the closed oriented surface $\Sigma_g$. 
However, it is known that another Hurwitz system $\tilde{W_f}$ is obtained from $W_f$ 
by successive application of the following transformations: 

\begin{itemize}

\item $(\ldots, g_i,g_{i+1},\ldots)\mapsto (\ldots, g_{i+1}, {g_{i+1}}^{-1}g_ig_{i+1}, \ldots)$ and its inverse transformation; 

\item $(g_1,\ldots,g_n)\mapsto (h^{-1}g_1h,\ldots,h^{-1}g_nh)$,  

\end{itemize}

\noindent
where $g_i,h\in \mathcal{M}_g$ (cf. \cite{GS}). 
Two sequences of elements in $\mathcal{M}_g$ is called {\it Hurwitz equivalent} 
if one is obtained from the other by successive application of the transformations above

\end{rem}

Let $\hat{f}:M\rightarrow S^2$ be a genus-$g$ SBLF 
with non-empty indefinite fold singularities. 
We denote by $M_h$ the higher side of $\hat{f}$. 
Then the restriction $\text{res}\hat{f}:M_h\rightarrow D^2$ is a Lefschetz fibration over $D^2$. 
So a monodromy representation and a Hurwitz system of $\text{res}\hat{f}$ can be defined 
and are called a {\it monodromy representation} and a {\it Hurwitz system} of $\hat{f}$, respectively. 
We denote them by $\varrho_{\hat{f}}$ and $W_{\hat{f}}$. 
For the Lefschetz fibration $\text{res}\hat{f}:M_h\rightarrow D^2$, 
we choose a base point $y_0$ and paths $\alpha_1,\ldots,\alpha_n$ 
as in the preceding paragraph. 
We also take a path $\alpha:[0,1]\rightarrow S^2$ satisfying the following conditions: 

\begin{itemize}

\item $\alpha$ starts at the base point $y_0$ 
and connects $y_0$ to a point in the image of the lower side of $\hat{f}$; 

\item for each $i=1,\ldots, n$, $\alpha\cap \alpha_i=\{y_0\}$; 

\item $\alpha$ intersects the image $\hat{f}(Z_{\hat{f}})$ at one point transversely; 

\item $\Im{\alpha} \cap \hat{f}(\mathcal{C}_{\hat{f}})=\emptyset$; 

\item $\alpha_1,\ldots,\alpha_n,\alpha$ appear in this order when we travel counterclockwise around $y_0$. 

\end{itemize}

\noindent
We put $q=\alpha(t)\in \Im{\alpha}\cap \hat{f}(Z_{\hat{f}})$. 
The preimage $\hat{f}^{-1}(\alpha([0,t+\varepsilon]))$ 
is obtained from the preimage $\hat{f}^{-1}(\alpha([0,t-\varepsilon]))\cong \hat{f}^{-1}(p_0)\times [0,t-\varepsilon]$ 
by the $2$-handle attachment. 
We regard the attaching circle $c$ of the $2$-handle as a simple closed curve in $\hat{f}^{-1}(p_0)\cong \Sigma_g$. 
We call this simple closed curve a {\it vanishing cycle} of the indefinite fold singularity of $\hat{f}$. 

\begin{lem}[Baykur \cite{Ba}]

The product $\varrho_{\hat{f}}(a_1)\cdot\cdots\cdot\varrho_{\hat{f}}(a_n)$ is contained in $\mathcal{M}_g(c)$, 
where $\mathcal{M}_g(c)$ is the subgroup of the group $\mathcal{M}_g$ which consists of elements 
represented by a map preserving the curve $c$, that is, 
\[
\mathcal{M}_g(c)=\{[T]\in \mathcal{M}_g | T(c)=c\}. 
\]

\end{lem}

For an element $\psi\in \mathcal{M}_g(c)$, 
we take a representative $T\in \psi$ so that $T$ preserves the curve $c$. 
Then $T$ induces the homeomorphism $T: \Sigma_g\setminus c\rightarrow \Sigma_g\setminus c$ 
and this homeomorphism can be extended to the homeomorphism $\hat{T}:\Sigma_{g-1}\rightarrow \Sigma_{g-1}$ 
by regarding $\Sigma_g\setminus c$ as the genus-$(g-1)$ surface with two punctures. 
Eventually, we can define the homomorphism $\Phi_c$ as follows: 
\[
\begin{array}{rccc}
\Phi_c : &\mathcal{M}_g(c) & \longrightarrow & \mathcal{M}_{g-1}  \\
& \rotatebox{90}{$\in$}\hspace{.3em} &  & \rotatebox{90}{$\in$} \\ 
& \psi=[T]\hspace{.3em} & \longmapsto     & \hspace{.3em}[\hat{T}].  
\end{array}
\]

\begin{rem}

Let $c\subset \Sigma_g$ be a separating simple closed curve. 
We can regard $\Sigma_g\setminus c$ as the disjoint union of the two once punctured surfaces of genus $h$ and $g-h$. 
So we can define the homomorphism $\Phi_c:\mathcal{M}_g(c)\rightarrow \mathcal{M}_h\times \mathcal{M}_{g-h}$ 
as we define $\Phi_c$ for a non-separating curve $c$, 
where $\mathcal{M}_g(c^\text{ori})$ is the subgroup of $\mathcal{M}_g(c)$ which consists of elements represented by maps preserving $c$ and its orientation.  

\end{rem}

\begin{lem}[Baykur \cite{Ba}]\label{conditionHur}

The product $\varrho_{\hat{f}}(a_1)\cdot\cdots\cdot\varrho_{\hat{f}}(a_n)$ is contained in the kernel of $\Phi_c$. 
Conversely, if simple closed curves $c,c_1,\ldots,c_n\subset \Sigma_g$ satisfy the following conditions: 

\begin{itemize}

\item $c$ is non-separating; 

\item $t_{c_1}\cdot\cdots\cdot t_{c_n}\in\Ker{\Phi_c}$, 

\end{itemize}

\noindent
then there exists a genus-$g$ SBLF $f:M\rightarrow S^2$ such that 
$W_f=(t_{c_1},\ldots, t_{c_n})$ and a vanishing cycle of the indefinite fold of $f$ is $c$.  

\end{lem}

\subsection{The hyperelliptic mapping class group}

Let $\Sigma_g$ be a closed oriented surface of genus $g\ge1$. 
Denote by $\iota:\Sigma_g\to \Sigma_g$ an involution described in Figure \ref{involution}. 

\begin{figure}[htbp]
\begin{center}
\includegraphics[width=120mm]{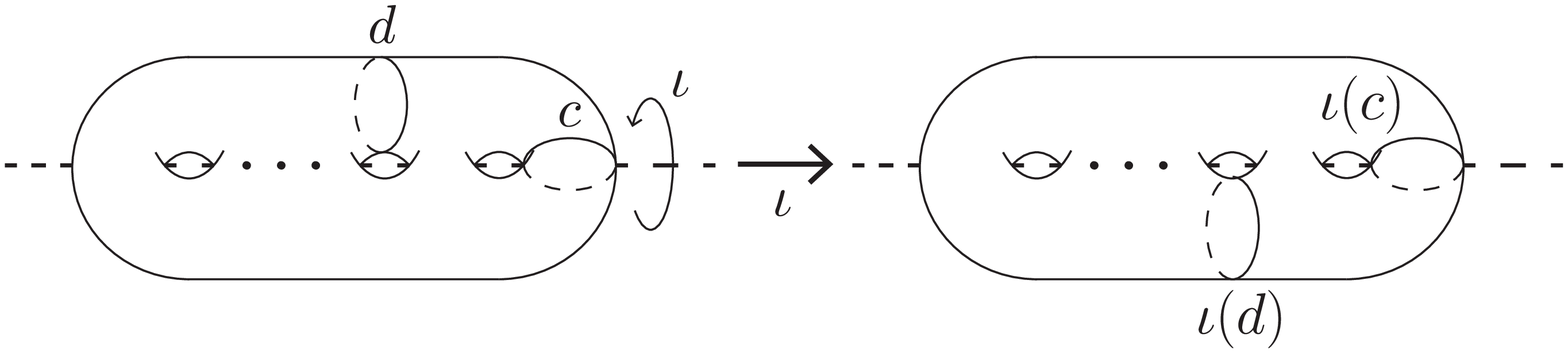}
\end{center}
\caption{the hyperelliptic involution on the surface $\Sigma_g$. }
\label{involution}
\end{figure}

Let $C(\iota)$ denote the centralizer of $\iota$ in the diffeomorphism group $\Diff_+\Sigma_g$, 
and endow $C(\iota)\subset \Diff_+\Sigma_g$ with the relative topology. 
The inclusion homomorphism $C(\iota)\to \Diff_+\Sigma_g$ induces 
a natural homomorphism $\pi_0C(\iota)\to \mathcal{M}_g$ between their path-connected components. 

\begin{thm}[Birman-Hilden \cite{Birman-Hilden}]\label{thm:hypMCG}
When $g\ge2$, the homomorphism $\pi_0C(\iota)\to \mathcal{M}_g$ is injective.
\end{thm}

Denote the image of this homomorphism by $\mathcal{H}_g$ for $g\ge1$. 
This group is called the {\it hyperelliptic mapping class group}. 
In fact, they showed the above result in more general settings later. See \cite{Birman-Hilden} for details. 

\par

A Lefschetz fibration is said to be {\it hyperelliptic} if we can take an identification of the fiber of a base point 
with the closed oriented surface so that the image of the monodromy representation of the fibration 
is contained in the hyperelliptic mapping class group. 
So it is natural to generalize this definition to directed (and especially simplified) BLFs as follows: 
Let $f:M\rightarrow S^2$ be a directed BLF. 
We use the same notations as we use in the argument before Definition \ref{defn:directedBLF}. 
We take a disk neighborhood $D\subset S^2\setminus f(Z_f)$ of $p_n$ so that $f(\mathcal{C}_f)$ is contained in $D$.  
We put $r_i=\alpha\left(\frac{t_i+t_{i+1}}{2}\right)$ ($i=1,\ldots,m-1$) and $r_m=p_n$. 
Let $d_i\subset f^{-1}(r_i)$ be the vanishing cycle of $Z_i$ determined by $\alpha$. 
Once we fix an identification of $f^{-1}(r_m)$ with $\Sigma_{g_1}\amalg\cdots\amalg\Sigma_{g_k}$, 
we obtain an involution $\iota_i$ on $f^{-1}(r_i)$ induced by the hyperelliptic involution on $f^{-1}(r_m)$ 
since we can identify $f^{-1}(r_{i-1})\setminus \{\text{two points}\}$ with $f^{-1}(r_i)\setminus d_i$ by using $\alpha$. 
$f$ is said to be {\it hyperelliptic} if it satisfies the following conditions for a suitable identification of $f^{-1}(r_m)$ 
with $\Sigma_{g_1}\amalg\cdots\amalg\Sigma_{g_k}$: 

\begin{itemize}

\item the image of the monodromy representation of the Lefschetz fibration $\res{f}:f^{-1}(D) \rightarrow D$ is contained in the group $\mathcal{H}_g$; 

\item $d_i$ is preserved by the involution $\iota_i$ up to isotopy. 

\end{itemize}

In this paper, we will call a hyperelliptic SBLF HSBLF for short.

\begin{rem}

Every SBLF whose genus is less than or equal to $2$ is hyperelliptic 
since $\mathcal{H}_g=\mathcal{M}_g$ and all simple closed curves in $\Sigma_g$ is preserved by $\iota$ if $g\leq 2$.  

\end{rem}



\subsection{Handle decompositions}\label{subsec_handle}

Let $f:M\rightarrow S^2$ be a genus-$g$ SBLF 
and $M_h$ (resp. $M_r$ and $M_l$) the higher side (resp. the round cobordism and the lower side) of $f$. 
Then $\res{f}:M_h\rightarrow D^2$ is a Lefschetz fibration over the disk. 
We choose $y_0\in D^2$ and $\alpha_1,\ldots,\alpha_n\subset D^2$ as in subsection \ref{subsec_monodromy}. 
Let $D\subset\Int{D^2}\setminus \mathcal{C}_f$ be a disk whose boundary intersects each path $\alpha_i$ at one point transversely. 
Denote by $w_i\in \partial D$ the intersection between $\partial D$ and $\alpha_i$ 
and by $c_i\subset f^{-1}(w_i)$ a vanishing cycle of the Lefschetz singularity in the fiber $f^{-1}(y_i)$.  

\begin{thm}[Kas \cite{Kas}]

$M_h$ is obtained by attaching $n$ $2$-handles to $f^{-1}(D)\cong D\times \Sigma_g$ 
whose attaching circles are $c_1,\ldots,c_n$
and framings of these handles are $-1$ relative to the framing along the fiber. 

\end{thm}

We call $R^\lambda=[0,1]\times D^\lambda\times D^{3-\lambda}/((1,x_1,x_2,x_3)\sim(0,\pm x_1,x_2,\pm x_3))$ 
a ($4$-dimensional) {\it round $\lambda$-handle} ($\lambda=1,2$) 
and $X^4\cup_\varphi R^\lambda$ a $4$-manifold obtained by attaching a round $\lambda$-handle to the $4$-manifold $X^4$, 
where $\varphi: [0,1]\times \partial D^\lambda\times D^{3-\lambda}/\sim\rightarrow \partial X$ is an embedding. 
A round handle $R^\lambda$ is said to be {\it untwisted} if the sign in the equivalence relation is plus 
and {\it twisted} otherwise. 

\begin{thm}[Baykur \cite{Ba}]

$M_h\cup M_r$ is obtained by attaching a round $2$-handle to $M_h$. 
Moreover, a circle $\{t\}\times \partial D^2\times \{0\}$ in the attaching region of $R^2$ 
is attached along a vanishing cycle of indefinite fold singularities of $f$. 

\end{thm}

We remark that the isotopy class of the attaching map $\varphi:[0,1]\times \partial D^2 \times D^1/\sim\rightarrow \partial M_h$ is uniquely determined 
by a vanishing cycle of indefinite fold of $f$ if the genus of $f$ is greater than or equal to $2$. 
In particular, the total space of $f$ is uniquely determined by a vanishing cycle of indefinite fold of $f$ and ones of Lefschetz singularities of $f$. 
if the genus of $f$ is greater than or equal to $3$. 
However, there exist infinitely many SBLFs of genus $g\leq 2$ such that they have the same vanishing cycles but the total spaces of them are mutually distinct (see \cite{BK} or \cite{H}). 

\par

Round $2$-handle attachment is given by $2$-handle attachment followed by $3$-handle attachment (cf. \cite{Ba}). 
So we obtain a handle decomposition of $M_h\cup M_r$ by the above theorems. 
Since $M_l$ contains no singularities of $f$, 
the map $\res{f}:M_l\rightarrow D^2$ is the trivial $\Sigma_{g-1}$-bundle. 
In particular, $M_l\cong D^2\times\Sigma_{g-1}$ 
and we obtain a handle decomposition of $M=M_h\cup M_r\cup M_l$. 
Moreover, we can draw a Kirby diagram of $M$ by the decomposition 
(For more details on this decomposition and corresponding diagram, see \cite{Ba}). 

\par

We also obtain a handle decomposition of the total space of a directed BLF $f:M\rightarrow S^2$ by the same argument as above. 
Indeed, we can decompose $M$ into $D^2\times (\Sigma_{g_1}\amalg\cdots\amalg\Sigma_{g_m})$, 
$n_1$ $2$-handles, $n_2$ round $2$-handles and $D^2\times (\Sigma_{h_1}\amalg\cdots\amalg\Sigma_{h_m})$, 
where $n_1$ is the number of the Lefschetz singularities of $f$ and $n_2$ is the number of the components of the set of indefinite fold singularities of $f$.


\section{A subgroup $\mathcal{H}_g(c)$ of the hyperelliptic mapping class group which preserves a curve $c$}\label{section:preserve_c}

Let $c$ be an essential simple closed curve in the surface $\Sigma_g$ 
which is preserved by the involution $\iota\in \Diff_+\Sigma_g$ as a set. 
Let $\mathcal{H}_g(c)$ denote the subgroup of the hyperelliptic mapping class group defined by $\mathcal{H}_g(c):=\mathcal{H}_g\cap \mathcal{M}_g(c)$. 
As introduced in Theorem \ref{thm:hypMCG}, 
the hyperelliptic mapping class group $\mathcal{H}_g$ is isomorphic to the group consisting of the path-connected components of $C(\iota)$. 
Hence, the group $\mathcal{H}_g(c)$ consists of the mapping classes 
which can be represented by both of elements in the centralizer $C(\iota)$ and elements in $\Diff_+(\Sigma_g,c)$. 
Let $\mathcal{H}_g^s(c)$ denote the subgroup of $\pi_0C(\iota)$ defined by $\mathcal{H}_g^s(c):=\{[T]\in \pi_0C(\iota)\,|\, T(c)=c\}$. 
In this section, we will prove the following lemma.

\begin{lem}\label{lemma:isomH_g(c)}
Let $g\ge2$. The natural isomorphism $\pi_0C(\iota)\to \mathcal{H}_g$ in Theorem \ref{thm:hypMCG} restricts to an isomorphism between the groups $\mathcal{H}_g^s(c)$ and $\mathcal{H}_g(c)$. 
\end{lem}

To prove the lemma, It is enough to show that the homomorphism maps $\mathcal{H}_g^s(c)$ onto $\mathcal{H}_g(c)$. 
Let $[T]$ be a mapping class in $\mathcal{H}_g(c)$. 
We can choose a representative $T:\Sigma_g\to \Sigma_g$ in the centralizer $C(\iota)$. 
Since it is isotopic to some diffeomorphism on $\Sigma_g$ which preserves the curve $c$, the curve $T(c)$ is isotopic to $c$.

We call an isotopy $L_0:\Sigma_g\times [0,1]\to \Sigma_g$ is symmetric if and only if $L_0(*,t)\in C(\iota)$ for any $t\in [0,1]$. 
In the following, we will construct a symmetric isotopy $L:\Sigma_g\times [0,1]\to \Sigma_g$ satisfying
\[
L(*,0)=T,\text{ and } L(c,1)=c\subset \Sigma_g.
\]
It indicates that $L(*,1)$ represents an element in $\mathcal{H}_g^s(c)$, and  $[L(*,1)]=[T]\in\pi_0C(\iota)$. Hence, we see that the homomorphism $\mathcal{H}_g^s(c)\to \mathcal{H}_g(c)$ is surjective.

To construct the symmetric isotopy $L:\Sigma_g\times [0,1]\to \Sigma_g$, we need a proposition, so called the bigon criterion.

\begin{prop}[Farb-Margalit Proposition 1.7 \cite{Farb-Margalit}]
Let $S$ be a compact surface. The geometric intersection number of two transverse simple closed curves in $S$ is minimal if and only if they do not form a bigon.
\end{prop}

We may assume that the curves $c$ and $T(c)$ are transverse by changing the diffeomorphism $T$ in terms of some symmetric isotopy. Since $c$ and $T(c)$ are isotopic, the minimal intersection number of them is $0$. Hence, there exist bigons such that each of their boundaries is the union of an arc of $c$ and an arc of $T(c)$. Choose an innermost bigon $\Delta$ among them.

Let $\alpha$ be the arc $c\cap \partial\Delta$ and $\beta$ the arc $T(c)\cap \partial\Delta$, respectively. Since $\Delta$ is a bigon, the endpoints of them coincide. Denote them by $\{x_1,x_2\}\subset\partial\Delta$. 

\begin{lem}\label{lemma:innermost}
\[
\Int\Delta\cap(T(c)\cup c)=\emptyset
\]
\end{lem}

\begin{proof}
If the set $\Int\Delta\cap c$ is non-empty, there exists an arc of $c$ in $\Delta$ which forms an bigon with the arc $\beta$. Since the bigon $\Delta$ is innermost, it is a contradition. We can also show that $\Int\Delta\cap T(c)=\emptyset$ in the same way.
\end{proof}

Note that the bigon $\iota(\Delta)$ is also innermost. By Lemma \ref{lemma:innermost}, we have $\Delta\cap\iota(\Delta)=\partial\Delta\cap\partial\iota(\Delta)$.

\begin{lem}\label{lemma:x1x2}
\[
\partial\Delta\cap\partial\iota(\Delta)\subset\{x_1,x_2\}.
\]
\end{lem}

\begin{proof}
Since $\partial \alpha=\partial \beta=\alpha\cap\beta=\{x_1,x_2\}$, 
it suffices to show that $\Int \alpha\cap \partial\iota(\Delta)=\Int \beta\cap \partial\iota(\Delta)=\emptyset$. 
Since $\alpha\cap T(c)=\{x_1,x_2\}$, we have $\Int\alpha\cap\iota(\beta)=\emptyset$. 
Next, we will show that $\Int\alpha\cap\Int\iota(\alpha)= \emptyset$. 
We assume $\Int\alpha\cap\Int\iota(\alpha)\ne \emptyset$. 
Since $c$ is simple and contains $\alpha$ and $\iota(\alpha)$, 
$\alpha$ and $\iota(\alpha)$ must coincide. 
In particular, we have $\partial\alpha = \partial\iota(\alpha)$. 
So $\beta\cup\iota(\beta)$ forms a simple closed curve, 
and this curve is null-homotopic 
because both of the arcs $\beta$ and $\iota(\beta)$ are homotopic to $\alpha=\iota(\alpha)$ relative to their boundaries. 
Since $T(c)$ is simple and contains $\beta$ and $\iota(\beta)$, 
$T(c)$ and $\beta\cup\iota(\beta)$ must coincide. 
This contradicts that $T(c)$ is essential. 
In the same way, we can show that $\Int \beta\cap \partial\iota(\Delta)=\emptyset$.
\end{proof}

Let $\Sigma_g^\iota$ denote the fixed point set of the involution $\iota$ on $\Sigma_g$.

\begin{lem}\label{lem:branch-scc}
If $c$ is non-separating, the set $c\cap\Sigma_g^\iota$ consists of $2$ points, and
\[
c\cap\Sigma_g^\iota=T(c)\cap\Sigma_g^\iota.
\]

If $c$ is separating, 
\[
c\cap\Sigma_g^\iota=T(c)\cap\Sigma_g^\iota=\emptyset.
\]
\end{lem}

\begin{proof}
Endow the curves $c$ and $T(c)$ with arbitrary orientations. 

First, consider the case when $c$ is a non-separating simple closed curve. In this case, the curve $T(c)$ is also non-separating. They represent nontrivial homology classes in $H_1(\Sigma_g;\mathbf{Z})$. Since the involution $\iota$ acts on $H_1(\Sigma_g;\mathbf{Z})$ by $-1$, it changes the orientations of $c$ and $T(c)$. Hence, both of the sets $c\cap\Sigma_g^\iota$ and $T(c)\cap\Sigma_g^\iota$ consist of 2 points.

We will show that $T(c)\cap\Sigma_g^\iota=c\cap\Sigma_g^\iota$.
Since $c$ and $T(c)$ are isotopic, the Dehn twists $t_c$ and $t_{T(c)}$ represent the same element in $\mathcal{H}_g$. The mapping classes $\Psi([t_c])$ and $\Psi([t_{T(c)}])$ in $\mathcal{M}_0^{2g+2}$ permute the branched points $p(c\cap\Sigma_g^\iota)$ and $p(T(c)\cap\Sigma_g^\iota)$, respectively. Hence, the sets $p(c\cap\Sigma_g^\iota)$ and $p(T(c)\cap\Sigma_g^\iota)$ coincide. It shows that $c\cap\Sigma_g^\iota=T(c)\cap\Sigma_g^\iota$.

Next, let $c$ be a separating simple closed curve. Since $\iota$ preserves the orientations of the subsurfaces bounded by $c$ or $T(c)$, it also preserves the orientation of $c$ and $T(c)$. In general, if an involution acts on a circle preserving its orientation, it does not have a fixed point. Hence, we have $c\cap\Sigma_g^\iota=T(c)\cap\Sigma_g^\iota=\emptyset$.
\end{proof}

\begin{proof}[Proof of Lemma \ref{lemma:isomH_g(c)}]
Let $c$ be a non-separating curve. By Lemma \ref{lem:branch-scc}, the geometric intersection number of $c$ and $T(c)$ is at least $2$. Hence, there is an innermost bigon $\Delta$. By Lemma \ref{lemma:x1x2}, the cardinality $\sharp(\Delta\cap\iota(\Delta))$ is equal to $0$, $1$, or $2$ as shown in Figure \ref{bigons}.

\begin{figure}[htbp]
\begin{center}
\includegraphics[width=110mm]{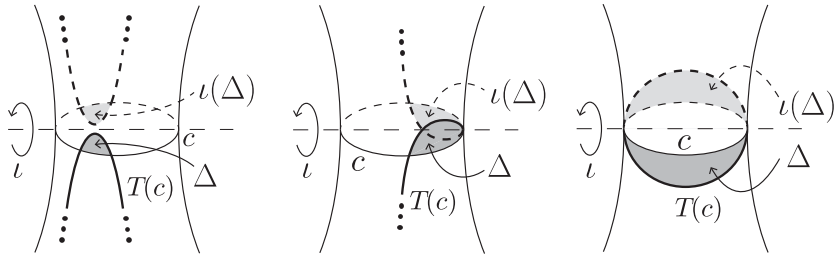}
\end{center}
\caption{Left: $\sharp(\Delta\cap\iota(\Delta))=0$. Center: $\sharp(\Delta\cap\iota(\Delta))=1$. Right: $\sharp(\Delta\cap\iota(\Delta))=2$. 
The bold curves describe the curves $T(c)$. }
\label{bigons}
\end{figure} 

Firstly, assume that $\sharp(\Delta\cap\iota(\Delta))=0$. In this case, there is a symmetric isotopy $L_1:\Sigma_g\times[0,1]\to \Sigma_g$ such that $L_1(*,0)$ is the identity, and $L_1(*,1)$ collapses the bigon $\Delta$ as in Figure \ref{isotopyL1}. Therefore, we can decrease the geometric intersection number of $c$ and $T(c)$ by $4$ by replacing the diffeomorphism $T$ by $L_1(*,1)T$.

\begin{figure}[htbp]
\begin{center}
\includegraphics[width=70mm]{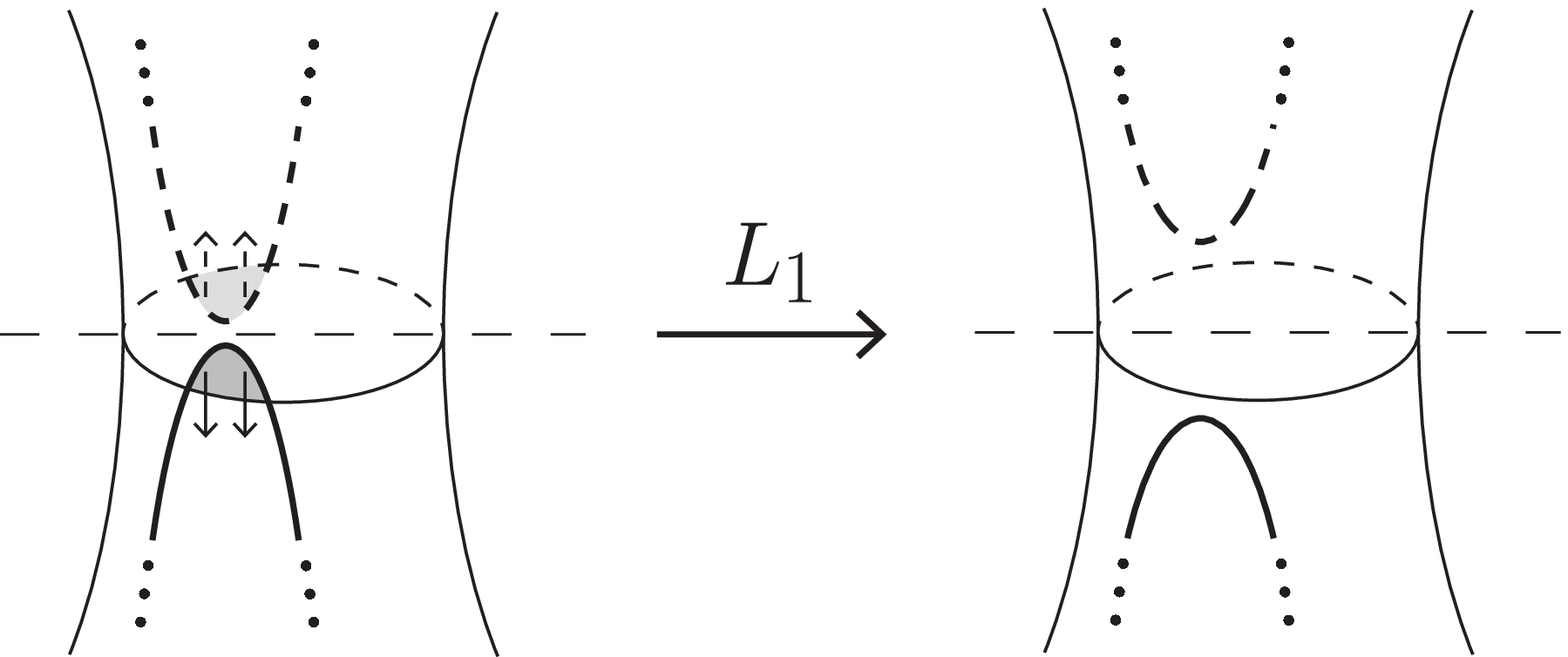}
\end{center}
\caption{}
\label{isotopyL1}
\end{figure}

Secondly, assume that $\sharp(\Delta\cap\iota(\Delta))=1$. In this case, we also have a symmetric isotopy $L_2:\Sigma_g\times[0,1]\to \Sigma_g$ which decreases the geometric intersection number by $2$ as in Figure \ref{isotopyL2}. Note that $\Delta\cap\iota(\Delta)$ is a branched point, and $L_2(*,t)$ fixes it for any $t\in[0,1]$.

\begin{figure}[htbp]
\begin{center}
\includegraphics[width=70mm]{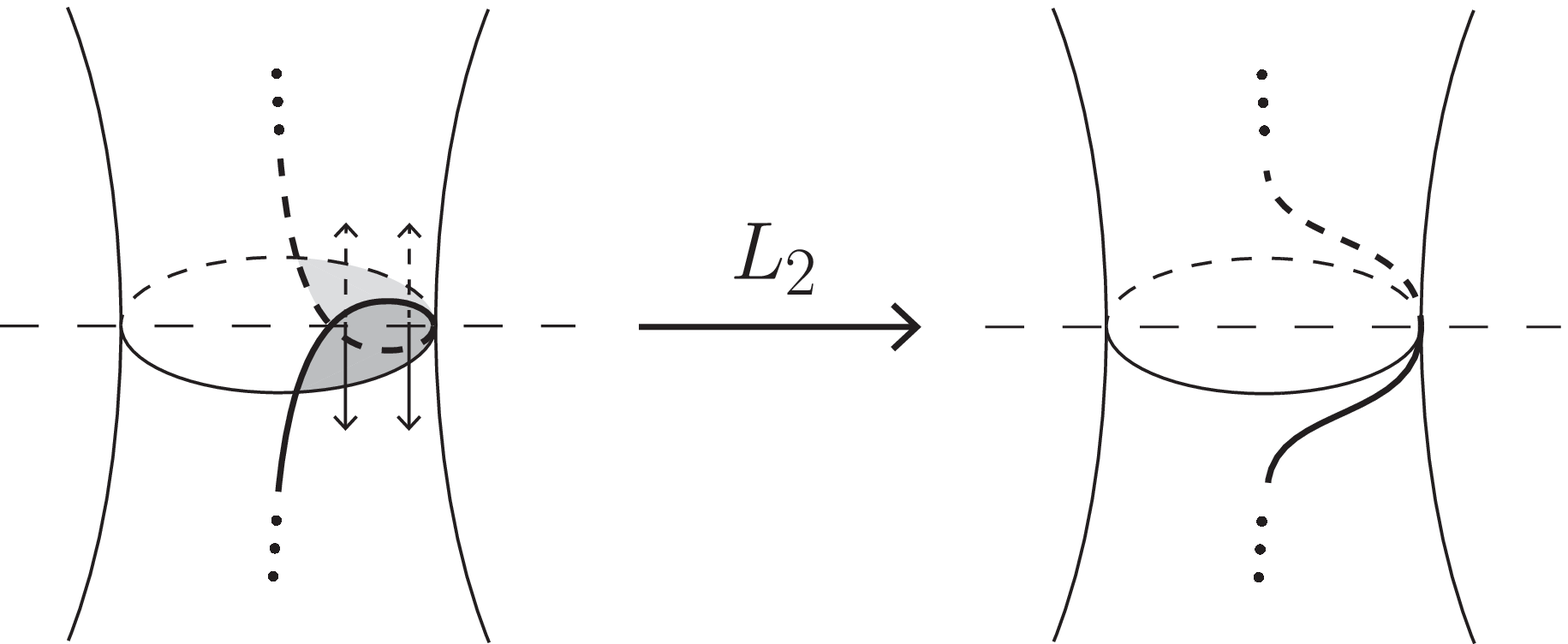}
\end{center}
\caption{}
\label{isotopyL2}
\end{figure}

After replacing the diffeomorphism $T$ in these two cases, the branch points $\{x_1,x_2\}$ remains in $c\cap T(c)$. Hence, if we repeat to replace $T$, the case when $\sharp(\Delta\cap\iota(\Delta))=2$ will definitely occur. In this case, there is a symmetric isotopy $L_3:\Sigma_g\times[0,1]\to \Sigma_g$ such that
\begin{align*}
&L_3(*,0)\text{ is the identity map},\\
&L_3(\beta,1)=\alpha,\\
&L_3(\iota(\beta),1)=\iota(\alpha),
\end{align*}
as in Figure \ref{isotopyL3}. It indicates that $L_3(*,1)T$ preserves the curve $c$. By combining these isotopies, we have obtained a desired symmetric isotopy.

\begin{figure}[htbp]
\begin{center}
\includegraphics[width=70mm]{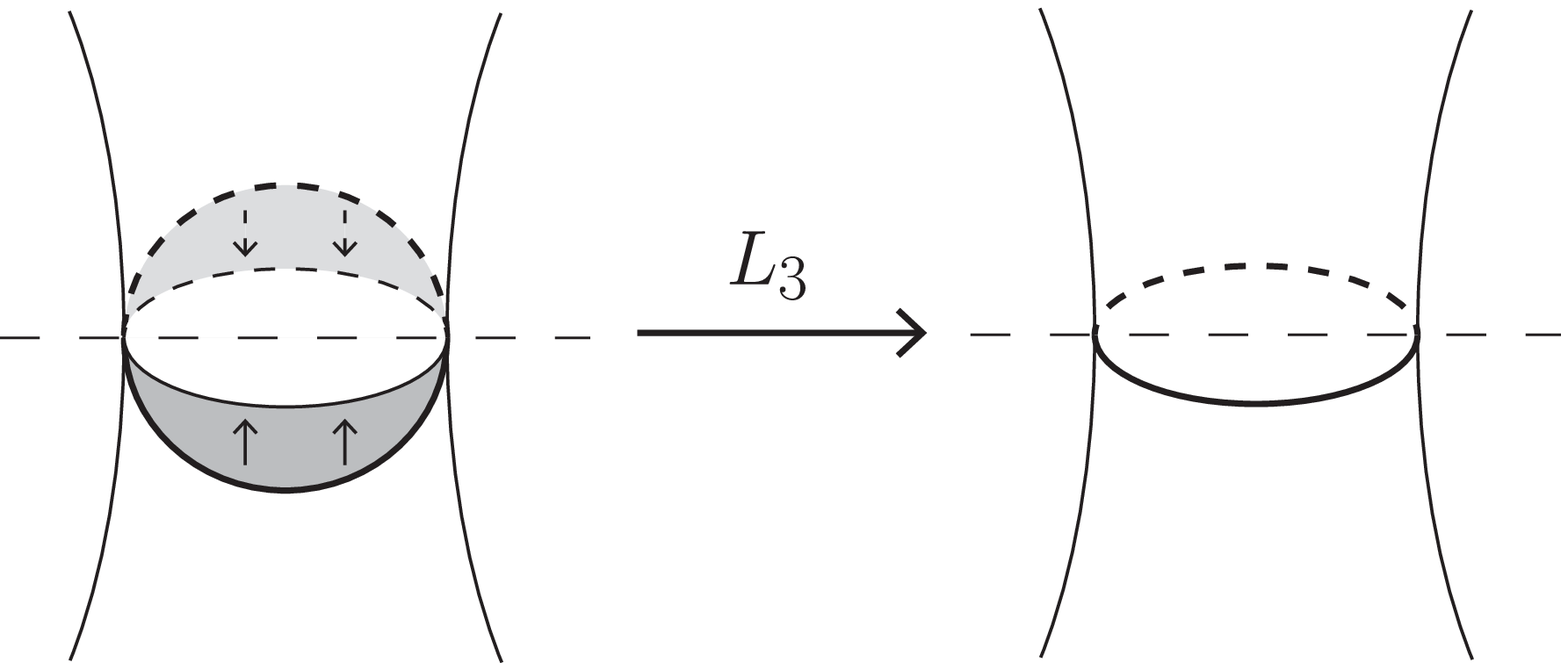}
\end{center}
\caption{}
\label{isotopyL3}
\end{figure}

Next, let $c$ be a separating curve. If the geometric intersection number of $c$ and $T(c)$ is 0, the curves $c$ and $T(c)$ bound an annulus $A$. 
Since $\iota$ acts on $A$ without fixed points, $A/\braket{\iota}$ is also an annulus. 
Hence, we can make a symmetric isotopy which moves $T(c)$ to $c$. 

Assume that the geometric intersection number is not 0. Since we have $c\cap \Sigma_g^\iota=T(c)\cap \Sigma_g^\iota=\emptyset$, the cardinality $\sharp(\Delta\cap\iota(\Delta))\ne1$. By Lemma \ref{lemma:x1x2}, we have $\sharp(\Delta\cap\iota(\Delta))=0\text{ or }2$. By the same argument as in the case when $c$ is non-separating, we can collapse the bigons $\Delta$ and $\iota(\Delta)$.
\end{proof}


\section{An involution on HSBLF}\label{section:involution}

In this section, we prove Theorem \ref{main1}. 

\begin{proof}[Proof of (i) in Theorem \ref{main1}]

Let $f:M\rightarrow S^2$ be genus-$g\geq 3$ HSBLF, 
$c_1,\ldots,c_n\subset\Sigma_g$ vanishing cycles of Lefschetz singularities of $f$ 
and $c\subset\Sigma_g$ a vanishing cycle of indefinite fold singularities of $f$. 
We assume that $c_1,\ldots,c_n$ and $c$ are preserved by the involution $\iota:\Sigma_g\rightarrow \Sigma_g$. 
By the argument in subsection \ref{subsec_handle}, 
We can decompose $M$ as follows: 
\[
M = D^2\times \Sigma_g \cup (h^2_1\amalg\cdots\amalg h^2_n)\cup R^2 \cup D^2\times \Sigma_{g-1}, 
\]
where $h^2_i=D^2\times D^2$ is the $2$-handle attached along the simple closed curve $\{p_i\}\times c_i\in \partial D^2\times \Sigma_g$ 
and $R^2$ is a round $2$-handle. 
We first prove existence of an involution $\omega$ by using the above decomposition. 

\vspace{.5em}

\noindent
{\bf Step 1}: 
We define an involution $\omega_1$ on $D^2\times \Sigma_g$ as follows: 
\[
\begin{array}{rccc}
\omega_1=\text{id}\times \iota : & D^2\times \Sigma_g & \longrightarrow & D^2\times \Sigma_g  \\
& \rotatebox{90}{$\in$} &  & \rotatebox{90}{$\in$} \\ 
& (z,x) & \longmapsto     & (z,\iota(x)).  
\end{array}
\]
In the following steps, we will define an involution on each component in the above decomposition of $M$ 
which is compatible with the involution $\omega_1$. 

\vspace{.5em}

\noindent
{\bf Step 2}: 
We next define an involution $\omega_{2,i}$ on the $2$-handle $h^2_i$ 
attached along $\{q_i\}\times c_i\subset \partial D^2\times \Sigma_g$. 
We will abuse the notation by denoting the attaching circle $\{q_i\}\times c_i$ by $c_i$. 

We take a tubular neighborhood $\nu c_i$ in $\{q_i\}\times \Sigma_g$ and an identification
\[
\nu c_i\cong S^1\times [-1,1]
\]
so that $c_i=S^1\times \{0\}$. 
We assume that the standard orientation of $S^1\times [-1,1]$ 
corresponds to the orientation of $\{q_i\}\times \Sigma_g$. 
We take a sufficiently small neighborhood $I_{q_i}$ of $q_i$ in $\partial D^2$ as follows: 
\[
I_{q_i}=\{q_i\cdot \text{exp}(\sqrt{-1}\theta)\in\partial D^2 | \theta \in[-\varepsilon_1,\varepsilon_1]\}, 
\]
where $\varepsilon_1>0$ is a sufficiently small number. 
Moreover, we identify the neighborhood $I_{q_i}$ with the unit interval $[-1,1]$ 
by using the following map: 
\[
\begin{array}{ccc}
 [-1,1] & \xrightarrow{\sim} & I_{q_i}  \\
 \rotatebox{90}{$\in$} &  & \rotatebox{90}{$\in$} \\ 
 s & \longmapsto     & q_i\cdot \text{exp}(\sqrt{-1}\varepsilon_1 s).  
\end{array}
\]
We regard $I_{q_i}\times [-1,1]$ the subset of $\mathbb{C}$ by the following embedding: 
\[
\begin{array}{ccc}
 I_{q_i}\times [-1,1] & \hookrightarrow & \{z\in\mathbb{C}|\left| \Re{z}\right|\leq 1,\left| \Im{z}\right|\leq 1 \}  \\
 \rotatebox{90}{$\in$} &  & \rotatebox{90}{$\in$} \\ 
 (s,t) & \longmapsto     & s+t\sqrt{-1}.  
\end{array}
\]
We put $J=\{z\in\mathbb{C}|\left| \Re{z}\right|\leq 1,\left| \Im{z}\right|\leq 1 \}$. 
The orientation of $\partial D^2\times \Sigma_g$ is reverse to the natural orientation of $J\times S^1$. 
So the attaching map of the $2$-handle $h^2_i$ is described as follows: 
\[
\begin{array}{rccc}
\varphi_i : & \partial D^2\times D^2 & \longrightarrow & J\times S^1\subset \partial D^2\times \Sigma_g  \\
& \rotatebox{90}{$\in$} &  & \rotatebox{90}{$\in$} \hspace{6em}\\ 
& (w_1,w_2) & \longmapsto     & (\varepsilon_2 w_2w_1, w_1) , \hspace{3.5em}
\end{array}
\]
where $\varepsilon_2>0$ is a sufficiently small number. 
We remark that the map $\varphi_i$ is orientation-preserving 
if we give the natural orientation of $\partial D^2\times D^2$. 

\vspace{.5em}

\noindent
{\bf Case 2.1}: 
If $c_i$ is non-separating, we can take a tubular neighborhood $\nu c_i \cong S^1\times [-1,1]$ 
so that the involution $\omega_1$ acts on $\nu c_i$ as follows: 
\[
\begin{array}{rccc}
\omega_1|_{\nu c_i} : & S^1\times [-1,1] & \longrightarrow & S^1\times [-1,1]  \\
& \rotatebox{90}{$\in$} &  & \rotatebox{90}{$\in$} \\ 
& (z,t) & \longmapsto     & (\overline{z}, -t).
\end{array}
\]
Since the involution $\omega_1:D^2\times \Sigma_g\rightarrow D^2\times\Sigma_g$ preserves the first component, 
$\omega_1$ acts on $I_{q_i}\times\nu c_i\cong J\times S^1$ as follows: 
\[
\begin{array}{rccc}
\omega_1|_{J\times S^1} : & J\times S^1 & \longrightarrow & J\times S^1  \\
& \rotatebox{90}{$\in$} &  & \rotatebox{90}{$\in$} \\ 
& (z_1,z_2) & \longmapsto     & (\overline{z_1},\overline{z_2}).  
\end{array}
\]
We define an involution $\omega_{2,i}$ on the $2$-handle $h^2_i$ as follows: 
\[
\begin{array}{rccc}
\omega_{2,i} : & D^2\times D^2 & \longrightarrow & D^2\times D^2  \\
& \rotatebox{90}{$\in$} &  & \rotatebox{90}{$\in$} \\ 
& (w_1,w_2) & \longmapsto     & (\overline{w_1},\overline{w_2}).  
\end{array}
\]
Then the following diagram commutes: 
\[
\begin{CD}
\partial D^2\times D^2 @>\omega_{2,i}>> \partial D^2\times D^2 \\
@V \varphi_i VV @VV \varphi_i V \\
J\times S^1 @>\omega_1>>J\times S^1. 
\end{CD}
\]
So we can define an involution $\omega_1\cup\omega_{2,i}$ on the manifold $D^2\times \Sigma_g\cup_{\varphi_i}h^2_i$. 

\vspace{.5em}

\noindent
{\bf Case 2.2}: 
If $c_i$ is separating, we can take a tubular neighborhood $\nu c_i \cong S^1\times [-1,1]$ 
so that the involution $\omega_1$ acts on $\nu c_i$ as follows: 
\[
\begin{array}{rccc}
\omega_1|_{\nu c_i} : & S^1\times [-1,1] & \longrightarrow & S^1\times [-1,1]  \\
& \rotatebox{90}{$\in$} &  & \rotatebox{90}{$\in$} \\ 
& (z,t) & \longmapsto     & (-z, t).
\end{array}
\]
Then $\omega_1$ acts on $I_{q_i}\times\nu c_i\cong J\times S^1$ as follows: 
\[
\begin{array}{rccc}
\omega_1|_{J\times S^1} : & J\times S^1 & \longrightarrow & J\times S^1  \\
& \rotatebox{90}{$\in$} &  & \rotatebox{90}{$\in$} \\ 
& (z_1,z_2) & \longmapsto     & (z_1,-z_2).  
\end{array}
\]
We define an involution $\omega_{2,i}$ on the $2$-handle $h^2_i$ as follows: 
\[
\begin{array}{rccc}
\omega_{2,i} : & D^2\times D^2 & \longrightarrow & D^2\times D^2  \\
& \rotatebox{90}{$\in$} &  & \rotatebox{90}{$\in$} \\ 
& (w_1,w_2) & \longmapsto     & (-w_1,-w_2).  
\end{array}
\]
Then the following diagram commutes: 
\[
\begin{CD}
\partial D^2\times D^2 @>\omega_{2,i}>> \partial D^2\times D^2 \\
@V \varphi_i VV @VV \varphi_i V \\
J\times S^1 @>\omega_1>>J\times S^1. 
\end{CD}
\]
So we can define an involution $\omega_1\cup\omega_{2,i}$ on the manifold $D^2\times \Sigma_g\cup_{\varphi_i}h^2_i$. 

Combining Case 2.1 and Case 2.2, we can define the involution $\tilde{\omega}_2=\omega_1\cup (\omega_{2,1}\cup\cdots\cup\omega_{2,n})$
on the $4$-manifold $M_h=M\cup (h^2_1\amalg\cdots\amalg h^2_n)$. 
Before giving an involution on the round $2$-handle, 
we look at the $\Sigma_g$-bundle structure of $\partial M_h$. 
The projection $\pi_h:\partial M_h \rightarrow \partial D^2$ of this bundle is described as follows: 
\begin{align*}
\pi_h(z,x) & =z & ((z,x)\in \partial D^2\times \Sigma_g\setminus (\amalg \Int{\varphi_i(\partial D^2\times D^2)} ), \\
\pi_h(w_1,w_2) & =q_i\cdot \text{exp}(\sqrt{-1}\varepsilon_1\varepsilon_2(\Re{w_1}\Re{w_2}-\Im{w_1}\Im{w_2})) & ((w_1,w_2)\in D^2\times \partial D^2\subset \partial h^2_i).
\end{align*}
Indeed, the map $\pi_h$ is well-defined. 
To see this, we only need to verify the following equation: 
\[
q_i\cdot \text{exp}(\sqrt{-1}\varepsilon_1\varepsilon_2 (\Re{w_1}\Re{w_2}-\Im{w_1}\Im{w_2}))=p_1\circ \varphi_i (w_1,w_2), 
\]
where $(w_1,w_2)\in D^2\times \partial D^2\subset \partial h^2_i$ and $p_1:J\times S^1\rightarrow I_{q_i}$ is the projection. 
$p_1\circ \varphi_i (w_1,w_2)$ is calculated as follows: 
{\allowdisplaybreaks
\begin{align*}
p_1\circ \varphi_i (w_1,w_2) & = p_1(\varepsilon_2 w_2 w_1, w_1) \\
& = q_i\cdot \text{exp}(\sqrt{-1}\varepsilon_1 \Re{(\varepsilon_2 w_2 w_1)}) \\
& = q_i\cdot \text{exp} (\sqrt{-1}\varepsilon_1 \varepsilon_2(\Re{w_1}\Re{w_2}-\Im{w_1}\Im{w_2}) )
\end{align*}
}
So we can verify that $\pi_h$ is well-defined. 

\begin{lem}\label{lem:higher_monodromy}

The involution $\tilde{\omega}_2$ preserves the fibers of $\pi_h$. 
Moreover, there exists a lift $X$ of the vector field $\dfrac{d}{d\theta}\text{exp}(\sqrt{-1}\theta)$ by the map $\pi_h$
which is compatible with the involution $\tilde{\omega}_2$, that is, 
\[
\tilde{\omega}_{2\ast} (X)=X. 
\]

\end{lem}

\begin{proof}[Proof of Lemma \ref{lem:higher_monodromy}]

To show that $\tilde{\omega}_2$ preserves the fibers of $\pi_h$, 
it is sufficient to prove $\pi_h\circ \tilde{\omega}_2=\pi_h$. 
This equation can be proved easily by direct calculation. 

To prove existence of a lift $X$, we construct $X$ explicitly. 
We define a vector field $X_1$ on $\partial D^2\times \Sigma_g\setminus (\amalg \varphi_i(\partial D^2\times D^2)$ 
as follows:
\[
X_1(\text{exp}(\sqrt{-1}\theta_0), x)=\left. \dfrac{d}{d\theta}\text{exp}(\sqrt{-1}\theta)\right|_{\theta=\theta_0}
\in T_{(\text{exp}(\sqrt{-1}\theta_0), x)}(\partial D^2\times \Sigma_g), 
\]
for a point $(\text{exp}(\sqrt{-1}\theta_0), x)\in \partial D^2\times \Sigma_g\setminus (\amalg \Int{\varphi_i(\partial D^2\times D^2)}$.  
Then $X_1$ is described in $J\times S^1$ as follows: 
\[
X_1(s+t\sqrt{-1},z)=\dfrac{1}{\varepsilon_1} \left.\dfrac{\partial}{\partial s}\right|_{s}\in T_{(s+t\sqrt{-1},z)}(J\times S^1). 
\]
We also define a vector field $X_2$ on $D^2\times \partial D^2\subset\partial h^2_i$ as follows: 
\[
X_2(w_1,w_2)=\dfrac{\varrho(\left|w_1\right|^2)}{\varepsilon_1 \varepsilon_2 \left|w_1\right|^2}\left(x_1\dfrac{\partial}{\partial x_2}-y_1\dfrac{\partial}{\partial y_2}  \right)
+\dfrac{1-\varrho(\left|w_1\right|^2)}{\varepsilon_1 \varepsilon_2}\left( x_2\dfrac{\partial}{\partial x_1}-y_2\dfrac{\partial}{\partial y_1} \right), 
\]
where $w_i=x_i+y_i\sqrt{-1}$ and $\varrho:[0,1]\rightarrow [0,1]$ is a monotone increasing smooth function 
which satisfies the following conditions: 

\begin{itemize}

\item $\varrho(t)=0$ for $t\in\left[0,\frac{1}{3}\right]$; 

\item $\varrho(t)=1$ for $t\in\left[\frac{2}{3},1\right]$. 

\end{itemize} 

\noindent
Then, for $(w_1,w_2)\in \partial D^2\times \partial D^2$, $d\varphi_i(X_2(w_1,w_2))$ is calculated as follows: 
{\allowdisplaybreaks
\begin{align*}
& d\varphi_i(X_2(w_1,w_2)) \\
= & d\varphi_i\left(\dfrac{1}{\varepsilon_1 \varepsilon_2}\left(x_1\dfrac{\partial}{\partial x_2}-y_1\dfrac{\partial}{\partial y_2}  \right)\right) & (\because \left| w_1 \right| = 1) \\
= & \dfrac{1}{\varepsilon_1 \varepsilon_2}x_1 d\varphi_i\left(\dfrac{\partial}{\partial x_2}\right)-
\dfrac{1}{\varepsilon_1 \varepsilon_2}y_1 d\varphi_i\left(\dfrac{\partial}{\partial y_2}  \right) \\
= & \dfrac{1}{\varepsilon_1}x_1 \left(x_1 \dfrac{\partial}{\partial s}+y_1\dfrac{\partial}{\partial t} \right)-
\dfrac{1}{\varepsilon_1}y_1 \left(-y_1\dfrac{\partial}{\partial s}+x_1\dfrac{\partial}{\partial t} \right) \\
= & \dfrac{1}{\varepsilon_1}({x_1}^2+{y_1}^2)\dfrac{\partial}{\partial s} \\
= & X_1(\varphi_i(w_1,w_2)). 
\end{align*}
}
Hence we can define a vector field $X=X_1\cup X_2$ on the manifold $\partial M_h$. 
Moreover, it can be shown that $X_1$ and $X_2$ is a lift of the vector field $\dfrac{d}{d\theta}\text{exp}(\sqrt{-1}\theta)$ 
by the map $\pi_h$. 
Thus, the vector field $X$ is a lift of $\dfrac{d}{d\theta}\text{exp}(\sqrt{-1}\theta)$. 
We can show that the vector field $X$ is compatible with the involution $\tilde{\omega}_2$ by direct calculation. 
This completes the proof of Lemma \ref{lem:higher_monodromy}. 

\end{proof}

We choose a base point $q_0\in \partial D^2\setminus (\amalg I_{q_i})$ 
and define a map $\Theta_X:f^{-1}(q_0)\rightarrow f^{-1}(q_0)$ as follows: 
\[
\Theta_X(x)=c_{X,x}(2\pi), 
\] 
where $c_{X,x}$ is the integral curve of the vector field $X$ constructed in Lemma \ref{lem:higher_monodromy} 
which satisfies $c_{X,x}(0)=x$. 
We identify $f^{-1}(q_0)$ with the surface $\Sigma_g$ via the projection onto the second component. 
Then the map $\Theta_X$ is contained in the centralizer $C(\iota)\subset\Diff_+\Sigma_g$ 
since the vector field $X$ is compatible with $\tilde{\omega}_2$. 
The isotopy class represented by $\Theta_X$ is the monodromy of the boundary of $M_h$. 
In particular, this class is contained in the group $\mathcal{H}_g(c)$. 
By Lemma \ref{lemma:isomH_g(c)}, there exists an isotopy $H_t:\Sigma_g\rightarrow \Sigma_g$ 
satisfying the following conditions: 
\begin{itemize}

\item $H_0=\Theta_X$; 

\item $H_1$ preserves the curve $c$ as a set; 

\item for each level $t$, $H_t$ is in the centralizer $C(\iota)$. 

\end{itemize}
Thus, we obtain the following isomorphism as $\Sigma_g$-bundles: 
\[
\partial M_h \cong [0,1]\times \Sigma_g/((1,x)\sim(0,H_1(x))). 
\]
We identify the above $\Sigma_g$-bundles via the isomorphism. 
Then the involution $\tilde{\omega}_2$ acts on $\partial M_h$ as follows: 
\[
\tilde{\omega}_2(t,x)=(t,\iota(x)), 
\] 
where $(t,x)$ is an element in $[0,1]\times \Sigma_g/((1,x)\sim(0,H_1(x)))\cong\partial M_h $. 

\vspace{.5em}

\noindent
{\bf Step 3}: 
In this step, we define an involution $\omega_3$ on the round $2$-handle $R^2$. 
Since $c$ is non-separating and $c$ is preserved by $\iota$, $c$ contains two fixed points of the involution $\iota$. 
We denote these points by $v_1$ and $v_2$. 
In addition, we can take a tubular neighborhood $\nu c\cong S^1\times [-1,1]$ in $\Sigma_g$
so that the involution $\iota$ acts on $\nu c$ as follows: 
\[
\iota(z,t)=(\overline{z}, -t). 
\]
By perturbing the map $H_1$, we can assume that $H_1$ preserves the neighborhood $\nu c$. 
Then the attaching region of the round $2$-handle $R^2$ is $[0,1]\times \nu c/((1,x)\sim(0,H_1(x)))$. 

\vspace{.5em}

\noindent
{\bf Case 3.1}: 
If $H_1$ preserves the orientation of $c$ and two points $v_1$ and $v_2$, 
then the round handle $R^2$ is untwisted and the restriction $H_1|_{\nu c}$ is described as follows: 
\[
H_1(z,t)=(z,t), 
\]
where $(z,t)\in S^1\times [-1,1]\cong \nu c$. 
Moreover, the attaching map of the round handle is described as follows: 
\[
\begin{array}{rccc}
\varphi : & [0,1]\times \partial D^2\times D^1/\sim & \longrightarrow & [0,1] \times S^1\times [-1,1]/\sim  \\
& \rotatebox{90}{$\in$} &  & \rotatebox{90}{$\in$} \\ 
& (s,z,t) & \longmapsto     & (s,z,t),  
\end{array}
\]
where $[0,1]\times \partial D^2\times D^1$ is the attaching region of $R^2$ 
and $[0,1]\times S^1\times [-1,1]\cong[0,1]\times \nu c$ is the subset of $\partial M_h$. 
We define an involution $\omega_3$ on the round handle as follows: 
\[
\begin{array}{rccc}
\omega_3 : & [0,1]\times D^2\times D^1/\sim & \longrightarrow & [0,1] \times D^2\times D^1/\sim  \\
& \rotatebox{90}{$\in$} &  & \rotatebox{90}{$\in$} \\ 
& (s,z,t) & \longmapsto     & (s,\overline{z},-t),  
\end{array}
\]
Then the following diagram commutes: 
\[
\begin{CD}
[0,1]\times\partial D^2\times D^1 @>\omega_3>> [0,1]\times\partial D^2\times D^1 \\
@V \varphi VV @VV \varphi V \\
[0,1]\times S^1\times [-1,1] @>\tilde{\omega}_2>>[0,1]\times S^1\times [-1,1]. 
\end{CD}
\]
Therefore, we obtain an involution $\tilde{\omega}_3=\tilde{\omega}_2\cup\omega_3$ on $M_h\cup M_r=M_h\cup R^2$. 

\vspace{.5em}

\noindent
{\bf Case 3.2}: 
If $H_1$ preserves the orientation of $c$ but does not preserve two points $v_1$ and $v_2$, 
then the round handle $R^2$ is untwisted and the restriction $H_1|_{\nu c}$ is described as follows: 
\[
H_1(z,t)=(-z,t), 
\]
The attaching map of the round handle is described as follows: 
\[
\begin{array}{rccc}
\varphi : & [0,1]\times \partial D^2\times D^1/\sim & \longrightarrow & [0,1] \times S^1\times [-1,1]/\sim  \\
& \rotatebox{90}{$\in$} &  & \rotatebox{90}{$\in$} \\ 
& (s,z,t) & \longmapsto     & (s,\text{exp}(\pi\sqrt{-1}s)z,t). 
\end{array}
\]
We define an involution $\omega_3$ on the round handle as follows: 
\[
\begin{array}{rccc}
\omega_3 : & [0,1]\times D^2\times D^1/\sim & \longrightarrow & [0,1] \times D^2\times D^1/\sim  \\
& \rotatebox{90}{$\in$} &  & \rotatebox{90}{$\in$} \\ 
& (s,z,t) & \longmapsto     & (s,\text{exp}(-2\pi\sqrt{-1}s)\overline{z},-t),  
\end{array}
\]
Then we can define an involution $\tilde{\omega}_3=\tilde{\omega}_2\cup\omega_3$ on $M_h\cup M_r=M_h\cup R^2$ 
by the same reason as in Case 3.1. 

\vspace{.5em}

\noindent
{\bf Case 3.3}: 
If $H_1$ does not preserve the orientation of $c$ but preserves two points $v_1$ and $v_2$, 
then the round handle $R^2$ is twisted and the restriction $H_1|_{\nu c}$ is described as follows: 
\[
H_1(z,t)=(\overline{z},-t), 
\]
where $(z,t)\in S^1\times [-1,1]\cong \nu c$. 
Moreover, the attaching map of the round handle is described as follows: 
\[
\begin{array}{rccc}
\varphi : & [0,1]\times \partial D^2\times D^1/\sim & \longrightarrow & [0,1] \times S^1\times [-1,1]/\sim  \\
& \rotatebox{90}{$\in$} &  & \rotatebox{90}{$\in$} \\ 
& (s,z,t) & \longmapsto     & (s,z,t). 
\end{array}
\]
We define an involution $\omega_3$ on the round handle as follows: 
\[
\begin{array}{rccc}
\omega_3 : & [0,1]\times D^2\times D^1/\sim & \longrightarrow & [0,1] \times D^2\times D^1/\sim  \\
& \rotatebox{90}{$\in$} &  & \rotatebox{90}{$\in$} \\ 
& (s,z,t) & \longmapsto     & (s,\overline{z},-t),  
\end{array}
\]
Then we can define an involution $\tilde{\omega}_3=\tilde{\omega}_2\cup\omega_3$ on $M_h\cup M_r=M_h\cup R^2$. 

\vspace{.5em}

\noindent
{\bf Case 3.4}: 
If $H_1$ preserves neither the orientation of $c$ nor two points $v_1$ and $v_2$, 
then the round handle $R^2$ is twisted and the restriction $H_1|_{\nu c}$ is described as follows: 
\[
H_1(z,t)=(-\overline{z},-t), 
\]
where $(z,t)\in S^1\times [-1,1]\cong \nu c$. 
Moreover, the attaching map of the round handle is described as follows: 
\[
\begin{array}{rccc}
\varphi : & [0,1]\times \partial D^2\times D^1/\sim & \longrightarrow & [0,1] \times S^1\times [-1,1]/\sim  \\
& \rotatebox{90}{$\in$} &  & \rotatebox{90}{$\in$} \\ 
& (s,z,t) & \longmapsto     &  (s,\text{exp}(\pi\sqrt{-1}s)z,t). 
\end{array}
\]
We define an involution $\omega_3$ on the round handle as follows: 
\[
\begin{array}{rccc}
\omega_3 : & [0,1]\times D^2\times D^1/\sim & \longrightarrow & [0,1] \times D^2\times D^1/\sim  \\
& \rotatebox{90}{$\in$} &  & \rotatebox{90}{$\in$} \\ 
& (s,z,t) & \longmapsto     & (s,\text{exp}(-2\pi\sqrt{-1}s)\overline{z},-t),  
\end{array}
\]
Then we can define an involution $\tilde{\omega}_3=\tilde{\omega}_2\cup\omega_3$ on $M_h\cup M_r=M_h\cup R^2$. 

Eventually, we obtain the involution $\tilde{\omega}_3$ on $M_h\cup M_r$ in any cases. 
Next we look at $\Sigma_{g-1}$-bundle structure of $\partial(M_h\cup M_r)$. 
The projection $\pi_r:\partial (M_h\cup M_r)\rightarrow [0,1]/\{0,1\}$ of this bundle is described as follows: 
\begin{align*}
\pi_r(s,x) =s & \hspace{1em}\bigl((s,x)\in \bigl([0,1]\times \Sigma_g/(1,x)\sim(0,H_1(x))\bigr)\setminus \left([0,1]\times \nu c/\sim \right) \bigr); \\
\pi_r(s,z,t) =s & \hspace{1em} \bigl( (s,z,t)\in [0,1]\times D^2\times \partial D^1 \bigr). 
\end{align*}
Indeed, it is easy to show that $\pi_r$ is well defined. 

\begin{lem}\label{lem:round_monodromy}

The involution $\tilde{\omega}_3$ preserves the fibers of $\pi_r$. 
Moreover, there exists a lift $\tilde{X}$ of the vector field $\dfrac{d}{ds}$ on $[0,1]/\{0,1\}$ by the map $\pi_r$
which is compatible with the involution $\tilde{\omega}_3$. 

\end{lem}

\begin{proof}[Proof of Lemma \ref{lem:round_monodromy}]

It is obvious that the involution $\tilde{\omega}_3$ preserves the fibers of $\pi_r$. 
We construct $\tilde{X}$ as we do in Lemma \ref{lem:higher_monodromy}. 
We define a vector field $\tilde{X}_1$ on $\bigl([0,1]\times \Sigma_g/\sim\bigr)\setminus\bigl([0,1]\times \nu c/\sim\bigr)$ as follows: 
\[
\tilde{X}_1(s,x)=\dfrac{d}{ds}.  
\]

We first consider the case $H_1$ preserves the points $v_1$ and $v_2$. 
Then we define a vector field $\tilde{X}_2$ on the round handle $R^2$ as follows: 
\[
\tilde{X}_2(s,z,t)=\dfrac{d}{ds}, 
\]
where $(s,z,t)\in [0,1]\times D^2\times \partial D^1\subset \partial R^2$. 
Then it is easy to show that $d\varphi\left(\dfrac{d}{ds}\right)=\dfrac{d}{ds}$. 
Hence we can define vector field $\tilde{X}=\tilde{X}_1\cup\tilde{X}_2$ on $\partial M_h\cup M_r$. 
It is obvious that $\tilde{X}$ is a lift of the vector field $\dfrac{d}{ds}$ on $[0,1]/\{0,1\}$ by $\pi_r$ 
and is compatible with the involution $\tilde{\omega}_3$. 

We next consider the case $H_1$ does not preserve the points $v_1$ and $v_2$. 
Then we define a vector field $\tilde{X}_2$ on $R^2$ as follows: 
\[
\tilde{X}_2(s,x+y\sqrt{-1},t)=\dfrac{d}{ds}+\pi y\dfrac{\partial}{\partial x} -\pi x \dfrac{\partial}{\partial y},  
\]
where $(s,x+y\sqrt{-1},t)\in [0,1]\times D^2 \times \partial D^1\subset \partial R^2$. 
The differential $d\varphi(\tilde{X}_2(s,x+\sqrt{-1}y,t))$ is calculated as follows: 
{\allowdisplaybreaks
\begin{align*}
& d\varphi(\tilde{X}_2(s,x+\sqrt{-1}y,t)) \\
= & d\varphi\left(\dfrac{d}{ds}+\pi y\dfrac{\partial}{\partial x} -\pi x \dfrac{\partial}{\partial y}\right) \\
= & \left(\dfrac{d}{ds}+\pi (-x\sin{\pi s}-y\cos{\pi s})\dfrac{d}{dx}+\pi(x\cos{\pi s}-y\sin{\pi s})\dfrac{d}{dy}\right) \\
& \hspace{1em}+ \pi y \left(\cos{\pi s}\dfrac{d}{dx}+\sin{\pi s}\dfrac{d}{dy}\right)- \pi x \left(-\sin{\pi s}\dfrac{d}{dx}+\cos{\pi s}\dfrac{d}{dy}\right) \\
= & \dfrac{d}{ds} \\
= & \tilde{X}_1(\varphi (s,x+\sqrt{-1}y,t)). 
\end{align*}
}
Hence we can define a vector field $\tilde{X}=\tilde{X}_1\cup\tilde{X}_2$ on $\partial(M_h\cup M_r)$. 
It is obvious that $\tilde{X}$ is a lift of the vector field $\dfrac{d}{ds}$ on $[0,1]/\{0,1\}$ by $\pi_r$. 
To verify that $\tilde{X}$ is compatible with the involution $\tilde{\omega}_3$, 
We prove that the following equation holds for any points $x\in \partial(M_h\cup M_r)$: 
\[
d\tilde{\omega}_3(\tilde{X}(x))=\tilde{X}(\tilde{\omega}_3(x)). 
\]
If $x$ is contained in $[0,1]\times \Sigma_g/\sim\setminus\bigl([0,1]\times \nu c/\sim\bigr)$, 
the above equation can be proved easily. 
If $x=(s,x+\sqrt{-1}y,t)\in [0,1]\times D^2\times \partial D^1\subset \partial R^2$, 
then $d\tilde{\omega}_3(\tilde{X}(x))$ is calculated as follows: 
{\allowdisplaybreaks
\begin{align*}
& d\tilde{\omega}_3(\tilde{X}(x)) \\
= & d\tilde{\omega}_3(\dfrac{d}{ds}+\pi y\dfrac{\partial}{\partial x} -\pi x \dfrac{\partial}{\partial y}) \\
= & \left(\dfrac{d}{ds}+ 2\pi (-x\sin{2\pi s}-y\cos{2\pi s})\dfrac{\partial}{\partial x}+ 2\pi (-x\cos{2\pi s}+y\sin{2\pi s})\dfrac{\partial}{\partial y}\right) \\
& + \pi y\left( \cos{2\pi s}\dfrac{\partial}{\partial x}-\sin{2\pi s}\dfrac{\partial}{\partial y} \right) - \pi x\left( -\sin{2\pi s}\dfrac{\partial}{\partial x}-\cos{2\pi s}\dfrac{\partial}{\partial y} \right) \\
= & \dfrac{d}{ds}+ \pi (-x\sin{2\pi s}-y\cos{2\pi s})\dfrac{\partial}{\partial x}+ \pi (-x\cos{2\pi s}+y\sin{2\pi s})\dfrac{\partial}{\partial y} \\
= & \tilde{X}(\tilde{\omega}_3(x)). 
\end{align*}
}
Thus, $\tilde{X}$ is compatible with the involution $\tilde{\omega}_3$. 
This completes the proof of Lemma \ref{lem:round_monodromy}. 

\end{proof}

We define the map $\Theta_{\tilde{X}}:\pi_r^{-1}(0)\rightarrow \pi_r^{-1}(0)$ as follows: 
\[
\begin{array}{rccc}
\Theta_{\tilde{X}} : & \pi_r^{-1}(0) & \longrightarrow &\pi_r^{-1}(0)  \\
& \rotatebox{90}{$\in$} &  & \rotatebox{90}{$\in$} \\ 
& x & \longmapsto     &  c_{\tilde{X},x}(1), 
\end{array}
\]
where $c_{\tilde{X},x}$ is the integral curve of $\tilde{X}$ starting at $x$. 
We identify the fiber $\pi_r^{-1}(0)$ with the surface $\Sigma_{g-1}$. 
Then the map $\Theta_{\tilde{X}}$ is contained in the centralizer $C(\iota)$ 
since $\tilde{X}$ is compatible with $\tilde{\omega}_3$. 
Moreover, $\Theta_{\tilde{X}}$ is isotopic to the identity map. 
By Lemma \ref{thm:hypMCG}, we can take an isotopy $\tilde{H}_t:\Sigma_{g-1}\rightarrow\Sigma_{g-1}$ 
which satisfies the following conditions: 
\begin{itemize}

\item $\tilde{H}_0=\Theta_{\tilde{X}}$; 

\item $\tilde{H}_1$ is the identity map; 

\item $\tilde{H}_t$ is contained in the centralizer $C(\iota)$.  

\end{itemize}

By using this isotopy, we obtain the following isomorphism as $\Sigma_{g-1}$-bundle: 
\[
\partial(M_h\cup M_r) \cong [0,1]\times \Sigma_{g-1}/(1,x)\sim(0,x). 
\]
The involution $\tilde{\omega}_3$ acts on $[0,1]\times \Sigma_{g-1}/(1,x)\sim(0,x)$ via the above isomorphism as follows: 
\[
\tilde{\omega}_3(s,x)=(s,\iota(x)). 
\] 

\vspace{.5em}

\noindent
{\bf Step 4}: 
We define an involution $\omega_4$ on $D^2\times \Sigma_{g-1}$ as follows: 
\[
\omega_4(z,x)=(z,\iota(x)),  
\]
where $(z,x)\in D^2\times \Sigma_{g-1}$. 
Then it is obvious that the following diagram commutes: 
\[
\begin{CD}
[0,1]\times \Sigma_{g-1}/\sim @>\tilde{\omega}_3>> [0,1]\times\Sigma_{g-1}/\sim \\
@V \Phi VV @VV \Phi V \\
\partial D^2\times \Sigma_{g-1} @>\omega_4>>\partial D^2\times \Sigma_{g-1}, 
\end{CD}
\]
where $\Phi$ is the attaching map between $M_r$ and $D^2\times \Sigma_{g-1}$, 
which is given by $\Phi(s,x)=(\text{exp}(2\pi\sqrt{-1}s), x)$. 
Hence we obtain an involution $\omega=\tilde{\omega}_3\cup \omega_4$ on $M$. 

\vspace{.5em}

\par

We next look at the fixed point set of $\omega$. 
$\omega$ is equal to $\text{id}\times \iota$ on $D^2\times \Sigma_g$. 
So we obtain: 
\[
M^\omega\cup D^2\times \Sigma_g = D^2\times \{v_1,\ldots,v_{2g+2}\}, 
\]
where $v_1,\ldots,v_{2g+2}\in\Sigma_g$ are the fixed points of $\iota$.  
We remark that $M^\omega\cup D^2\times \Sigma_g$ has the natural orientation 
derived from the orientation of $D^2$. 

$\omega$ acts on the $2$-handle $h^2_i=D^2\times D^2$ as follows: 
\begin{align*}
\omega(w_1,w_2)= & \begin{cases}
(\overline{w_1},\overline{w_2}) & \text{($c_i$:non-separating)}, \\
(-w_1,-w_2) & \text{($c_i$:separating)}, 
\end{cases}
\end{align*}
where $(w_1,w_2)\in D^2\times D^2$. 
So the fixed point set ${h^2_i}^\omega$ is equal to $(D^2\cap\mathbb{R})\times (D^2\cap\mathbb{R})$ if $c_i$ is non-separating 
and is equal to $\{(0,0)\}$ if $c_i$ is separating. 
Furthermore, if $c_i$ is non-separating, we can give an orientation to $(D^2\cap\mathbb{R})\times (D^2\cap\mathbb{R})$ 
which is compatible with the orientation of $D^2\times \{v_1,\ldots,v_{2g+2}\}$. 
Hence the fixed point set ${M_h}^\omega$ is the union of the oriented surfaces and the $s$ points, 
where $s$ is the number of Lefschetz singularities of $f$ whose vanishing cycle is separating. 

$\omega$ acts on the round $2$-handle $R^2$ as follows: 
\begin{align*}
\omega(s,z,t)= & \begin{cases}
(s,\overline{z},-t) & \text{(if $H_1$ preserves the two points $v_1$ and $v_2$)}, \\
(s,\exp{(-2\pi\sqrt{-1}s)}\overline{z},-t) & \text{(otherwise)}, 
\end{cases}
\end{align*} 
where $(s,z,t)\in R^2 = [0,1]\times D^2\times D^1/\sim$. 
So we obtain: 
\begin{align*}
{R^2}^\omega = &\begin{cases}
[0,1]\times (D^2\cap\mathbb{R})\times \{0\}/\sim & \text{(if $H_1$ preserves the two points $v_1$ and $v_2$)}, \\
\{(s,z,0)\in R^2\hspace{.3em}|\hspace{.3em}z=r\exp{(-\pi\sqrt{-1}s)}, r\in [-1,1]\} & \text{(otherwise)}. 
\end{cases}
\end{align*}
Therefore, the fixed point set ${R^2}^\omega$ is equal to the annulus or the M\"{o}bius band. 
Moreover, it is easy to show that the $2$-dimensional part of the fixed point set $(M_h\cup M_r)^\omega$ 
does not admit any orientations even if ${R^2}^\omega$ is equal to the annulus. 
So $(M_h\cup M_r)^\omega$ is the union of the unorientable surfaces and the $s$ points. 

$\omega$ is equal to $\text{id}\times \iota$ on $D^2\times \Sigma_{g-1}$. 
So the fixed point set $(D^2\times \Sigma_{g-1})^\omega$ is equal to $D^2\times\{\tilde{v}_1,\ldots,\tilde{v}_{2g}\}$, 
where $\{\tilde{v}_1,\ldots,\tilde{v}_{2g}\}={\Sigma_{g-1}}^\iota$. 
Eventually, $M^\omega$ is the union of the closed surfaces and the $s$ points. 
The $2$-dimensional part of $M^\omega$ is orientable if $f$ has no indefinite fold singularities 
and is not orientable otherwise. 
This completes the proof of the statement in Theorem \ref{main1} on the fixed point set of $\omega$. 

\vspace{.5em}

We next extend the involution $\omega$ to the manifold $M\sharp s\overline{\mathbb{CP}^2}$. 
We assume that the curves $c_{k_1},\ldots,c_{k_s}$ are separating. 
We construct the manifold $M\sharp s\overline{\mathbb{CP}^2}$ by blowing up $M$ $s$ times at $(0,0)\in h^2_{k_i}$ ($i=1,\ldots,s$). 
In other word, $M\sharp s\overline{\mathbb{CP}^2}$ is decomposed as follows: 
\[
M\sharp s\overline{\mathbb{CP}^2} = D^2\times \Sigma_g\cup (h^2_1\amalg\hspace{.3em}\cdots\cdots\hspace{-3em}\raisebox{.8em}{\footnotesize{\mbox{$\hat{k_1},\ldots,\hat{k_s}$}}}\amalg h^2_n)\cup (\tilde{h}_{k_1}\amalg\cdots\amalg\tilde{h}_{k_s})\cup R^2 \cup D^2\times \Sigma_{g-1}, 
\]
where $\tilde{h}_{k_i}=\{((w_1,w_2),[l_1:l_2])\in D^2\times D^2\times \mathbb{CP}^1 \hspace{.3em} | \hspace{.3em} w_1l_2-w_2l_1=0\}\cong h_{k_i}\sharp\overline{\mathbb{CP}^2}$. 
We define an involution $\overline{\omega}$ on $M \sharp s\overline{\mathbb{CP}^2}$ as follows: 
\begin{align*}
\overline{\omega}(x)= & x & \text{($x\in M\sharp s \overline{\mathbb{CP}^2}\setminus (\tilde{h}_{k_1}\amalg\cdots\amalg \tilde{h}_{k_s})$)}, \\
\overline{\omega}((w_1,w_2),[l_1:l_2])= & ((-w_1,-w_2),[l_1:l_2]) & \text{($((w_1,w_2),[l_1:l_2])\in\tilde{h}_{k_i}$)}.  
\end{align*}
It is obvious that $\overline{\omega}$ is an extension of $\omega$. 
The fixed point set of $\overline{\omega}$ is the union of the $2$-dimensional part of $M^\omega$ and $s$ $2$-spheres. 

\par

We next prove that $M\sharp s \overline{\mathbb{CP}^2}/\overline{\omega}$ is diffeomorphic to $S\sharp 2s \overline{\mathbb{CP}^2}$, 
where $S$ is an $S^2$-bundle over $S^2$. 
Since $\Sigma_g/\iota$ is diffeomorphic to $S^2$, 
it is easy to see that $D^2\times \Sigma_g/\overline{\omega}$ is diffeomorphic to $D^2\times S^2$. 
So $M\sharp s \overline{\mathbb{CP}^2}$ is obtained by attaching $h_j/\overline{\omega}$ ($j\neq k_1,\ldots,k_s$), $\tilde{h}_{k_i}/\overline{\omega}$, 
$R^2/\overline{\omega}$ and $D^2\times \Sigma_{g-1}/\overline{\omega}\cong D^2\times S^2$ to $D^2\times S^2$. 

\begin{lem}\label{lem:quotient_nonsep}

Suppose that $c_i$ is non-separating. Then, 
\[
(D^2\times \Sigma_g\cup_{\varphi_i}h^2_i)/\overline{\omega}\cong D^2\times S^2. 
\]

\end{lem}

\begin{proof}[Proof of Lemma \ref{lem:quotient_nonsep}]

If we identify $h^2_i=D^2\times D^2$ with $D^4$, 
then $\overline{\omega}$ is equal to the covering transformation of the double covering $D^4\rightarrow D^4$ 
branched at the unknotted $2$-disk in $D^4$. 
In particular, we obtain $h^2_i/\overline{\omega}\cong D^4$. 
Moreover, the attaching region of $h^2_i$ corresponds to the $3$-disk in $\partial D^4=\partial h^2_i/\overline{\omega}$. 
Denote by $\overline{\varphi_i}:h^2_i/\overline{\omega}\rightarrow \partial D^2\times \Sigma_g/\overline{\omega}$ 
the embedding induced by $\varphi_i$. 
Then we obtain: 
{\allowdisplaybreaks
\begin{align*}
(D^2\times \Sigma_g\cup_{\varphi_i}h^2_i)/\overline{\omega} & \cong (D^2\times \Sigma_g/\overline{\omega})\cup_{\overline{\varphi_i}}h^2_i/\overline{\omega} \\
& \cong D^2\times S^2 \natural D^4 \\
& \cong D^2\times S^2. 
\end{align*}
} 
This completes the proof of Lemma \ref{lem:quotient_nonsep}. 

\end{proof}

\begin{lem}\label{lem:quotient_sep}

For each $i\in\{1,\ldots,s\}$, 
$(D^2\times \Sigma_g\cup_{\varphi_i}\tilde{h}^2_{k_i})/\overline{\omega}\cong D^2\times S^2\sharp 2\overline{\mathbb{CP}^2}$. 

\end{lem}

\begin{proof}[Proof of Lemma \ref{lem:quotient_sep}]

By eliminating the corner of $D^2\times D^2$, 
we identify $\tilde{h}^2_{k_i}$ with the following space: 
\[
H=\{((w_1,w_2),[l_1:l_2])\in D^4\times \mathbb{CP}^1\hspace{.3em} | \hspace{.3em} w_1l_2-w_2l_1=0 \}. 
\]
By the above identification, the attaching region of $\tilde{h}^2_{k_i}$ corresponds to the tubular neighborhood of the circle 
$\{((w_1,0),[1:0])\in \partial H \hspace{.3em} | \hspace{.3em} \left| w_1 \right|=1\}$ in $\partial H$. 
Let $p_2:H\rightarrow \mathbb{CP}^1$ be the projection onto the second component. 
Then $p_2$ is the $D^2$-bundle over the $2$-sphere with Euler number $-1$. 
We define $D_1,D_2\subset \mathbb{CP}^1$ and local trivializations $\psi_1$ and $\psi_2$ of $p_2$ as follows: 
\begin{align*}
D_1 = & \{[l_1:l_2]\in \mathbb{CP}^1 \hspace{.3em} | \hspace{.3em} \left|l_1\right|\geq \left|l_2\right|\}, \\
D_2 = & \{[l_1:l_2]\in \mathbb{CP}^1 \hspace{.3em} | \hspace{.3em} \left|l_2\right|\geq \left|l_1\right|\}, 
\end{align*}
\[
\begin{array}{rccc}
\psi_1 : & D^2\times D^2 & \longrightarrow &p_2^{-1}(D_1)  \\
& \rotatebox{90}{$\in$} &  & \rotatebox{90}{$\in$} \\ 
& (w_1,w_2) & \longmapsto     &  \left(\dfrac{w_2}{\sqrt{1+\left|w_1\right|^2}}(1,w_1),[1,w_1] \right), 
\end{array}
\]
\[
\begin{array}{rccc}
\psi_2 : & D^2\times D^2 & \longrightarrow &p_2^{-1}(D_2)  \\
& \rotatebox{90}{$\in$} &  & \rotatebox{90}{$\in$} \\ 
& (w_1,w_2) & \longmapsto     &  \left(\dfrac{w_2}{\sqrt{1+\left|w_1\right|^2}}(w_1,1),[w_1,1] \right). 
\end{array}
\]
Denote $p_2^{-1}(D_1)$ and $p_2^{-1}(D_2)$ by $H_1$ and $H_2$, respectively.  
We identify $H_1$ and $H_2$ with $D^2\times D^2$ by the above trivializations. 
Then $H$ can be identified with $D^2\times D^2 \cup_{\Psi} D^2\times D^2$, 
where $\Psi=\psi_1^{-1}\circ \psi_2:(w_1,w_2)\longmapsto (\dfrac{1}{w_1},w_1w_2)$. 
We remark that the attaching region of $H$ corresponds to $\partial D^2\times D^2\subset \partial H_1$. 

We define $\tilde{H}=\tilde{H}_1\cup_{\tilde{\Psi}}\tilde{H}_2$, where $\tilde{H}_i=D^2\times D^2$ ($i=1,2$) 
and $\tilde{\Psi}:\partial D^2\times D^2\rightarrow \partial D^2\times D^2$ is a diffeomorphism 
defined as follows: 
\[
\tilde{\Psi}(w_1,w_2)=(\dfrac{1}{w_1},{w_1}^2w_2). 
\] 
Then we can define $\mathcal{P}:H\rightarrow \tilde{H}$ as follows:
\begin{align*}
\mathcal{P}(w_1,w_2) = & \begin{cases}
(w_1,{w_2}^2)\in \tilde{H}_1 & \text{($(w_1,w_2)\in H_1$)}, \\
(w_1,{w_2}^2)\in \tilde{H}_2 & \text{($(w_1,w_2)\in H_2$)}. 
\end{cases}
\end{align*}
The map $\mathcal{P}$ is a double branched covering branched at the $0$-section of $\tilde{H}$ as a $D^2$-bundle.  
Moreover, $\tilde{\omega}$ is the non-trivial covering transformation of $\mathcal{P}$. 
So we obtain $H/\tilde{\omega}\cong \tilde{H}$. 

\par

Since the attaching region of $H$ is mapped by $\mathcal{P}$ to $D^2\times \partial D^2\subset \partial \tilde{H}_1$, 
we can regard $\tilde{H}_1$ and $\tilde{H}_2$ as $2$-handles. 
Thus, $(D^2\times \Sigma_g\cup_{\varphi_i}\tilde{h}^2_{k_i})/\overline{\omega}$ is obtained by attaching the $2$-handles $\tilde{H}_1$ and $\tilde{H}_2$ 
to $D^2\times S^2$. 
To prove the statement, we look at the attaching maps of $\tilde{H}_1$ and $\tilde{H}_2$. 

We take an identification $\nu c_{k_i}\cong J\times S^1$ 
as we take in Step 2 of the construction of $\omega$. 
Then the attaching map $\varphi_{k_i}$ of the $2$-handle $h^2_{k_i}$ satisfies $\varphi_{k_i}(w_1,w_2)=(\varepsilon_2w_2w_1,w_1)$. 
Since $H$ is obtained by eliminating the corner of $\tilde{h}^2_{k_i}$, 
The attaching map of $H_1$ is described as follows: 
\[
\begin{array}{rcccc}
\Phi : & \partial H_1\supset & D^2\times \partial D^2 & \longrightarrow & J\times S^1 \\
&& \rotatebox{90}{$\in$} & & \rotatebox{90}{$\in$} \\
&& (w_1,w_2) & \longmapsto & (\varepsilon_2{w_2}^2w_1,w_2). 
\end{array}
\]

For an element $(z_1,z_2)\in J\times S^1$, $\overline{\omega}(z_1,z_2)=(z_1,-z_2)$. 
So we obtain $J \times S^1/\overline{\omega}\cong J\times S^1$ 
and the quotient map $/\overline{\omega}:J\times S^1\rightarrow J\times S^1/\overline{\omega}\cong J\times S^1$ 
satisfies $/\overline{\omega}(z_1,z_2)=(z_1,{z_2}^2)$. 
Thus, the attaching map $\tilde{\Phi}:D^2\times \partial D^2\rightarrow J\times S^1$ of $\tilde{H}_1$ 
satisfies $\tilde{\Phi}(w_1,w_2)=(\varepsilon_2w_2w_1,w_2)$. 
It is easy to see that the attaching circle of $\tilde{H}_1$ is equal to the circle $c_{k_i}/\overline{\omega}$. 
Moreover, the framing of $\tilde{\Phi}$ is $-1$ relative to the framing along $\{\ast\}\times S^2\subset \partial D^2\times S^2$. 

By the definition of $\tilde{\Psi}$, the attaching circle of $\tilde{H}_2$ is equal to the belt circle of $\tilde{H}_1$, 
which is isotopic to the meridian of the attaching circle of $\tilde{H}_1$. 
In particular, there exists the natural framing of the attaching circle of $\tilde{H}_2$ 
which is represented by the meridian of the attaching circle of $\tilde{H}_1$ parallel to the attahcing circle of $\tilde{H}_2$. 
Since the Euler number of $\tilde{H}$ as a $D^2$-bundle is equal to $-2$, 
the framing of the attaching map $\tilde{\Psi}$ is equal to $-2$ relative to the natural framing. 
Therefore, we can draw a Kirby diagram of  $(D^2\times \Sigma_g\cup_{\varphi_i}\tilde{h}^2_{k_i})/\overline{\omega}$ 
as shown in Figure \ref{kirby_blowup}. It is obvious that this manifold is diffeomorphic to $D^2\times S^2\sharp 2\overline{\mathbb{CP}^2}$
and this completes the proof of Lemma \ref{lem:quotient_sep}. 

\begin{figure}[htbp]
\begin{center}
\includegraphics[width=90mm]{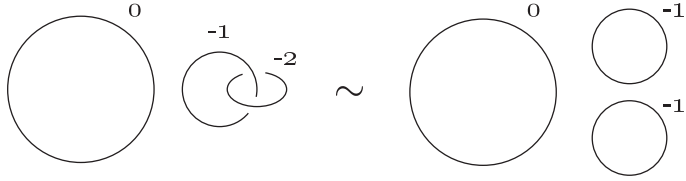}
\end{center}
\caption{the $-1$-framed knot describes $\tilde{H}_1$, while the $-2$-framed knot describes $\tilde{H}_2$. }
\label{kirby_blowup}
\end{figure}

\end{proof}

By applying the arguments in Lemma \ref{lem:quotient_nonsep} and \ref{lem:quotient_sep} successively, 
we obtain $M_h\sharp s\overline{\mathbb{CP}^2}/\overline{\omega}\cong D^2\times S^2 \sharp 2s \overline{\mathbb{CP}^2}$.  

\begin{lem}\label{lem:quotient_round}

$((M_h\cup M_r)\sharp s\overline{\mathbb{CP}^2})/\overline{\omega}\cong D^2\times S^2 \sharp 2s \overline{\mathbb{CP}^2}$. 

\end{lem}

\begin{proof}[Proof of Lemma \ref{lem:quotient_round}]

We can decompose $R^2$ into two components as follows: 
\[
R^2=\left[0,\frac{1}{2}\right]\times D^2\times D^1 \cup \left[\frac{1}{2},1\right]\times D^2\times D^1. 
\]
Denote $\left[0,\frac{1}{2}\right]\times D^2\times D^1$ and $\left[\frac{1}{2},1\right]\times D^2\times D^1$ by $R_1$ and $R_2$, respectively. 
It is easy to see that $R_i/\overline{\omega}$ is diffeomorphic to $D^4$ and 
$R_i$ is the double covering of $D^4\cong R_i/\overline{\omega}$ branched at the unknotted $2$-disk. 

The attaching region of $R_1$ is equal to $[0,\frac{1}{2}]\times \partial D^2 \times D^1$. 
The quotient $[0,\frac{1}{2}]\times \partial D^2 \times D^1/\overline{\omega}$ is a $3$-ball in $\partial D^4 \cong \partial R_1$. 
Thus, we obtain: 
{\allowdisplaybreaks
\begin{align*}
(M_h\cup R_1)/\overline{\omega} & \cong M_h/\overline{\omega}\cup R_1/\overline{\omega} \\
& \cong D^2\times S^2 \sharp 2s \overline{\mathbb{CP}^2} \natural D^4 \\
& \cong D^2\times S^2 \sharp 2s \overline{\mathbb{CP}^2}. 
\end{align*}
}

\par

The attaching region of $R_2$ is equal to $\left[\frac{1}{2},1\right]\times \partial D^2 \times D^1 \cup \lbrace{\frac{1}{2},1\rbrace}\times D^2\times D^1$. 
The quotient $\left[\frac{1}{2},1\right]\times \partial D^2 \times D^1/\overline{\omega}$ is a $3$-ball $D_0$ in $\partial D^4 \cong \partial R_2$, 
while $\lbrace{\frac{1}{2},1\rbrace}\times D^2\times D^1/\overline{\omega}$ is a disjoint union of two $3$-balls $D_1\amalg D_2$ in $\partial D^4$. 
Both of the intersections $D_0\cap D_1$ and $D_0\cap D_2$ are $2$-disks in $\partial D_0$. 
Eventually, the attaching region of $R_2$ is a $3$-ball in $\partial D^4$. 
So we obtain $(M_h\cup R_1\cup R_2)/\overline{\omega}\cong D^2\times S^2 \sharp 2s \overline{\mathbb{CP}^2}$. 
This completes the proof of Lemma \ref{lem:quotient_round}. 

\end{proof}

It is easy to see that $D^2\times \Sigma_{g-1}/\overline{\omega}$ is difeomorphic to $D^2\times S^2$ 
and attached to $(M_h\cup M_r)/\overline{\omega}$ so that the following diagram commutes: 
\[
\begin{array}{rcccl}
(M_h\cup M_r)/\overline{\omega} \supset & S^1\times S^2 & \longrightarrow & \partial D^2\times S^2 & \subset D^2\times \Sigma_{g-1}/\overline{\omega} \\
&  \raisebox{.7em}{\rotatebox{-90}{$\longrightarrow$}} & &  \raisebox{.7em}{\rotatebox{-90}{$\longrightarrow$}} & \\
& S^1 & \longrightarrow & \partial D^2, &
\end{array}
\]
where the upper horizontal arrow in the diagram represents the attaching map, 
the lower horizontal arrow represents the identity map 
and vertical arrows represent the projection onto the first component 
(In other word, the attaching map is a bundle map as a $S^2$-bundle over $S^1$). 
In particular, we obtain: 
\[
M \sharp s \overline{\mathbb{CP}^2}/\overline{\omega} \cong S \sharp 2s\overline{\mathbb{CP}^2}. 
\]
It is obvious that the quotient map $/\overline{\omega}:M\sharp s\overline{\mathbb{CP}^2}\rightarrow S\sharp 2s \overline{\mathbb{CP}^2}$ 
is a double branched covering. 
Thus, we complete the proof of the statement (i) in Theorem \ref{main1}. 

\end{proof}

\begin{proof}[Proof of (ii) in Theorem \ref{main1}]

Let $F_h\subset M$ be a regular fiber in the higher side of $f$. 
It is easy to see that $F_h$ represents the same rational homology class of $M$ as the one $F$ represents. 
Let $\omega:M\rightarrow M$ be the involution constructed in the proof of (i) in Theorem \ref{main1}. 
If $f$ has no indefinite fold singularities, that is, $f$ is a Lefschetz fibration, 
then the $2$-dimensional part of the fixed point set $M^\omega$ of the involution $\omega$ is an orientable surface 
and the algebraic intersection number between this part and $F_h$ is equal to $2g+2$, especially is non-zero. 
So the statement (ii) in Theorem \ref{main1} holds. 

Suppose that $f$ has indefinite fold singularities. 
We first prove that $F_h$ represents a non-trivial rational homology class of $M_h\cup M_r$. 
To prove this, we construct an element $\mathcal{S}$ in the group $H_2(M_h\cup M_r,\partial (M_h\cup M_r);\mathbb{Q})$ 
such that $[F_h]\cdot \mathcal{S}\neq 0$.  
Let $\tilde{S}$ be the intersection between the $2$-dimensional part of $M^\omega$ and $M_h$, 
which is the union of compact oriented surfaces. 
We use the notations $H_1$, $c$, $v_1$, $v_2$ and $R^2$ as we use in the proof of (i) in Theorem \ref{main1}. 

\vspace{.5em}

\noindent
{\bf Case 1}: If $H_1$ preserves the orientation of $c$ and two points $v_1$ and $v_2$, 
then $R^2$ is untwisted and $\tilde{S}\cap R^2=\{(s,\pm 1, 0)\in R^2 \hspace{.3em} | \hspace{.3em} s\in[0,1]\}$ 
is a disjoint union of two circles. 
We define four annuli $A_1$, $A_2$, $A_3$ and $A_4$ as follows: 
{\allowdisplaybreaks
\begin{align*}
A_1 = & \{(s,t,0)\in R^2 \hspace{.3em} | \hspace{.3em} s\in [0,1], t\in [0,1]\}, \\
A_2 = & \{(s,t,0)\in R^2 \hspace{.3em} | \hspace{.3em} s\in [0,1], t\in [-1,0]\}, \\
A_3 = & \{(s,0,t)\in R^2 \hspace{.3em} | \hspace{.3em} s\in [0,1], t\in [0,1]\}, \\
A_4 = & \{(s,0,t)\in R^2 \hspace{.3em} | \hspace{.3em} s\in [0,1], t\in [-1,0]\}. 
\end{align*}
}
Then $S=\tilde{S}\cup A_1\cup A_2\cup A_3 \cup A_4$ represents the homology class of the pair $(M_h\cup M_r, \partial(M_h\cup M_r))$ 
after giving suitable orientations to the annuli $A_1$, $A_2$, $A_3$ and $A_4$. 
We denote this class by $\mathcal{S}$, 
then the intersection number $\mathcal{S}\cdot [F_h]$ is equal to $2g+2$, especially is non-zero. 

\vspace{.5em}

\noindent
{\bf Case 2}: If $H_1$ preserves the orientation of $c$ but does not preserve two points $v_1$ and $v_2$, 
then $R^2$ is untwisted and $\tilde{S}\cap R^2=\{(s,\pm\exp{(-\pi\sqrt{-1}s)}, 0)\in R^2 \hspace{.3em} | \hspace{.3em} s\in[0,1]\}$ 
is a circle. 
We define three annuli $A_5$, $A_6$ and $A_7$ as follows: 
{\allowdisplaybreaks
\begin{align*}
A_5 = & \{(s,t\exp{(-\pi\sqrt{-1}s)},0)\in R^2 \hspace{.3em} | \hspace{.3em} s\in [0,1], t\in [0,1]\} \\ 
& \cup \{(s,-t\exp{(-\pi\sqrt{-1}s)},0)\in R^2 \hspace{.3em} | \hspace{.3em} s\in [0,1], t\in [0,1]\}, \\
A_6 = & \{(s,0,t)\in R^2 \hspace{.3em} | \hspace{.3em} s\in [0,1], t\in [0,1]\}, \\
A_7 = & \{(s,0,t)\in R^2 \hspace{.3em} | \hspace{.3em} s\in [0,1], t\in [-1,0]\}. 
\end{align*}
}
Then $S=\tilde{S}\cup A_5\cup A_6\cup A_7$ represents the homology class of the pair $(M_h\cup M_r, \partial(M_h\cup M_r))$ 
after giving suitable orientations to the annuli $A_5$, $A_6$ and $A_7$. 
We denote this class by $\mathcal{S}$, 
then the intersection number $\mathcal{S}\cdot [F_h]$ is equal to $2g+2$, especially is non-zero. 

\vspace{.5em}

\noindent
{\bf Case 3}: If $H_1$ does not preserve the orientation of $c$ but preserves two points $v_1$ and $v_2$, 
then $R^2$ is twisted and $\tilde{S}\cap R^2=\{(s,\pm 1, 0)\in R^2 \hspace{.3em} | \hspace{.3em} s\in[0,1]\}$ 
is a disjoint union of two circles. 
We define three annuli $A_8$, $A_9$ and $A_{10}$ as follows: 
{\allowdisplaybreaks
\begin{align*}
A_8 = & \{(s,t,0)\in R^2 \hspace{.3em} | \hspace{.3em} s\in [0,1], t\in [0,1]\}, \\
A_9 = & \{(s,t,0)\in R^2 \hspace{.3em} | \hspace{.3em} s\in [0,1], t\in [-1,0]\}, \\
A_{10} = & \{(s,0,t)\in R^2 \hspace{.3em} | \hspace{.3em} s\in [0,1], t\in [0,1]\} \\
& \cup \{(s,0,t)\in R^2 \hspace{.3em} | \hspace{.3em} s\in [0,1], t\in [-1,0]\}. 
\end{align*}
}
Then $S=\tilde{S}\cup A_8 \cup A_9\cup A_{10}$ represents the homology class of the pair $(M_h\cup M_r, \partial(M_h\cup M_r))$ 
after giving suitable orientations to the annuli $A_8$, $A_9$ and $A_{10}$. 
We denote this class by $\mathcal{S}$, 
then the intersection number $\mathcal{S}\cdot [F_h]$ is equal to $2g+2$, especially is non-zero. 

\vspace{.5em}

\noindent
{\bf Case 4}: If $H_1$ preserves neither the orientation of $c$ nor two points $v_1$ and $v_2$, 
then $R^2$ is twisted and $\tilde{S}\cap R^2=\{(s,\pm\exp{(-\pi\sqrt{-1}s)}, 0)\in R^2 \hspace{.3em} | \hspace{.3em} s\in[0,1]\}$ 
is a circle. 
We define two annuli $A_{11}$ and $A_{12}$ as follows: 
{\allowdisplaybreaks
\begin{align*}
A_{11} = & \{(s,t\exp{(-\pi\sqrt{-1}s)},0)\in R^2 \hspace{.3em} | \hspace{.3em} s\in [0,1], t\in [0,1]\} \\ 
& \cup \{(s,-t\exp{(-\pi\sqrt{-1}s)},0)\in R^2 \hspace{.3em} | \hspace{.3em} s\in [0,1], t\in [0,1]\}, \\
A_{12} = & \{(s,0,t)\in R^2 \hspace{.3em} | \hspace{.3em} s\in [0,1], t\in [0,1]\}, \\
& \cup\{(s,0,t)\in R^2 \hspace{.3em} | \hspace{.3em} s\in [0,1], t\in [-1,0]\}. 
\end{align*}
}
Then $S=\tilde{S}\cup A_{11}\cup A_{12}$ represents the homology class of the pair $(M_h\cup M_r, \partial(M_h\cup M_r))$ 
after giving suitable orientations to the annuli $A_{11}$ and $A_{12}$. 
We denote this class by $\mathcal{S}$, 
then the intersection number $\mathcal{S}\cdot [F_h]$ is equal to $2g+2$, especially is non-zero. 

Eventually, we can construct the element $\mathcal{S}$ satisfying the desired condition in any cases. 
So we have proved $[F_h]\neq 0$ in $H_2(M_h\cup M_r;\mathbb{Q})$. 

We are now ready to prove the statement (ii) in Theorem \ref{main1}. 
There exists the following exact sequence which is the part of the Meyer-Vietoris exact sequence: 
\[
H_2(S^1\times \Sigma_{g-1};\mathbb{Q})\xrightarrow{i_1\oplus i_2} H_2(M_h\cup M_r;\mathbb{Q})\oplus H_2(D^2\times \Sigma_{g-1};\mathbb{Q})\xrightarrow{j_1-j_2} H_2(M;\mathbb{Q}). 
\]
Suppose that $(j_1-j_2)([F_h],0)=[F_h]=0$. 
Then there exists an element $\mu\in H_2(S^1\times \Sigma_{g-1};\mathbb{Q})$ such that $(i_1\oplus i_2)(\mu)=([F_h],0)$. 
By a K\"{u}nneth formula, we obtain the following isomorphism: 
\[
H_2(S^1\times \Sigma_{g-1};\mathbb{Q})\cong H_2(\Sigma_{g-1};\mathbb{Q})\oplus \left(H_1(\Sigma_{g-1};\mathbb{Q})\otimes H_1(S^1;\mathbb{Q}) \right). 
\]
Since the map $i_2:H_2(S^1\times \Sigma_{g-1};\mathbb{Q})\rightarrow H_2(D^2\times \Sigma_{g-1};\mathbb{Q})\cong H_2(\Sigma_{g-1};\mathbb{Q})$ 
is regarded as the projection onto the first component via the above isomorphism, 
$\mu$ is an element in $H_1(\Sigma_{g-1};\mathbb{Q})\otimes H_1(S^1;\mathbb{Q})$. 
The involution $\omega$ acts on the component $H_2(\Sigma_{g-1};\mathbb{Q})$ trivially 
and on the component $H_1(\Sigma_{g-1};\mathbb{Q})\otimes H_1(S^1;\mathbb{Q})$ by multiplying $-1$. 
So we obtain: 
\[
\omega_\ast(\mu)=-\mu. 
\]
$i_1\circ \omega_\ast$ is equal to $\omega_\ast\circ i_1$ since $i_1$ is induced by the inclusion map. 
Thus, we obtain:   
{\allowdisplaybreaks
\begin{align*}
[F_h] & = \omega_\ast([F_h]) \\
& = \omega_\ast\circ i_1(\mu) \\
& = i_1 \circ \omega_\ast (\mu) \\
& = i_1 \circ (-\mu) = -[F_h]. \\
\end{align*}
}
This means that $2[F_h]=0$ in $H_2(M_h\cup M_r;\mathbb{Q})$. 
This contradicts $[F_h]\neq 0$.
Therefore, we obtain $[F_h]\neq 0$ in $H_2(M;\mathbb{Q})$ and this completes the proof of the statement. 

\end{proof}

\begin{rem}

By the argument similar to that in the proof of Theorem \ref{main1}, 
we can generalize Theorem \ref{main1} to directed BLFs as follows: 

\begin{thm}\label{main1-directed}

Let $f:M\rightarrow S^2$ be a hyperelliptic directed BLF. 
Suppose that the genus of every fiber component of $f$ is greater than or equal to $2$. 

\begin{enumerate}[(i)]

\item 
Let $s_1$ be the number of Lefschetz singularities of $f$ whose vanishing cycles are separating. 
We define $s_2$ as follows: 
\[
s_2 = \text{max}\{s\in\mathbb{N} \hspace{.3em} | \hspace{.3em} \text{$f^{-1}(x)$ has $s$ components. $x\in S^2$}\}. 
\]
Then there exists an involution 
\[
\omega: M\rightarrow M
\]
such that the fixed point set of $\omega$ is the union of  (possibly nonorientable) surfaces and $s_1$ isolated points. 
Moreover, $\omega$ can be extended to an involution 
\[
\overline{\omega}:M\sharp s_1\overline{\mathbb{CP}^2}\rightarrow M\sharp s_1\overline{\mathbb{CP}^2}
\]
such that $M\sharp s_1\overline{\mathbb{CP}^2}/\overline{\omega}$ is diffeomorphic to $\sharp s_2S\sharp 2s_1\overline{\mathbb{CP}^2}$, 
where $S$ is $S^2$-bundle over $S^2$,  
and the quotient map 
\[
/\overline{\omega}:M\sharp s_1\overline{\mathbb{CP}^2}\rightarrow M\sharp s_1\overline{\mathbb{CP}^2}/\overline{\omega}\cong \sharp s_2S\sharp 2s_1\overline{\mathbb{CP}^2}
\]
is the double branched covering. 

\item Let $F\in M$ be a regular fiber of $f$. 
Then $F$ represents a non-trivial rational homology class of $M$. 

\end{enumerate}

\end{thm}

\end{rem}


\section{A genus-$2$ SBLF structure on $\sharp n \mathbb{CP}^2$ for $n\geq 0$}\label{section:genus-2}

In this section, we prove Theorem \ref{main2}. 
Let $\tilde{c}_{1,1}, \tilde{c}_{1,2}, \tilde{c}_2, \tilde{c}_3, \tilde{c}_4, \tilde{c}_5\subset \Sigma_{2,1}$ be simple closed curves described as in Figure \ref{scc21}.

\begin{figure}[htbp]
\begin{center}
\includegraphics[width=50mm]{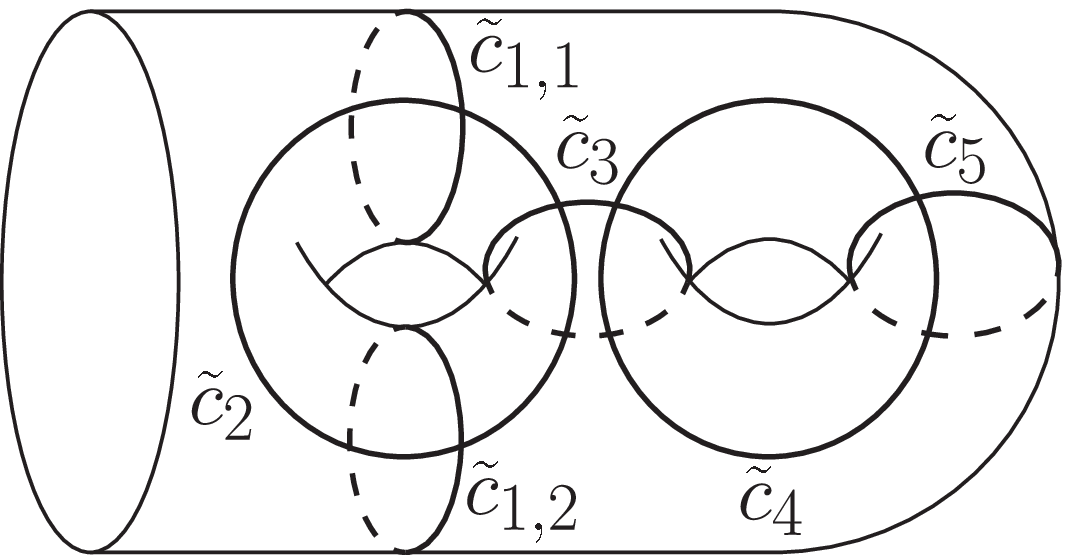}
\end{center}
\caption{}
\label{scc21}
\end{figure} 

\noindent
We denote by $\tilde{t}_{1,1},\tilde{t}_{1,2},\tilde{t}_2,\tilde{t}_3,\tilde{t}_4,\tilde{t}_5\in\mathcal{M}_{2,1}$ 
right-handed Dehn twists along the simple closed curves $\tilde{c}_{1,1}, \tilde{c}_{1,2}, \tilde{c}_2, \tilde{c}_3, \tilde{c}_4, \tilde{c}_5$, respectively. 
We define the simple closed curves $\alpha_i, \beta_j, \gamma_k\subset \Sigma_{2,1}$ as follows: 
\begin{align*}
\tilde{\alpha}_i & = ({\tilde{t}_5}^2\cdot {\tilde{t}_4}^{-i+1})(\tilde{c}_4) & \text{($i\in\mathbb{Z}$)}, \\
\tilde{\beta}_j & = {\tilde{t}_4}^{-j}(\tilde{c}_5) & \text{($j\in\mathbb{Z}$)}, \\
\tilde{\gamma}_1 & = (\tilde{t}_2\cdot \tilde{t}_3 \cdot \tilde{t}_4 \cdot {\tilde{t}_5}^2)(\tilde{c}_{1,1}), \\
\tilde{\gamma}_2 & = (\tilde{t}_2 \cdot \tilde{t}_3 \cdot \tilde{t}_4)(\tilde{c}_{1,1}).
\end{align*}
These curves are described as in Figure \ref{scc21-2}. 

\begin{figure}[htbp]
\begin{center}
\includegraphics[width=115mm]{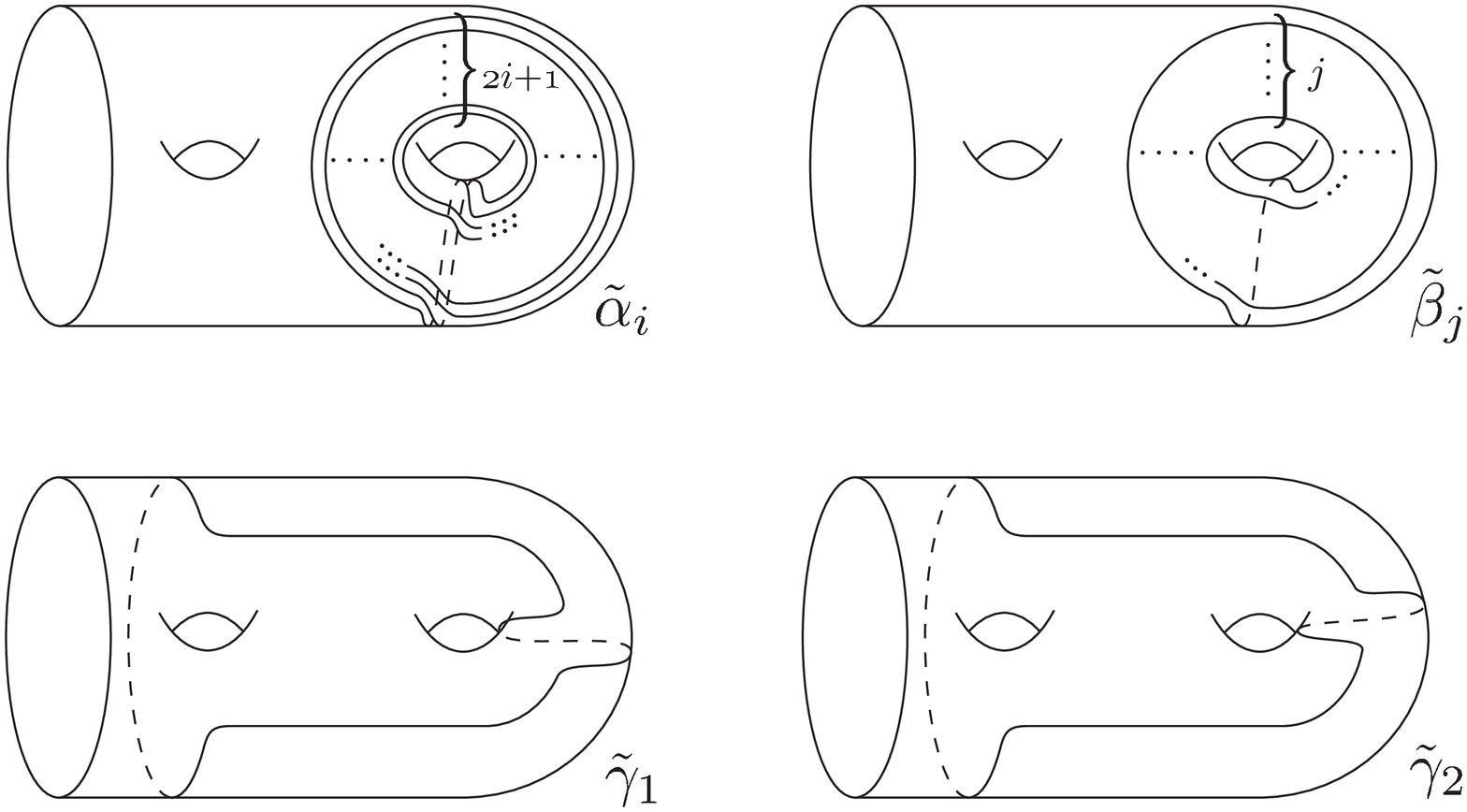}
\end{center}
\caption{}
\label{scc21-2}
\end{figure} 

\noindent
We also define elements $\tilde{\xi}, \tilde{\iota}_2\in\mathcal{M}_{2,1}$ as follows: 
\begin{align*}
\tilde{\xi} & =(\tilde{t}_4 \cdot \tilde{t}_5)^3, \\
\tilde{\iota}_2 & = (\tilde{t}_5 \cdot \tilde{t}_4 \cdot \tilde{t}_3 \cdot \tilde{t}_2)^5 \\
& = \tilde{t}_{1,1} \cdot \tilde{t}_2 \cdot \tilde{t}_3 \cdot \tilde{t}_4 \cdot {\tilde{t}_5}^2 \cdot \tilde{t}_4 \cdot \tilde{t}_3 \cdot \tilde{t}_2 \cdot \tilde{t}_{1,1}. 
\end{align*}

\noindent
We remark that the element $\tilde{\iota}_2$ is a lift of the hyperelliptic involution $\iota_2\in\mathcal{M}_2$
and that the element $\tilde{\xi}$ is a lift of the element $\xi\in\mathcal{M}_2$ which is represented by the element described as in Figure \ref{rep-xi}. 

\begin{figure}[htbp]
\begin{center}
\includegraphics[width=115mm]{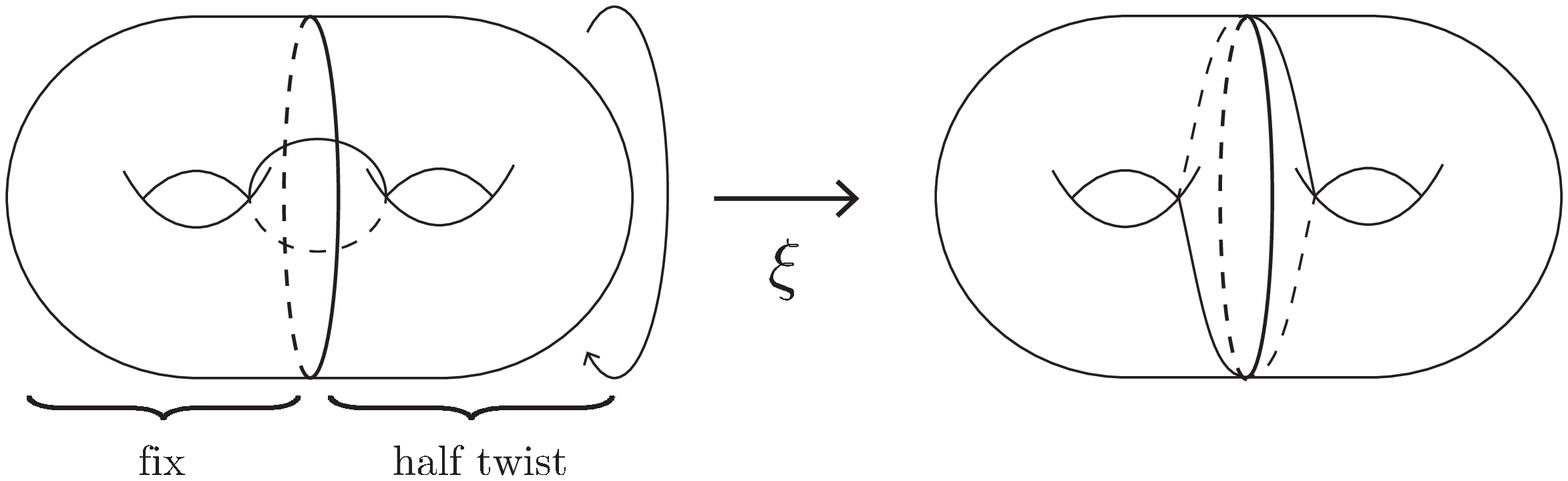}
\end{center}
\caption{}
\label{rep-xi}
\end{figure}

\noindent
We can easily obtain the following formulas: 
\begin{align*}
\tilde{\iota}_2 \cdot \tilde{t}_j & = \tilde{t}_j \cdot \tilde{\iota}_2 & \text{($j=2,3,4,5$)} ,\\
\tilde{\iota}_2 \cdot \tilde{t}_{1,1}  & = \tilde{t}_{1,2} \cdot \tilde{\iota}_2 .
\end{align*}

\begin{lem}\label{lem-gamma}

$t_{\tilde{\gamma}_1}\cdot t_{\tilde{\gamma}_2}= \tilde{\iota}_2 \cdot {\tilde{t}_5}^{-2}\cdot {\tilde{t}_{1,1}}^{-1}\cdot {\tilde{t}_{2}}^{-1}\cdot 
{\tilde{t}_{3}}^{-2}\cdot {\tilde{t}_{2}}^{-1}\cdot {\tilde{t}_{1,1}}^{-1} \in \Ker{(\Phi_{\tilde{c}_5}:\mathcal{M}_{2,1}(\tilde{c}_5)\rightarrow \mathcal{M}_{1,1})}$, 

\noindent
where $\mathcal{M}_{2,1}(\tilde{c}_5)$ is the subgroup of $\mathcal{M}_{2,1}$ which is defined as follows: 
\[
\mathcal{M}_{2,1}(\tilde{c}_5)=\{[T]\in\mathcal{M}_{2,1}|T(\tilde{c}_5)=\tilde{c}_5\}. 
\]

\end{lem}

\begin{proof}

We first prove the equation by direct calculation. 
By the definitions of $\tilde{\gamma}_1$ and $\tilde{\gamma}_2$, 
we obtain: 
\begin{align*}
t_{\tilde{\gamma}_1} & = (\tilde{t}_2\cdot \tilde{t}_3 \cdot \tilde{t}_4 \cdot {\tilde{t}_5}^2)^{-1} \cdot \tilde{t}_{1,1} \cdot 
 (\tilde{t}_2\cdot \tilde{t}_3 \cdot \tilde{t}_4 \cdot {\tilde{t}_5}^2) ,\\
t_{\tilde{\gamma}_2} & = (\tilde{t}_2\cdot \tilde{t}_3 \cdot \tilde{t}_4)^{-1} \cdot \tilde{t}_{1,1} \cdot 
 (\tilde{t}_2\cdot \tilde{t}_3 \cdot \tilde{t}_4) .\\
\end{align*}
So we can calculate $t_{\tilde{\gamma}_1}\cdot t_{\tilde{\gamma}_2}$ as follows: 
{\allowdisplaybreaks
\begin{align*}
t_{\tilde{\gamma}_1}\cdot t_{\tilde{\gamma}_2} & = (\tilde{t}_2\cdot \tilde{t}_3 \cdot \tilde{t}_4 \cdot {\tilde{t}_5}^2)^{-1} \cdot \tilde{t}_{1,1} \cdot 
(\tilde{t}_2\cdot \tilde{t}_3 \cdot \tilde{t}_4 \cdot {\tilde{t}_5}^2) \cdot
(\tilde{t}_2\cdot \tilde{t}_3 \cdot \tilde{t}_4)^{-1} \cdot \tilde{t}_{1,1} \cdot 
(\tilde{t}_2\cdot \tilde{t}_3 \cdot \tilde{t}_4) \\
& = {\tilde{t}_5}^{-2} \cdot {\tilde{t}_4}^{-1} \cdot {\tilde{t}_3}^{-1} \cdot {\tilde{t}_2}^{-1} \cdot (\tilde{t}_{1,1} \cdot 
\tilde{t}_2\cdot \tilde{t}_3 \cdot \tilde{t}_4 \cdot {\tilde{t}_5}^2) \cdot
{\tilde{t}_4}^{-1} \cdot {\tilde{t}_3}^{-1} \cdot {\tilde{t}_2}^{-1} \cdot \tilde{t}_{1,1} \cdot 
\tilde{t}_2\cdot \tilde{t}_3 \cdot \tilde{t}_4 \\
& = {\tilde{t}_5}^{-2} \cdot {\tilde{t}_4}^{-1} \cdot {\tilde{t}_3}^{-1} \cdot {\tilde{t}_2}^{-1} \cdot 
(\tilde{\iota}_2 \cdot {\tilde{t}_{1,1}}^{-1} \cdot {\tilde{t}_2}^{-1} \cdot {\tilde{t}_3}^{-1} \cdot {\tilde{t}_4}^{-1}) \cdot
{\tilde{t}_4}^{-1} \cdot {\tilde{t}_3}^{-1} \cdot {\tilde{t}_2}^{-1} \cdot \tilde{t}_{1,1} \cdot 
\tilde{t}_2\cdot \tilde{t}_3 \cdot \tilde{t}_4 \\
& = \tilde{\iota}_2 \cdot {\tilde{t}_5}^{-2} \cdot {\tilde{t}_4}^{-1} \cdot {\tilde{t}_3}^{-1} \cdot ({\tilde{t}_2}^{-1} \cdot 
{\tilde{t}_{1,1}}^{-1} \cdot {\tilde{t}_2}^{-1}) \cdot {\tilde{t}_3}^{-1} \cdot {\tilde{t}_4}^{-2} 
\cdot {\tilde{t}_3}^{-1} \cdot ({\tilde{t}_2}^{-1} \cdot \tilde{t}_{1,1} \cdot 
\tilde{t}_2) \cdot \tilde{t}_3 \cdot \tilde{t}_4 \\
& = \tilde{\iota}_2 \cdot {\tilde{t}_5}^{-2} \cdot {\tilde{t}_4}^{-1} \cdot {\tilde{t}_3}^{-1} \cdot ({\tilde{t}_{1,1}}^{-1} \cdot 
{\tilde{t}_2}^{-1} \cdot {\tilde{t}_{1,1}}^{-1}) \cdot {\tilde{t}_3}^{-1} \cdot {\tilde{t}_4}^{-2} 
\cdot {\tilde{t}_3}^{-1} \cdot (\tilde{t}_{1,1} \cdot 
\tilde{t}_2 \cdot {\tilde{t}_{1,1}}^{-1}) \cdot \tilde{t}_3 \cdot \tilde{t}_4 \\
& = \tilde{\iota}_2 \cdot {\tilde{t}_5}^{-2} \cdot {\tilde{t}_{1,1}}^{-1} \cdot {\tilde{t}_4}^{-1} \cdot ({\tilde{t}_3}^{-1} \cdot  
{\tilde{t}_2}^{-1} \cdot {\tilde{t}_3}^{-1}) \cdot {\tilde{t}_4}^{-2} 
\cdot ({\tilde{t}_3}^{-1} \cdot \tilde{t}_2 \cdot \tilde{t}_3) \cdot \tilde{t}_4 \cdot {\tilde{t}_{1,1}}^{-1} \\
& = \tilde{\iota}_2 \cdot {\tilde{t}_5}^{-2} \cdot {\tilde{t}_{1,1}}^{-1} \cdot {\tilde{t}_4}^{-1} \cdot ({\tilde{t}_2}^{-1} \cdot  
{\tilde{t}_3}^{-1} \cdot {\tilde{t}_2}^{-1}) \cdot {\tilde{t}_4}^{-2} 
\cdot (\tilde{t}_2 \cdot \tilde{t}_3 \cdot {\tilde{t}_2}^{-1}) \cdot \tilde{t}_4 \cdot {\tilde{t}_{1,1}}^{-1} \\
& = \tilde{\iota}_2 \cdot {\tilde{t}_5}^{-2} \cdot {\tilde{t}_{1,1}}^{-1} \cdot {\tilde{t}_2}^{-1} \cdot ({\tilde{t}_4}^{-1} \cdot  
{\tilde{t}_3}^{-1} \cdot {\tilde{t}_4}^{-1}) \cdot ({\tilde{t}_4}^{-1}
\cdot \tilde{t}_3 \cdot \tilde{t}_4) \cdot {\tilde{t}_2}^{-1} \cdot {\tilde{t}_{1,1}}^{-1} \\
& = \tilde{\iota}_2 \cdot {\tilde{t}_5}^{-2} \cdot {\tilde{t}_{1,1}}^{-1} \cdot {\tilde{t}_2}^{-1} \cdot ({\tilde{t}_3}^{-1} \cdot  
{\tilde{t}_4}^{-1} \cdot {\tilde{t}_3}^{-1}) \cdot (\tilde{t}_3 \cdot \tilde{t}_4 \cdot {\tilde{t}_3}^{-1}) \cdot {\tilde{t}_2}^{-1} \cdot {\tilde{t}_{1,1}}^{-1} \\
& = \tilde{\iota}_2 \cdot {\tilde{t}_5}^{-2} \cdot {\tilde{t}_{1,1}}^{-1} \cdot {\tilde{t}_2}^{-1} \cdot {\tilde{t}_3}^{-1} \cdot {\tilde{t}_3}^{-1} \cdot {\tilde{t}_2}^{-1} \cdot {\tilde{t}_{1,1}}^{-1}. 
\end{align*}
}

We next prove that the element $\tilde{\iota}_2 \cdot {\tilde{t}_5}^{-2} \cdot {\tilde{t}_{1,1}}^{-1} 
\cdot {\tilde{t}_2}^{-1} \cdot {\tilde{t}_3}^{-1} \cdot {\tilde{t}_3}^{-1} \cdot {\tilde{t}_2}^{-1} \cdot {\tilde{t}_{1,1}}^{-1}$
is contained in the kernel of $\Phi_{\tilde{c}_5}:\mathcal{M}_{2,1}(\tilde{c}_5) \rightarrow \mathcal{M}_{1,1}$. 
The elements $\tilde{\iota}_2, \tilde{t}_5$ and ${\tilde{t}_{1,1}}^{-1} 
\cdot {\tilde{t}_2}^{-1} \cdot {\tilde{t}_3}^{-1} \cdot {\tilde{t}_3}^{-1} \cdot {\tilde{t}_2}^{-1} \cdot {\tilde{t}_{1,1}}^{-1}$
are contained in the group $\mathcal{M}_{2,1}(\tilde{c}_5)$. 
It is obvious that $\Phi_{\tilde{c}_5}(\tilde{t}_5)=1$. 
We can calculate the product $\tilde{t}_4 \cdot {\tilde{t}_5}^2 \cdot \tilde{t}_4$ as follows: 
{\allowdisplaybreaks 
\begin{align}\label{4554}
\begin{split}
\tilde{t}_4 \cdot {\tilde{t}_5}^2 \cdot \tilde{t}_4 & ={\tilde{t}_5}^{-1} \cdot (\tilde{t}_5 \cdot \tilde{t}_4 \cdot \tilde{t}_5) \cdot 
(\tilde{t}_5 \cdot \tilde{t}_4 \cdot \tilde{t}_5) \cdot {\tilde{t}_5}^{-1} \\
& = {\tilde{t}_5}^{-1} \cdot (\tilde{t}_4 \cdot \tilde{t}_5 \cdot \tilde{t}_4) \cdot 
(\tilde{t}_5 \cdot \tilde{t}_4 \cdot \tilde{t}_5) \cdot {\tilde{t}_5}^{-1} \\
& = {\tilde{t}_5}^{-1} \cdot \tilde{\xi} \cdot {\tilde{t}_5}^{-1} 
\end{split}
\end{align}
}
Since $\tilde{\xi}\in \mathcal{M}_{2,1}(\tilde{c}_5)$ and $\Phi_{\tilde{c}_5}(\tilde{\xi})=1$, we obtain: 
{\allowdisplaybreaks
\begin{align*}
& \Phi_{\tilde{c}_5}(\tilde{\iota}_2 \cdot {\tilde{t}_5}^{-2} \cdot {\tilde{t}_{1,1}}^{-1} 
\cdot {\tilde{t}_2}^{-1} \cdot {\tilde{t}_3}^{-1} \cdot {\tilde{t}_3}^{-1} \cdot {\tilde{t}_2}^{-1} \cdot {\tilde{t}_{1,1}}^{-1}) \\
= & \Phi_{\tilde{c}_5}(\tilde{\iota}_2) \cdot \Phi_{\tilde{c}_5}({\tilde{t}_5}^{-2}) \cdot \Phi_{\tilde{c}_5}({\tilde{t}_{1,1}}^{-1} 
\cdot {\tilde{t}_2}^{-1} \cdot {\tilde{t}_3}^{-1} \cdot {\tilde{t}_3}^{-1} \cdot {\tilde{t}_2}^{-1} \cdot {\tilde{t}_{1,1}}^{-1}) \\
= & \Phi_{\tilde{c}_5}(\tilde{t}_{1,1} \cdot \tilde{t}_2 \cdot \tilde{t}_3 \cdot \tilde{t}_4 \cdot {\tilde{t}_5}^2 \cdot \tilde{t}_4 \cdot \tilde{t}_3 \cdot \tilde{t}_2 \cdot \tilde{t}_{1,1}) 
\cdot \Phi_{\tilde{c}_5}({\tilde{t}_{1,1}}^{-1} 
\cdot {\tilde{t}_2}^{-1} \cdot {\tilde{t}_3}^{-1} \cdot {\tilde{t}_3}^{-1} \cdot {\tilde{t}_2}^{-1} \cdot {\tilde{t}_{1,1}}^{-1}) \\
= & \Phi_{\tilde{c}_5}(\tilde{t}_{1,1} \cdot \tilde{t}_2 \cdot \tilde{t}_3) \cdot \Phi_{\tilde{c}_5}(\tilde{t}_4 \cdot {\tilde{t}_5}^2 \cdot \tilde{t}_4) \cdot 
\Phi_{\tilde{c}_5}(\tilde{t}_3 \cdot \tilde{t}_2 \cdot \tilde{t}_{1,1}) 
\cdot \Phi_{\tilde{c}_5}({\tilde{t}_{1,1}}^{-1} 
\cdot {\tilde{t}_2}^{-1} \cdot {\tilde{t}_3}^{-1} \cdot {\tilde{t}_3}^{-1} \cdot {\tilde{t}_2}^{-1} \cdot {\tilde{t}_{1,1}}^{-1}) \\
= & \Phi_{\tilde{c}_5}(\tilde{t}_{1,1} \cdot \tilde{t}_2 \cdot \tilde{t}_3) \cdot  
\Phi_{\tilde{c}_5}(\tilde{t}_3 \cdot \tilde{t}_2 \cdot \tilde{t}_{1,1}) 
\cdot \Phi_{\tilde{c}_5}({\tilde{t}_{1,1}}^{-1} 
\cdot {\tilde{t}_2}^{-1} \cdot {\tilde{t}_3}^{-1} \cdot {\tilde{t}_3}^{-1} \cdot {\tilde{t}_2}^{-1} \cdot {\tilde{t}_{1,1}}^{-1}) \\
= & 1. 
\end{align*}
}

This completes the proof of Lemma\ref{lem-gamma}.  

\end{proof}

\begin{lem}\label{lem-beta}

$t_{\tilde{\beta}_1} \cdot t_{\tilde{\beta}_{-1}} = \tilde{\xi} \cdot {\tilde{t}_5}^{-4} \in \Ker{\Phi_{\tilde{c}_5}}$. 

\end{lem}

\begin{proof}

By the definitions of the curves $\tilde{\beta}_1,\tilde{\beta}_2$, we obtain: 
\begin{align*}
t_{\tilde{\beta}_1} & = \tilde{t}_4 \cdot \tilde{t}_5 \cdot {\tilde{t}_4}^{-1} , \\
t_{\tilde{\beta}_{-1}} & = {\tilde{t}_4}^{-1} \cdot \tilde{t}_5 \cdot \tilde{t}_4 .
\end{align*}
So we can calculate $t_{\tilde{\beta}_1} \cdot t_{\tilde{\beta}_{-1}}$ as follows: 
{\allowdisplaybreaks
\begin{align*}
t_{\tilde{\beta}_1} \cdot t_{\tilde{\beta}_{-1}} & = (\tilde{t}_4 \cdot \tilde{t}_5 \cdot {\tilde{t}_4}^{-1}) \cdot ({\tilde{t}_4}^{-1} \cdot \tilde{t}_5 \cdot \tilde{t}_4) \\
& = {\tilde{t}_5}^{-1} \cdot \tilde{t}_4 \cdot \tilde{t}_5 \cdot \tilde{t}_5 \cdot \tilde{t}_4 \cdot {\tilde{t}_5}^{-1} \\
& = {\tilde{t}_5}^{-2} \cdot \tilde{\xi} \cdot {\tilde{t}_5}^{-2}  \hspace{1em}\text{(by (\ref{4554}))} \\
& = \tilde{\xi} \cdot {\tilde{t}_5}^{-4}. 
\end{align*}
}
Since $\Phi_{\tilde{c}_5}(\tilde{t}_5)=1$ and $\Phi_{\tilde{c}_5}(\tilde{\xi})=1$, 
the element $\tilde{\xi} \cdot {\tilde{t}_5}^{-4}$ is contained in the kernel of $\Phi_{\tilde{c}_5}$. 

\end{proof}

\begin{lem}\label{lem-alphabeta}

For $s\geq 3$, 
\[
t_{\tilde{\alpha}_1} \cdot t_{\tilde{\alpha}_2} \cdot \cdots \cdot t_{\tilde{\alpha}_{s-2}} \cdot t_{\tilde{\beta}_{s-1}} \cdot t_{\tilde{\beta}_{-1}} 
= {\tilde{\xi}}^{s-1} \cdot {\tilde{t}_5}^{-5s+6} \in \Ker{\Phi_{\tilde{c}_5}}. 
\]

\end{lem}

\begin{proof}

We denote by $P_s\in \mathcal{M}_{2,1}$ the left side of the equation in the statement of Lemma \ref{lem-alphabeta}. 
Then the following equations hold: 
\begin{align*}
P_3 & = t_{\tilde{\beta}_1} \cdot t_{\tilde{\beta}_{-1}} \cdot ({t_{\tilde{\beta}_{-1}}}^{-1} \cdot {t_{\tilde{\beta}_1}}^{-1} \cdot 
t_{\tilde{\alpha}_1} \cdot t_{\tilde{\beta}_2} \cdot t_{\tilde{\beta}_{-1}}), \\
P_s & = P_{s-1} \cdot ({t_{\tilde{\beta}_{-1}}}^{-1} \cdot {t_{\tilde{\beta}_{s-2}}}^{-1} \cdot 
t_{\tilde{\alpha}_{s-2}} \cdot t_{\tilde{\beta}_{s-1}} \cdot t_{\tilde{\beta}_{-1}}) \hspace{1em}\text{($s\geq 4$)}. 
\end{align*}
So, by Lemma \ref{lem-beta}, it is sufficient to prove the following equation: 
\[
{t_{\tilde{\beta}_{-1}}}^{-1} \cdot {t_{\tilde{\beta}_{s-2}}}^{-1} \cdot 
t_{\tilde{\alpha}_{s-2}} \cdot t_{\tilde{\beta}_{s-1}} \cdot t_{\tilde{\beta}_{-1}}=\tilde{\xi} \cdot {\tilde{t}_5}^{-5} \hspace{1em} \text{($s \geq 3$)}. 
\]
We prove this equation by direct calculation. 
By the definitions of the curves $\alpha_i$ and $\beta_j$, we obtain: 
\begin{align*}
t_{\tilde{\alpha}_i} & = ({\tilde{t}_5}^{2} \cdot {\tilde{t}_4}^{-i+1})^{-1} \cdot \tilde{t}_4 \cdot {\tilde{t}_5}^{2} \cdot {\tilde{t}_4}^{-i+1} ,\\
t_{\tilde{\beta}_j} & = {\tilde{t}_4}^j \cdot \tilde{t}_5 \cdot {\tilde{t}_4}^{-j} . 
\end{align*}
So we can calculate ${t_{\tilde{\beta}_{-1}}}^{-1} \cdot {t_{\tilde{\beta}_{s-2}}}^{-1} \cdot 
t_{\tilde{\alpha}_{s-2}} \cdot t_{\tilde{\beta}_{s-1}} \cdot t_{\tilde{\beta}_{-1}}$ as follows: 
{\allowdisplaybreaks
\begin{align*}
& {t_{\tilde{\beta}_{-1}}}^{-1} \cdot {t_{\tilde{\beta}_{s-2}}}^{-1} \cdot 
t_{\tilde{\alpha}_{s-2}} \cdot t_{\tilde{\beta}_{s-1}} \cdot t_{\tilde{\beta}_{-1}} \\
= &  ({\tilde{t}_4}^{-1} \cdot {\tilde{t}_5}^{-1} \cdot \tilde{t}_4) \cdot ({\tilde{t}_4}^{s-2} \cdot {\tilde{t}_5}^{-1} \cdot {\tilde{t}_4}^{-s+2}) 
\cdot  (({\tilde{t}_5}^{2} \cdot {\tilde{t}_4}^{-s+3})^{-1} \cdot \tilde{t}_4 \cdot {\tilde{t}_5}^{2} \cdot {\tilde{t}_4}^{-s+3}) 
\cdot ({\tilde{t}_4}^{s-1} \cdot \tilde{t}_5 \cdot {\tilde{t}_4}^{-s+1}) \cdot ({\tilde{t}_4}^{-1} \cdot \tilde{t}_5 \cdot \tilde{t}_4) \\
= &  {\tilde{t}_4}^{-1} \cdot {\tilde{t}_5}^{-1} \cdot {\tilde{t}_4}^{s-1} \cdot {\tilde{t}_5}^{-1} \cdot {\tilde{t}_4}^{-s+2} 
\cdot  {\tilde{t}_4}^{s-3} \cdot {\tilde{t}_5}^{-2} \cdot \tilde{t}_4 \cdot {\tilde{t}_5}^{2} \cdot {\tilde{t}_4}^{2} 
\cdot \tilde{t}_5 \cdot {\tilde{t}_4}^{-s} \cdot \tilde{t}_5 \cdot \tilde{t}_4 \\
= &  ({\tilde{t}_4}^{-1} \cdot {\tilde{t}_5}^{-1} \cdot {\tilde{t}_4}^{s-1}) \cdot {\tilde{t}_5}^{-1} \cdot {\tilde{t}_4}^{-1} 
\cdot {\tilde{t}_5}^{-2} \cdot \tilde{t}_4 \cdot {\tilde{t}_5}^{2} \cdot {\tilde{t}_4}^{2} 
\cdot \tilde{t}_5 \cdot ({\tilde{t}_4}^{-s} \cdot \tilde{t}_5 \cdot \tilde{t}_4) \\
= &  ({\tilde{t}_5}^{s-1} \cdot {\tilde{t}_4}^{-1} \cdot {\tilde{t}_5}^{-1}) \cdot {\tilde{t}_5}^{-1} \cdot {\tilde{t}_4}^{-1} 
\cdot {\tilde{t}_5}^{-2} \cdot \tilde{t}_4 \cdot {\tilde{t}_5}^{2} \cdot {\tilde{t}_4}^{2} 
\cdot \tilde{t}_5 \cdot (\tilde{t}_5 \cdot \tilde{t}_4 \cdot {\tilde{t}_5}^{-s}) \\
= &  {\tilde{t}_5}^{s-1} \cdot ({\tilde{t}_4}^{-1} \cdot {\tilde{t}_5}^{-2} \cdot {\tilde{t}_4}^{-1}) 
\cdot {\tilde{t}_5}^{-2} \cdot (\tilde{t}_4 \cdot {\tilde{t}_5}^{2} \cdot \tilde{t}_4) \cdot (\tilde{t}_4 
\cdot {\tilde{t}_5}^2 \cdot \tilde{t}_4) \cdot {\tilde{t}_5}^{-s} \\
= &  {\tilde{t}_5}^{s-1} \cdot ({\tilde{t}_5}^{2} \cdot \tilde{\xi}^{-1}) 
\cdot {\tilde{t}_5}^{-2} \cdot ({\tilde{t}_5}^{-2} \cdot \tilde{\xi}) \cdot ({\tilde{t}_5}^{-2} \cdot \tilde{\xi}) \cdot {\tilde{t}_5}^{-s} \\
= &  {\tilde{\xi}}^{s-1} \cdot {\tilde{t}_5}^{-5s+6} .
\end{align*}
}
This completes the proof of Lemma \ref{lem-alphabeta}. 

\end{proof}

For each $s\geq 2$, we define a sequence $W_s$ of elements of the mapping class group $\mathcal{M}_2$ as follows: 
\begin{align*}
W_s & = \begin{cases}
(t_{{\beta}_1}, t_{{\beta}_{-1}}, t_{{\gamma}_1}, t_{{\gamma}_2}) & \text{($s=2$)}, \\
(t_{{\alpha}_{1}}, t_{{\alpha}_{2}}, \ldots , t_{{\alpha}_{s-2}}, t_{{\beta}_{s-1}}, t_{{\beta}_{-1}}, t_{{\gamma}_1}, t_{{\gamma}_2}) & \text{($s\geq 3$)},  
\end{cases}
\end{align*}
where $\alpha_i,\beta_j,\gamma_k\subset \Sigma_2$ are the images of $\tilde{\alpha}_i, \tilde{\beta}_j, \tilde{\gamma}_k \subset \Sigma_{2,1}$ 
by the natural inclusion $\Sigma_{2,1}\hookrightarrow \Sigma_2$. 
By Lemma \ref{lem-gamma}, \ref{lem-beta} and \ref{lem-alphabeta}, 
there exists a genus-$2$ SBLF $f_s:M_s \rightarrow S^2$ which has a section with self-intersection $0$ and satisfies $W_{f_s}=W_s$.  

\begin{lem}\label{fundamental}

$\pi_1(M_s)\cong \mathbb{Z}$. 
Moreover, a generator of the group $\pi_1(M_s)$ is represented by the simple closed curve $c_1$ in the regular fiber. 

\end{lem}

\begin{proof}

Since $f_s$ has a section, 
it can be shown by Van Kampen's theorem that the group $\pi_1(M_s)$ is isomorphic to the fundamental group of the union of 
the higher side and the round cobordism of $f_s$. 
In particular, we obtain: 
\begin{align*}
\pi_1(M_s)\cong \begin{cases}
\pi_1(\Sigma_2)/<c_5, \beta_1, \beta_{-1}, \gamma_1, \gamma_2> & \text{($s=2$)}, \\
\pi_1(\Sigma_2)/<c_5, \alpha_1, \alpha_2, \ldots, \alpha_{s-2}, \beta_{s-1}, \beta_{-1}, \gamma_1, \gamma_2> & \text{($s\geq 3$)}. 
\end{cases}
\end{align*}
We identify the group $\pi_1(M_s)$ with the above quotient group via the isomorphism. 
The group $\pi_1(\Sigma_2)$ is generated by the elements 
$\lambda_i\in\pi_1(\Sigma_2)$ ($i=1,2,3,4$), where $\lambda_i$ is represented by the loop $\Lambda_i$ 
which is described as shown in Figure \ref{generator}. 

\begin{figure}[htbp]
\begin{center}
\includegraphics[width=50mm]{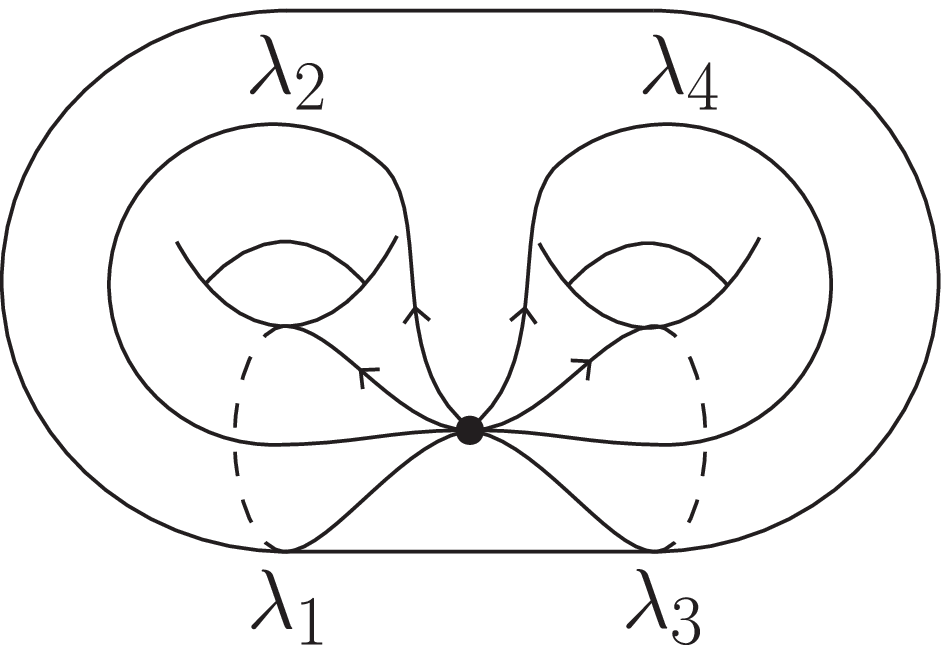}
\end{center}
\caption{}
\label{generator}
\end{figure}

Since $c_5,\beta_{-1}$ and $\gamma_1$ are free homotopic the loops 
$\lambda_3, \lambda_3\cdot{\lambda_4}^{-1}$ and ${\lambda_3}^{-1}\cdot{\lambda_4}^{-1}\cdot\lambda_2$, 
respectively. 
So the elements $\lambda_2, \lambda_3$ and $\lambda_4$ vanish in the group $\pi_1(M_s)$.  
The loops $\alpha_i, \beta_j$ and $\gamma_k$ are free homotopic to loops represented 
by words which consist of only $\lambda_2, \lambda_3 ,\lambda_4$ and their inverse. 
Thus, $\pi_1(M_s)$ is generated by the element $\lambda_1$. 
In particular, we obtain $\pi_1(M_s)\cong\mathbb{Z}$. 
The remaining statement holds since $c_1$ is free homotopic to the loop $\lambda_1$ 

\end{proof}

By Lemma \ref{fundamental}, we obtain the simply connected $4$-manifold $\tilde{M_s}$ 
from $M_s$ by the following construction: 

\begin{enumerate}[Step.1]

\item We first remove from $M_s$ the interior of the regular neighborhood of the regular fiber $F_0$ 
of $f_s$ in the lower side. 

\item The boundary of $M_s\setminus \text{Int}(\nu F_0)$ is the trivial torus bundle over the circle. 
In particular, we obtain, 
\[
\pi_1 (\partial (M_s\setminus \text{Int}(\nu F_0))) \cong \pi_1(S^1)\times \pi_1(T^2)\cong \mathbb{Z}\times\mathbb{Z}^2. 
\]
Moreover, a certain primitive element $\mu$ in the group $\pi_1(T^2)\cong \mathbb{Z}^2\in \pi_1(\partial (M_s\setminus \text{Int}(\nu F_0)))$ is mapped to 
the generator of $\pi_1(M_s\setminus \text{Int}(\nu F_0))\cong \mathbb{Z}$ by the natural homomorphism 
$\pi_1 (\partial (M_s\setminus \text{Int}(\nu F_0)))\rightarrow\pi_1 (M_s\setminus \text{Int}(\nu F_0))$. 
We can attach $\nu F_0$ to $M_s\setminus \text{Int}(\nu F_0)$ so that 
the attaching circle of the $2$-handle of $\nu F_0$ is along the simple closed curve 
which represents $([S^1],\mu)\in\pi_1 (\partial (M_s\setminus \text{Int}(\nu F_0)))$, 
where $[S^1]\in \pi_1(S^1)$ is the generator. 
We denote by $\tilde{M_s}$ the $4$-manifold obtained by the above attachment. 

\end{enumerate}

\noindent
In other word, we obtain the manifold $\tilde{M_s}$ from $M_s$ by logarithmic transformation on $F_0$ 
with multiplicity $1$. 
In particular, the BLF structure on $M_s\setminus \text{Int}(\nu F_0)$ is naturally extended 
to that on $\tilde{M_s}$. 
We denote this SBLF by $\tilde{f}_s: \tilde{M_s}\rightarrow S^2$.  
We also remark that Kirby diagrams of $M_s$ and $\tilde{M_s}$ can be drawn as shown in Figure \ref{kirby1} and \ref{kirby2}, respectively, 
by using the method in \cite{Ba}. 
The $2$-handles corresponding to the vanishing cycles $\alpha_1,\ldots,\alpha_{s-2}$ 
is in the shaded parts in Figure \ref{kirby1} and Figure \ref{kirby2}. 
These parts are empty if $s$ is equal to $2$.  

\par

The following theorem states that $\tilde{f}_s$ gives the explicit example of genus-$2$ 
SBLF structure on $\sharp (s-2)\mathbb{CP}^2$. 

\begin{thm}\label{diffeoMs}

For each $s\geq 2$, $\tilde{M_s}$ is diffeomorphic to the manifold $\sharp (s-2)\mathbb{CP}^2$. 

\end{thm}

\begin{proof}

We prove this theorem by Kirby calculus. 
A Kirby diagram of $\tilde{M_s}$ is shown in Figure \ref{kirby2}, 
where the framing of the $2$-handles drawn in the bold curves is $0$, 
the framing of the $2$-handle of $\nu F_0$ is described by the broken curve 
and the framings of the other $2$-handles are all $-1$.   
We obtain Figure \ref{kirby3} by sliding several $2$-handles to the $2$-handle of the round $2$-handle and 
isotopy moves give Figure \ref{kirby4}. 
We next slide the $2$-handles corresponding to the vanishing cycles $\alpha_1,\ldots,\alpha_{s-2}$ 
to the $2$-handle of the round $2$-handle. 
We can eliminate the obvious canceling pair and we obtain Figure \ref{kirby5}. 
Sliding the $2$-handle corresponding to $\gamma_2$ to that corresponding to $\gamma_1$ 
gives Figure \ref{kirby6}. 
We get Figure \ref{kirby7} by isotopy moves.  
By sliding the $2$-handle corresponding to $\gamma_1$ to that corresponding to $\beta_{-1}$, 
we obtain Figure \ref{kirby8}. 
Isotopy moves give Figure \ref{kirby9} and Figure \ref{kirby10} is obtained by 
sliding the $2$-handle corresponding to $\beta_{-1}$ to the $0$-framed meridian 
and isotopy moves. 
The diagram described in Figure \ref{kirby10} can be divided into two components. 
We look at the left component in Figure \ref{kirby10}. 
By sliding the outer $2$-handle to the $2$-handle of $\nu F_0$, 
we can change the left component in Figure \ref{kirby10} 
into the diagram as described in Figure \ref{kirby11}. 
Isotopy moves give Figure \ref{kirby12} and Figure \ref{kirby13} 
and we obtain the diagram in the left side of Figure \ref{kirby14} by canceling the obvious canceling pair.  
To obtain the diagram in the right side of Figure \ref{kirby14}, 
we slide the $-1$-framed $2$-handle to the $1$-framed $2$-handle 
and then eliminate the canceling pair. 
Since the handle decomposition of $\tilde{M_s}$ has three $3$-handles, 
the $0$-framed unknots in Figure \ref{kirby14} can be eliminated. 
Eventually, we can change the diagram in Figure \ref{kirby10} 
into the diagram in Figure \ref{kirby15}. 
The diagram in Figure \ref{kirby16} is same as in Figure \ref{kirby15}, 
but the $2$-handles in the shaded part are described in Figure \ref{kirby16}. 
By isotopy moves, we get Figure \ref{kirby17}. 
This diagram is similar to Figure 24 in \cite{H}. 
By using similar technique to \cite{H}, 
we can change the diagram in Figure \ref{kirby17} 
into $s-2$ unknots with $-1$ framing. 
This completes the proof of Theorem \ref{diffeoMs}. 

\end{proof}

As a result of Theorem \ref{diffeoMs}, Theorem \ref{main2} holds.


\vspace{1em}

\noindent
{\bf Acknowledgments. }
The authors would like to express their gratitude to Hisaaki Endo for his continuous support during the course of this work, and Nariya Kawazumi for his helpful comments for the draft of this paper. 
The first author is supported by Yoshida Scholarship 'Master21' and he is grateful to Yoshida Scholarship Foundation for their support. The second author is supported by JSPS Research Fellowships for Young Scientists (22-2364).

\newpage

\begin{figure}[htbp]
\begin{center}
\includegraphics[width=141mm]{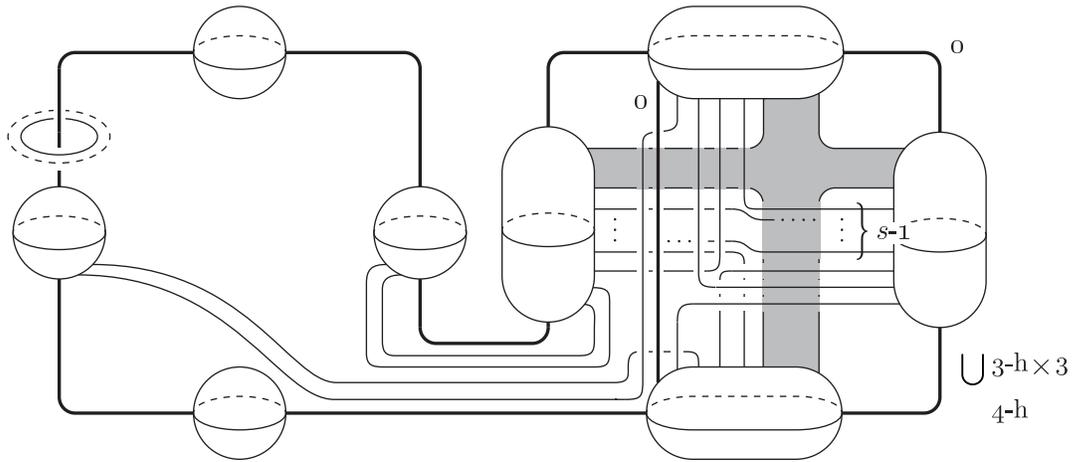}
\end{center}
\caption{a Kirby diagram of $M_s$}
\label{kirby1}
\end{figure}

\begin{figure}[htbp]
\begin{center}
\includegraphics[width=141mm]{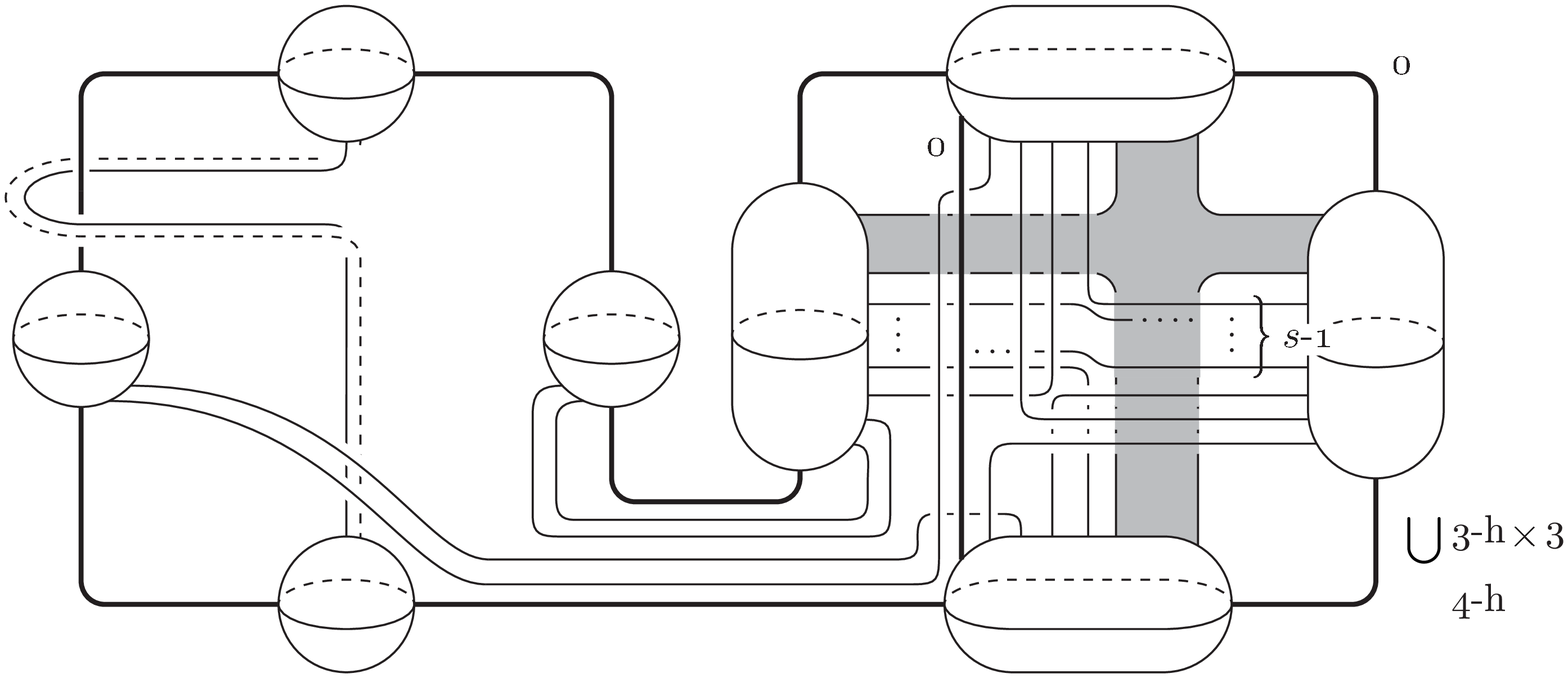}
\end{center}
\caption{a Kirby diagram of $\tilde{M_s}$}
\label{kirby2}
\end{figure}

\begin{figure}[htbp]
\begin{center}
\includegraphics[width=110mm]{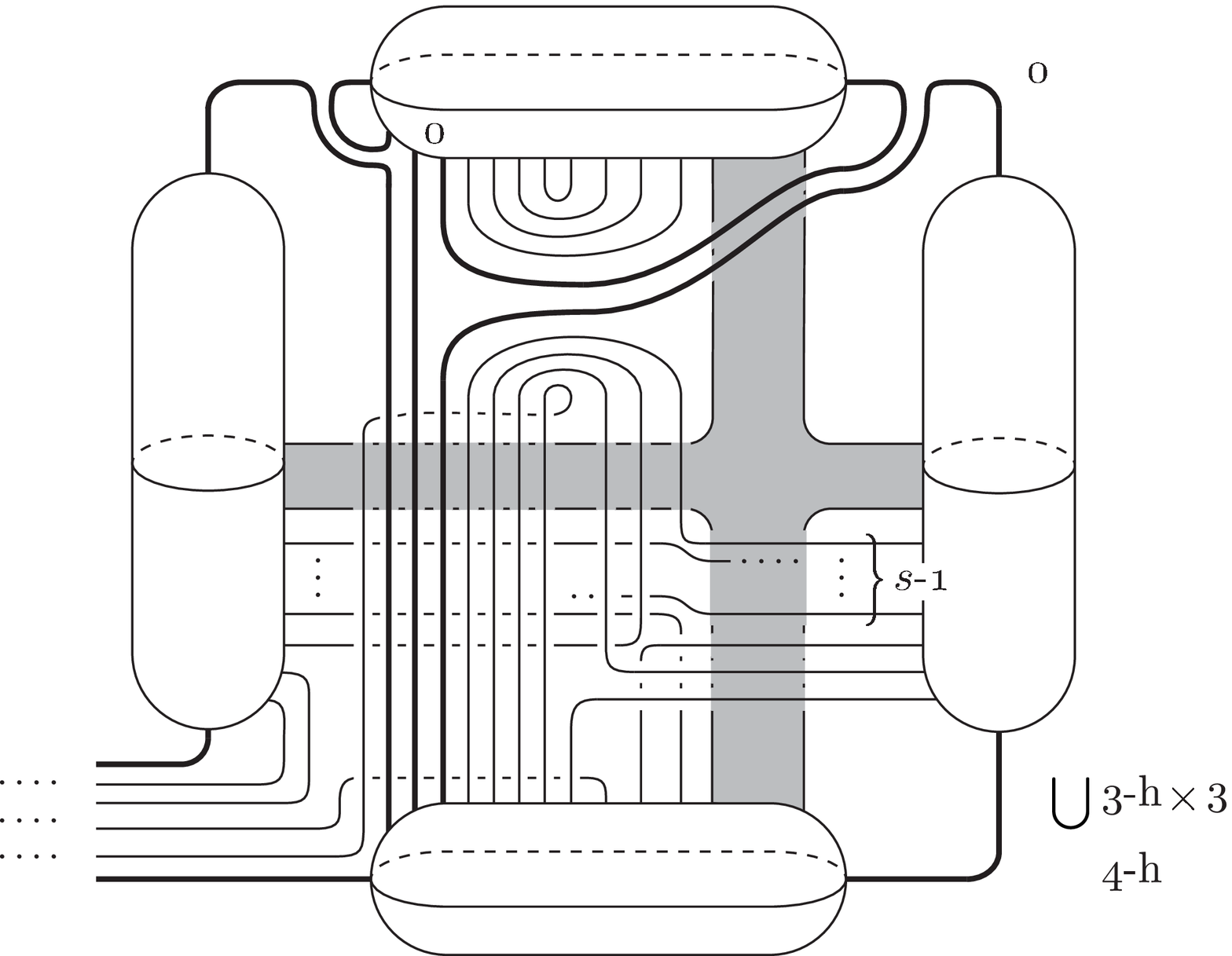}
\end{center}
\caption{}
\label{kirby3}
\end{figure}

\begin{figure}[htbp]
\begin{center}
\includegraphics[width=141mm]{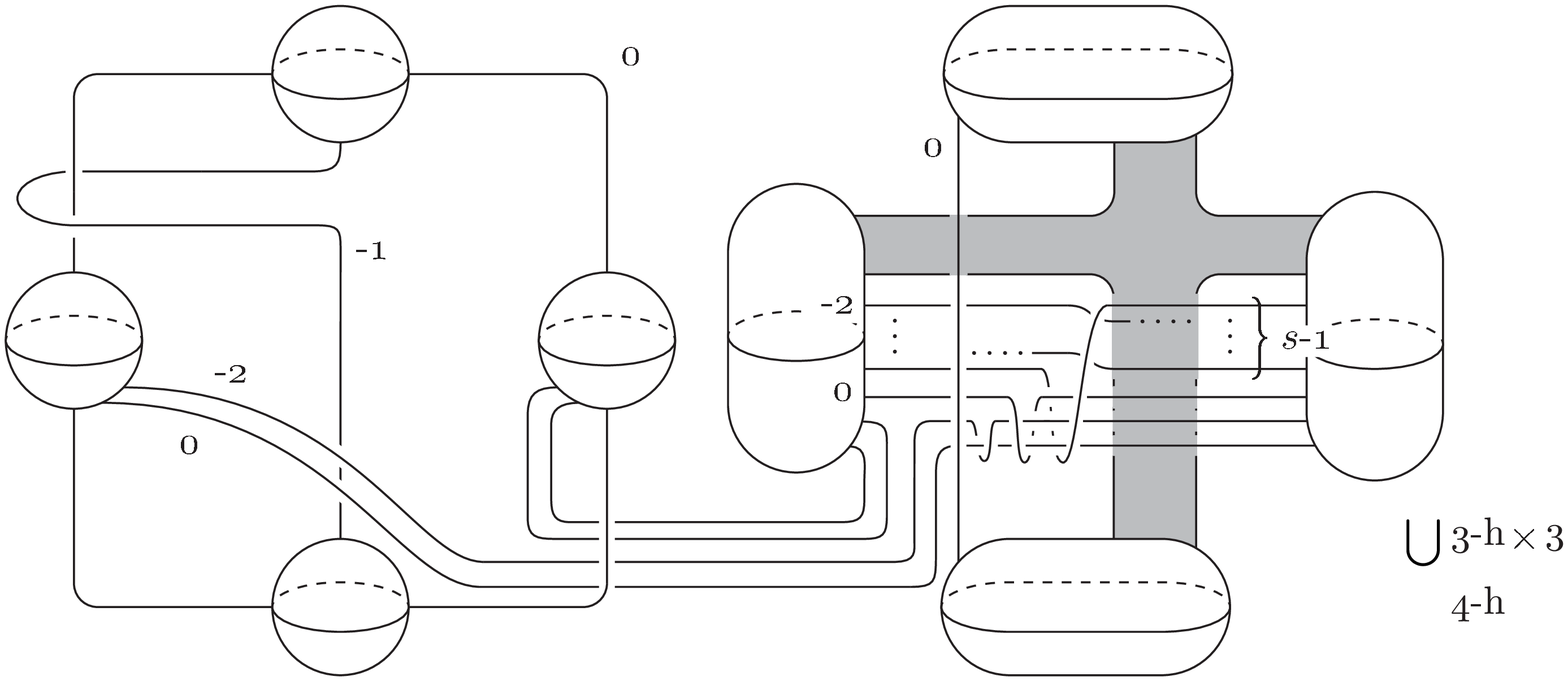}
\end{center}
\caption{}
\label{kirby4}
\end{figure}

\begin{figure}[htbp]
\begin{center}
\includegraphics[width=141mm]{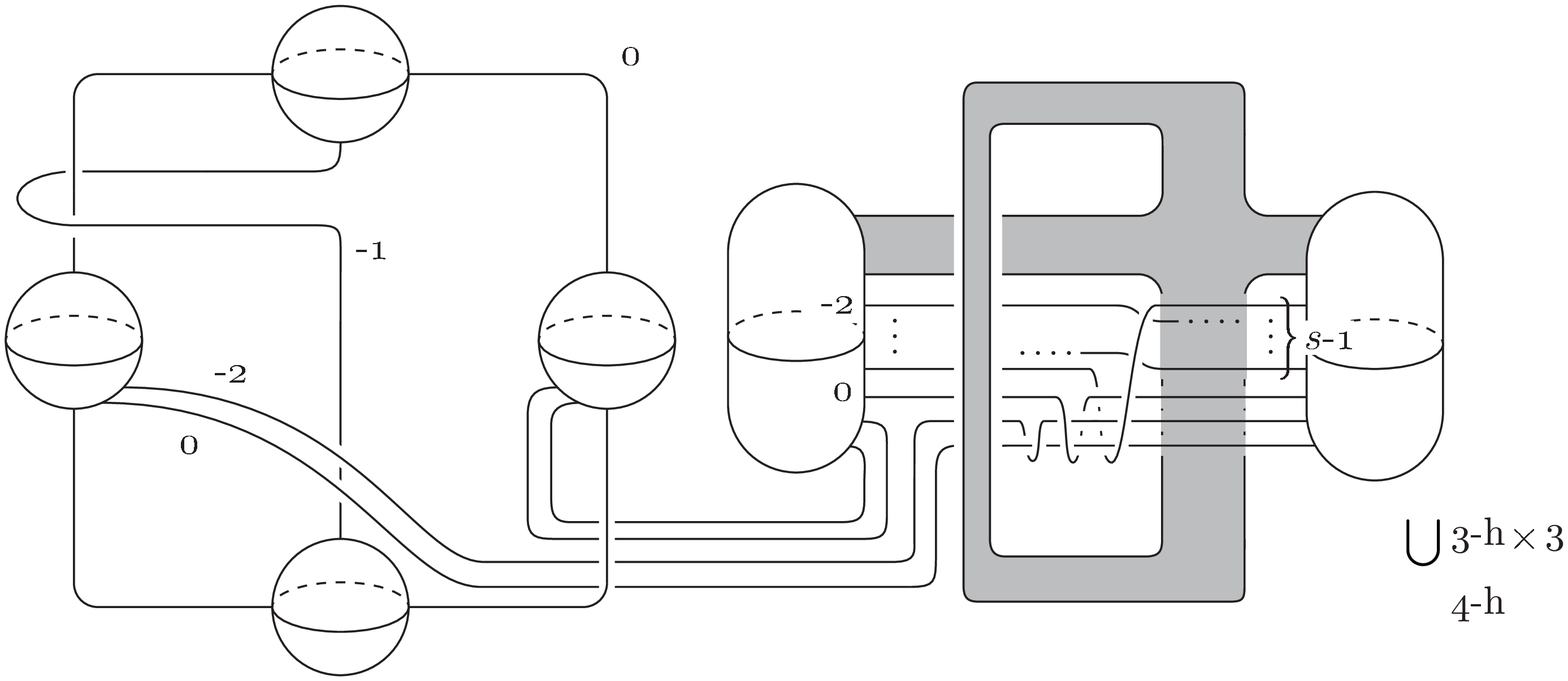}
\end{center}
\caption{}
\label{kirby5}
\end{figure}

\begin{figure}[htbp]
\begin{center}
\includegraphics[width=141mm]{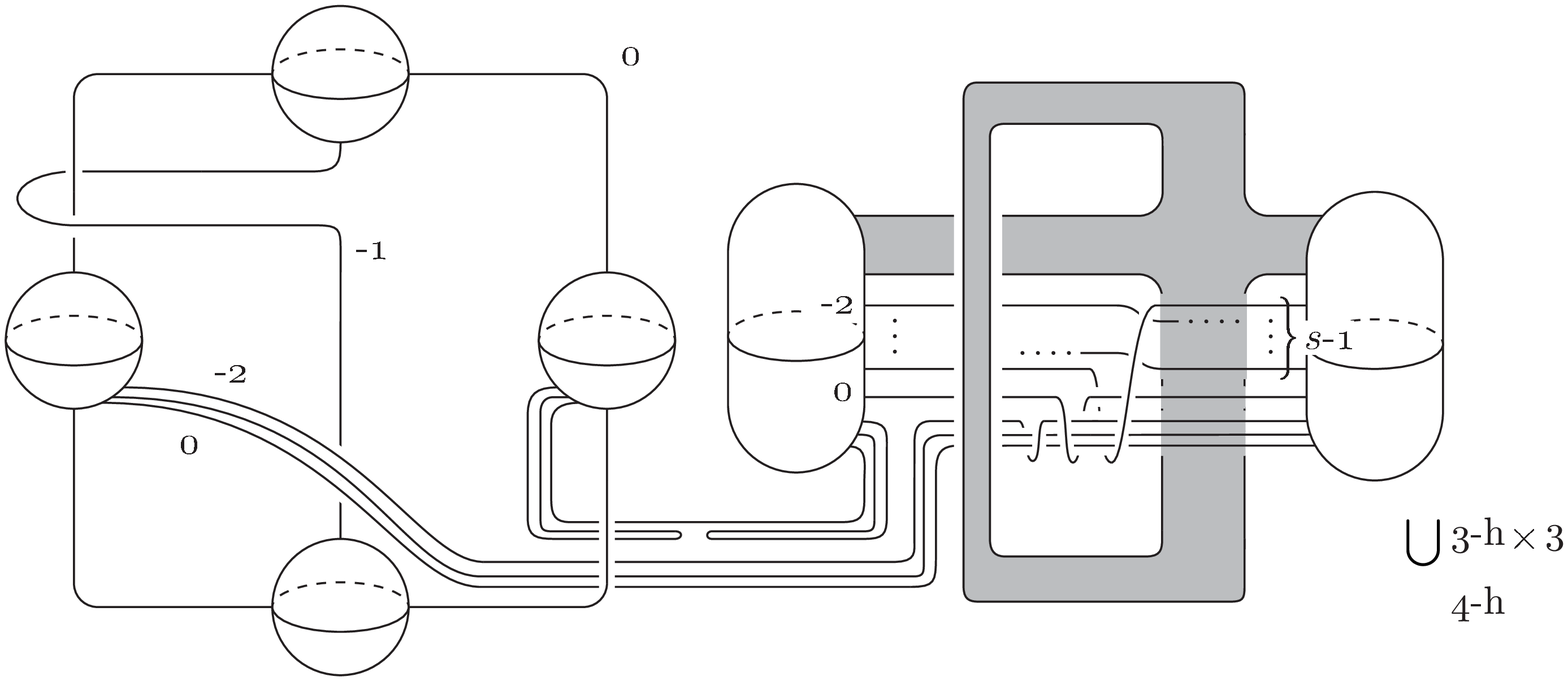}
\end{center}
\caption{}
\label{kirby6}
\end{figure}

\begin{figure}[htbp]
\begin{center}
\includegraphics[width=141mm]{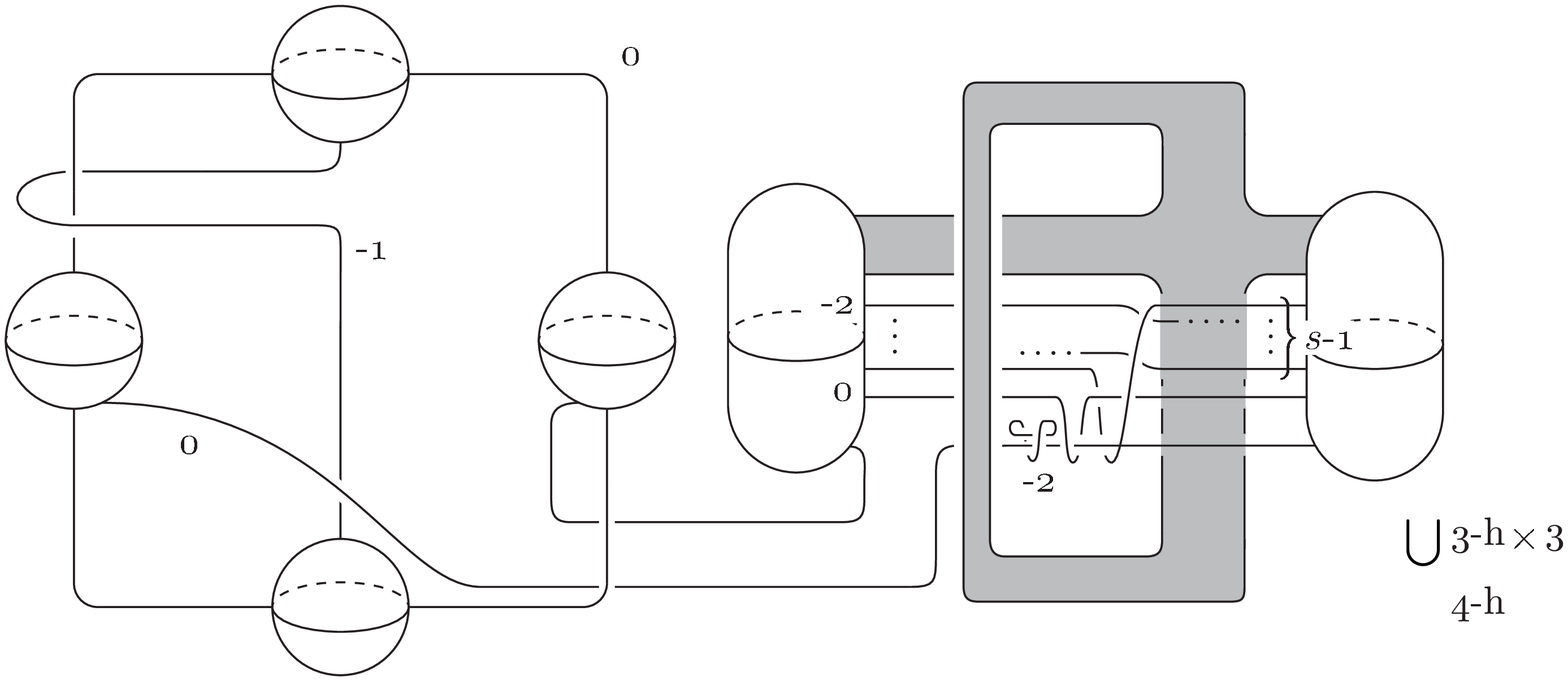}
\end{center}
\caption{}
\label{kirby7}
\end{figure}

\begin{figure}[htbp]
\begin{center}
\includegraphics[width=141mm]{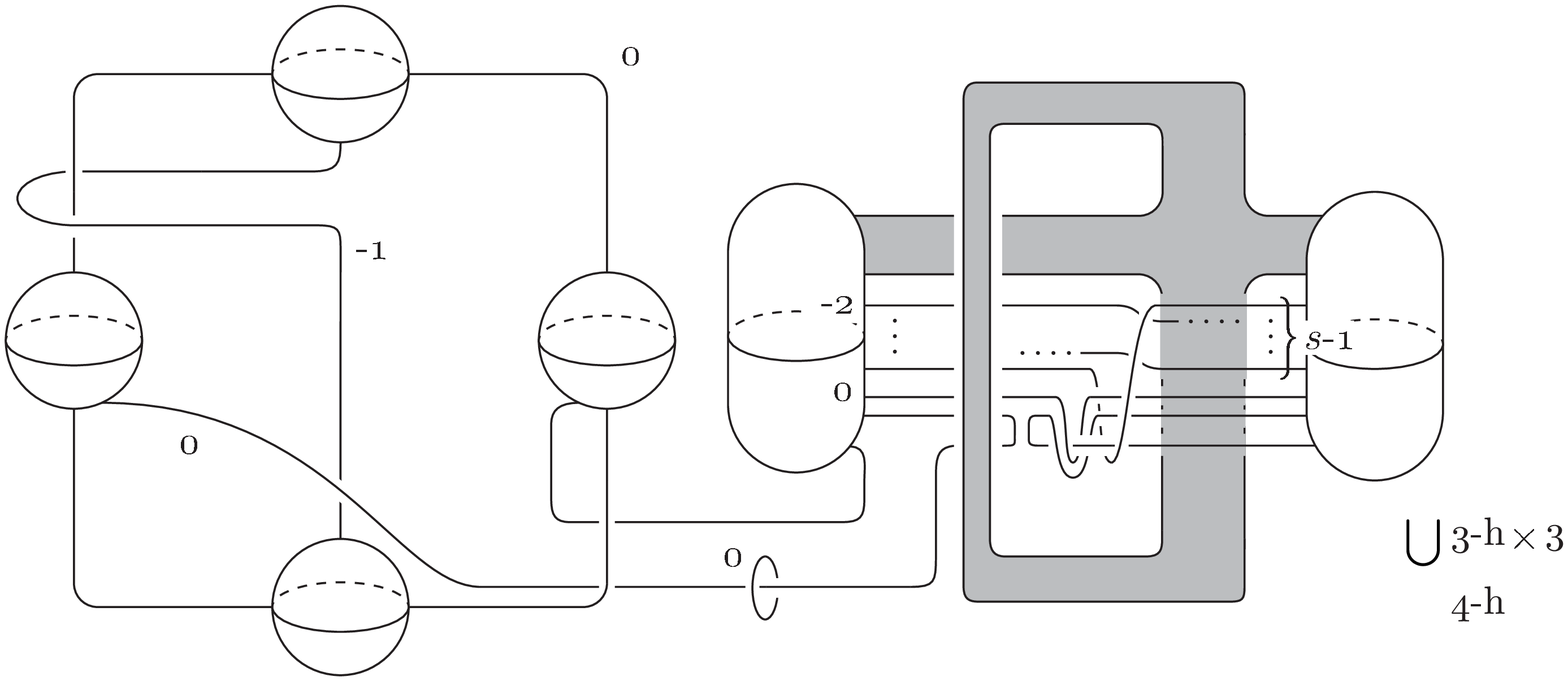}
\end{center}
\caption{}
\label{kirby8}
\end{figure}

\begin{figure}[htbp]
\begin{center}
\includegraphics[width=141mm]{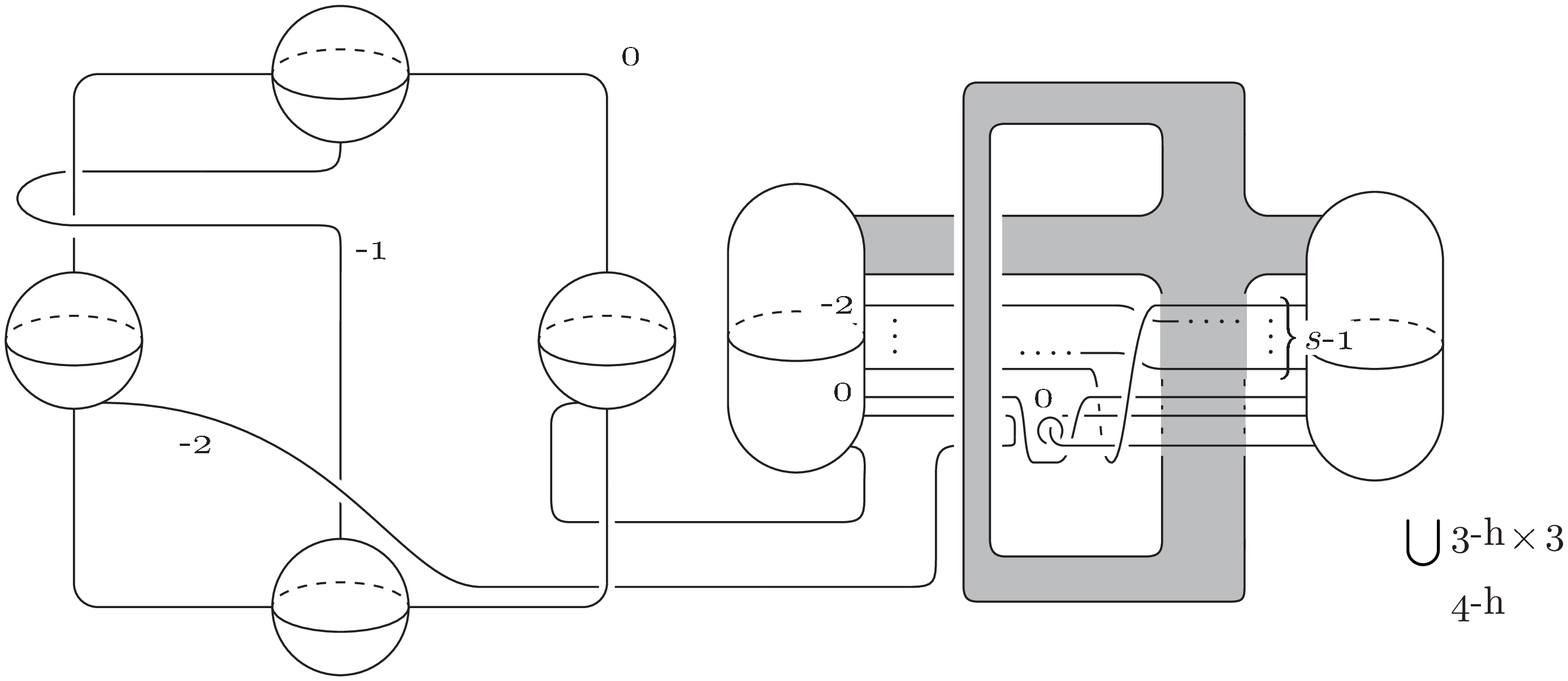}
\end{center}
\caption{}
\label{kirby9}
\end{figure}

\begin{figure}[htbp]
\begin{center}
\includegraphics[width=141mm]{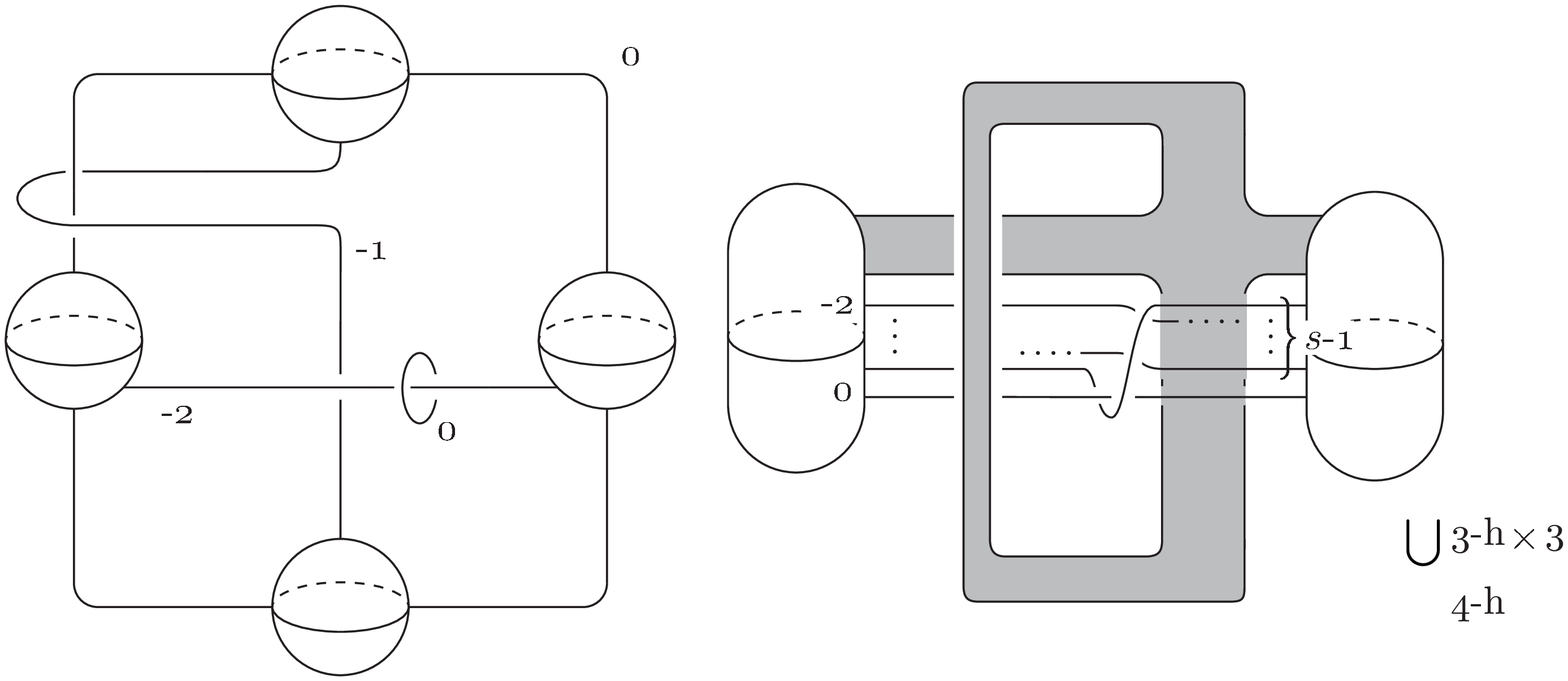}
\end{center}
\caption{}
\label{kirby10}
\end{figure}

\begin{figure}[htbp]
\begin{minipage}{70mm}
\begin{center}
\includegraphics[width=60mm]{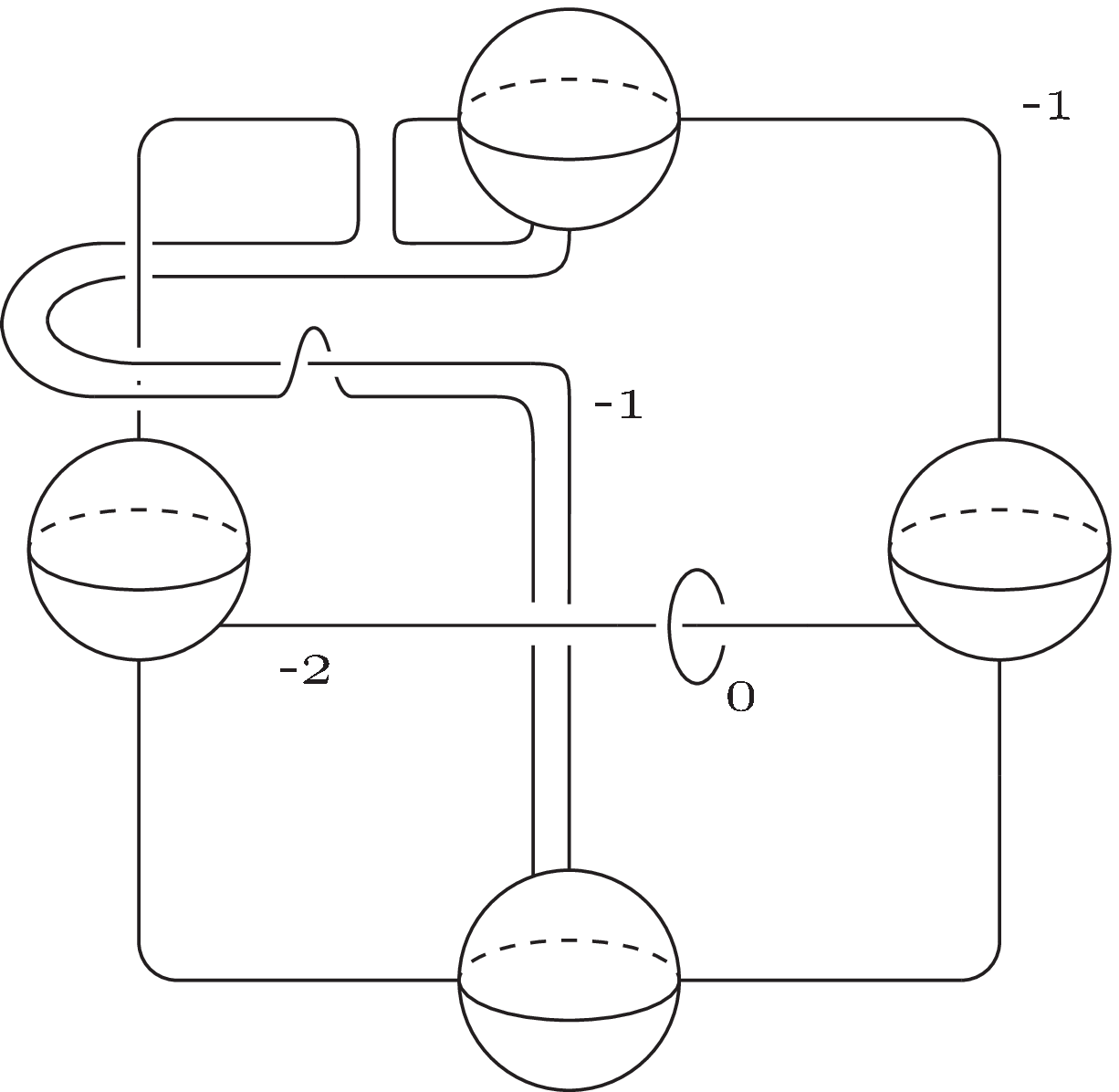}
\end{center}
\caption{}
\label{kirby11}
\end{minipage}
\hspace{.5em}
\begin{minipage}{70mm}
\begin{center}
\includegraphics[width=60mm]{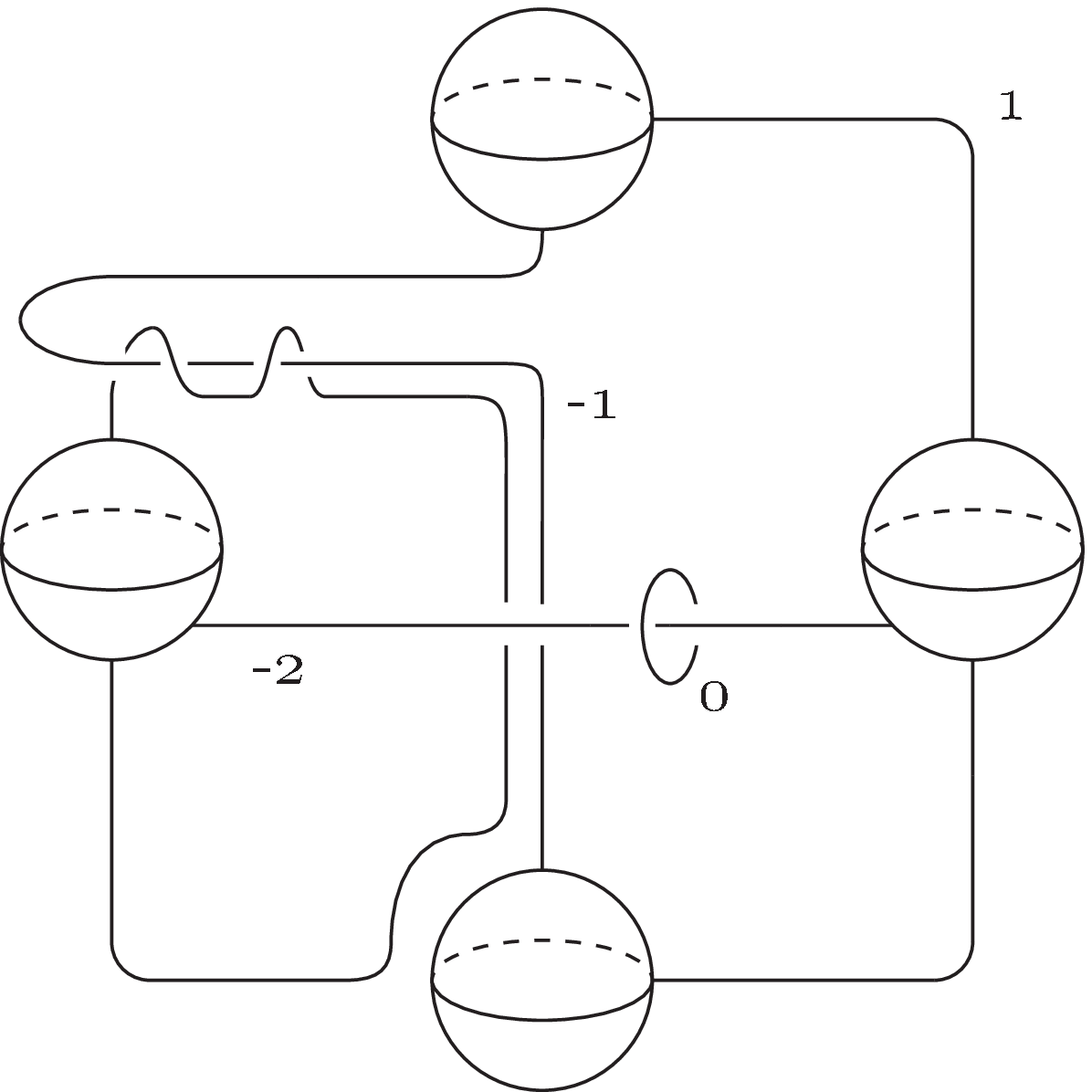}
\end{center}
\caption{}
\label{kirby12}
\end{minipage}
\end{figure}

\begin{figure}[htbp]
\begin{minipage}{70mm}
\begin{center}
\includegraphics[width=60mm]{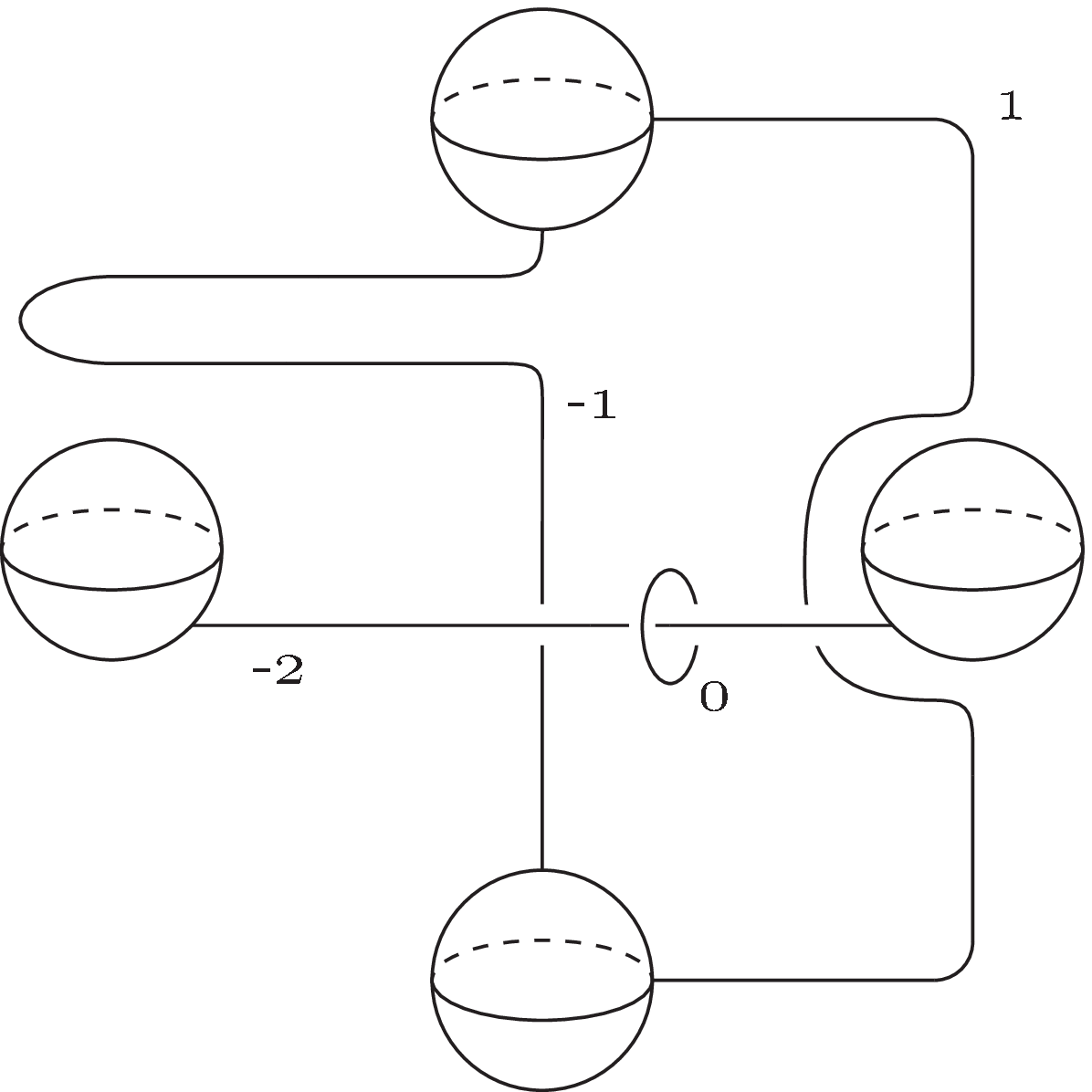}
\end{center}
\caption{}
\label{kirby13}
\end{minipage}
\hspace{1em}
\begin{minipage}{75mm}
\begin{center}
\includegraphics[width=85mm]{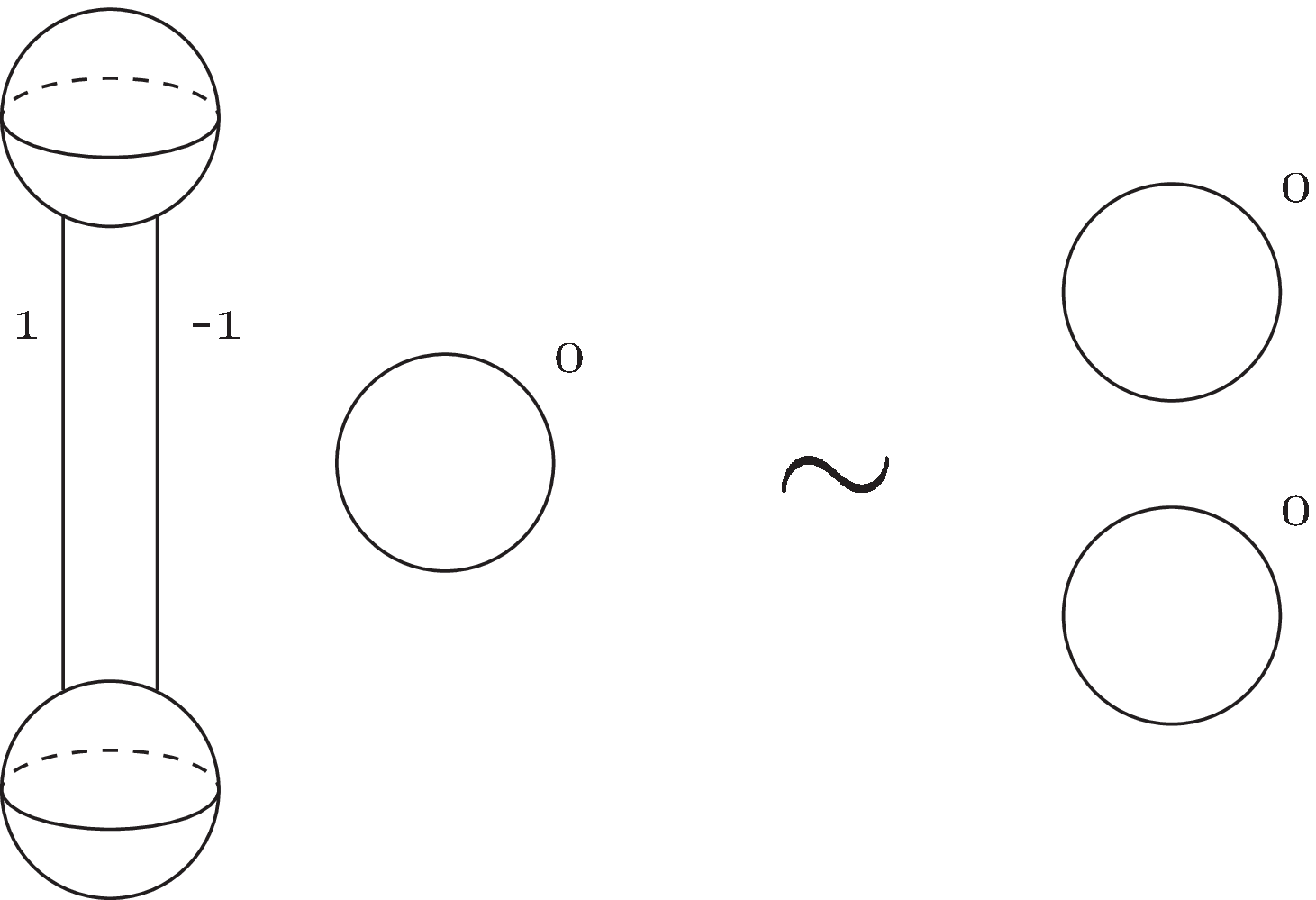}
\end{center}
\caption{}
\label{kirby14}
\end{minipage}
\end{figure}

\begin{figure}[htbp]
\begin{center}
\includegraphics[width=75mm]{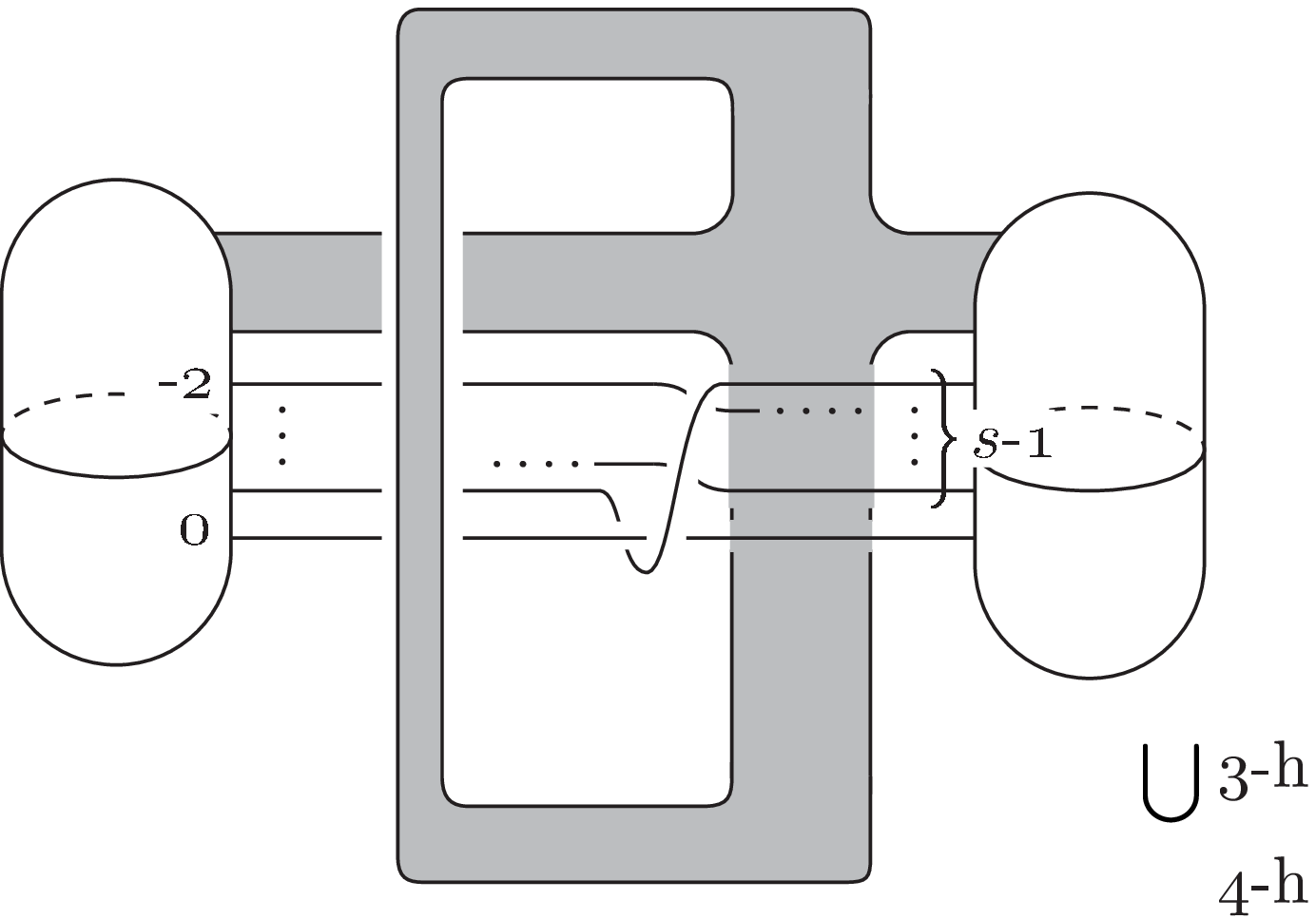}
\end{center}
\caption{}
\label{kirby15}
\end{figure}

\begin{figure}[htbp]
\begin{center}
\includegraphics[width=145mm]{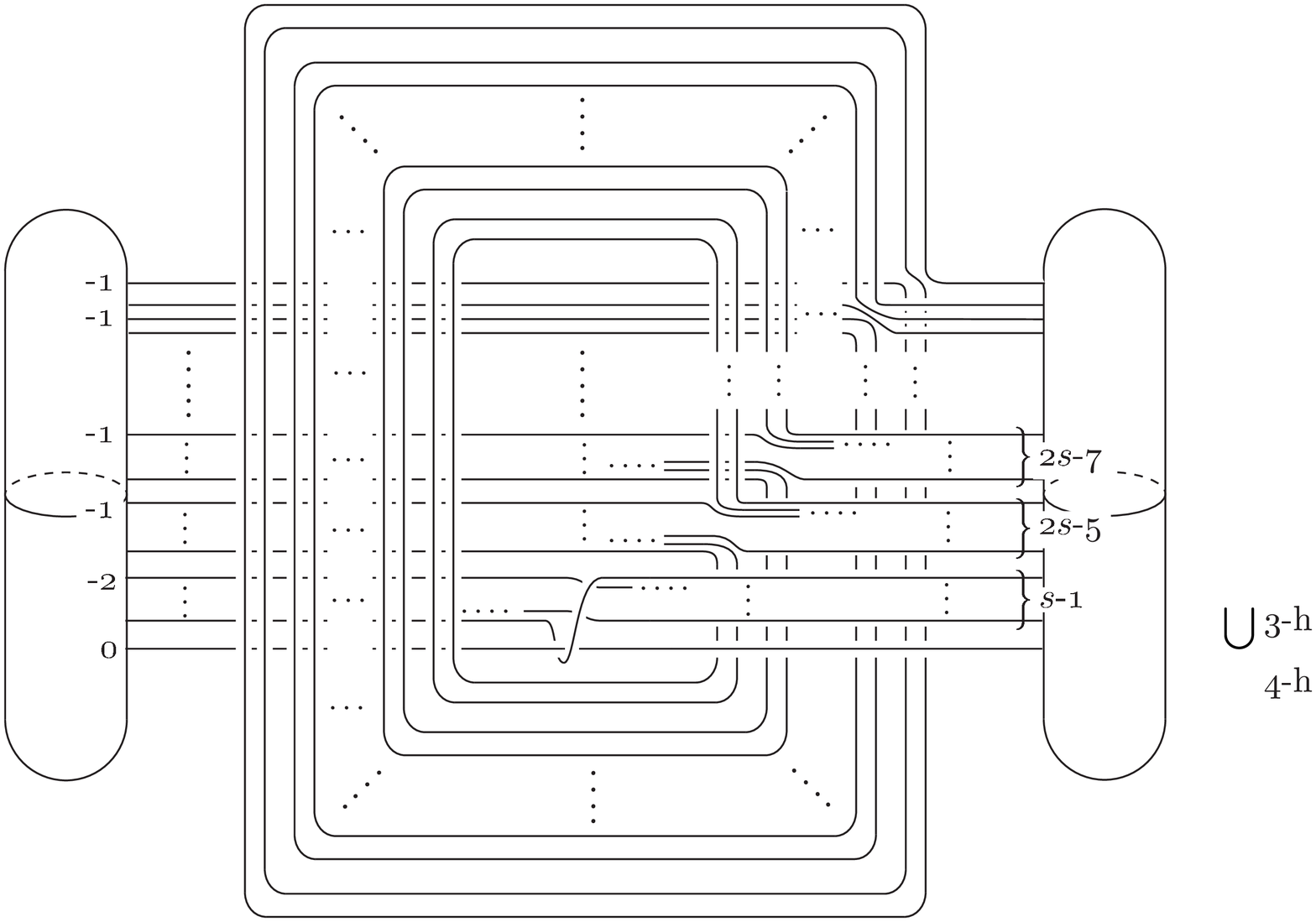}
\end{center}
\caption{}
\label{kirby16}
\end{figure}

\begin{figure}[htbp]
\begin{center}
\includegraphics[width=145mm]{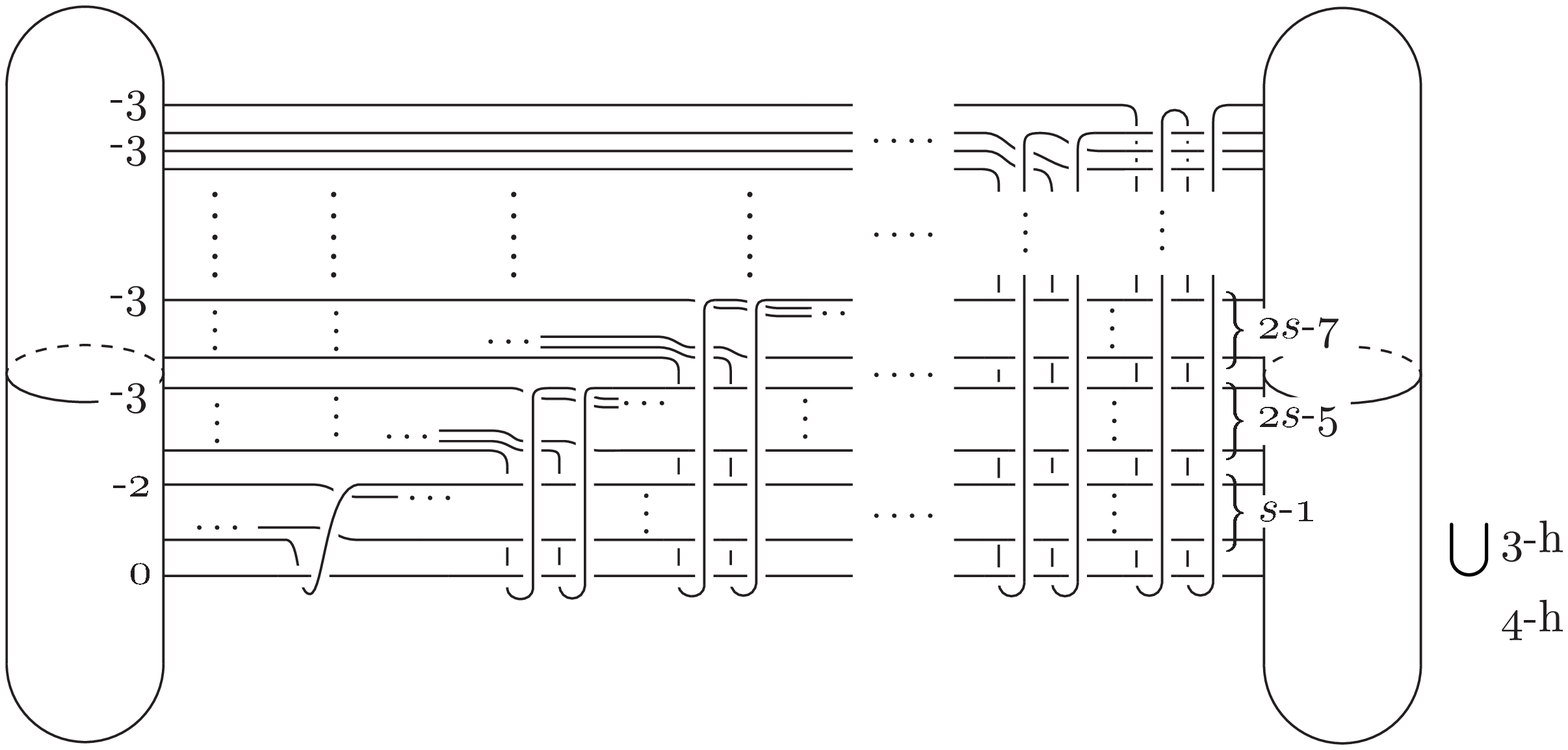}
\end{center}
\caption{}
\label{kirby17}
\end{figure}

\end{document}